\documentclass{amsart}

\usepackage{amsmath, amssymb, amsthm, mathdots}
\usepackage{bbm, lmodern, mathrsfs, euscript}
\usepackage{tikz, tikz-cd}
\usepackage{xcolor}
\usepackage{arydshln}
\usepackage{enumitem}
\usepackage[colorlinks=true, pdfstartview=FitV, linkcolor=blue, citecolor=blue, urlcolor=blue, breaklinks=true]{hyperref}
\usepackage{adjustbox}
\usepackage{graphicx} 
\usepackage{cleveref}

\usepackage{forloop}
\newcounter{fonts}
\let\eeee\edef
\forloop{fonts}{1}{\value{fonts} < 27}{%
\expandafter\eeee\csname \Alph{fonts}\Alph{fonts}\endcsname{\noexpand\mathbb{\Alph{fonts}}}
\expandafter\eeee\csname b\Alph{fonts}\endcsname{\noexpand\mathbf{\Alph{fonts}}}
\expandafter\eeee\csname c\Alph{fonts}\endcsname{\noexpand\mathcal{\Alph{fonts}}}
\expandafter\eeee\csname f\Alph{fonts}\endcsname{\noexpand\mathfrak{\Alph{fonts}}}
\expandafter\eeee\csname fr\alph{fonts}\endcsname{\noexpand\mathfrak{\alph{fonts}}}
\expandafter\eeee\csname s\Alph{fonts}\endcsname{\noexpand\mathscr{\Alph{fonts}}}
\expandafter\eeee\csname e\Alph{fonts}\endcsname{\noexpand\EuScript{\Alph{fonts}}}
\expandafter\eeee\csname t\Alph{fonts}\endcsname{\noexpand\widetilde{\Alph{fonts}}}
\expandafter\eeee\csname t\alph{fonts}\endcsname{\noexpand\tilde{\alph{fonts}}}
}
\renewcommand{\to}{\rightarrow}

\setlist[enumerate,1]{label={\rm(\roman*)}}

\newtheorem{thm}{Theorem}[section]
\newtheorem{defn}[thm]{Definition}
\newtheorem{prop}[thm]{Proposition}
\newtheorem{cor}[thm]{Corollary}
\newtheorem{lem}[thm]{Lemma}

\crefname{defn}{definition}{definitions}
\crefname{lem}{lemma}{lemmas}

\newtheoremstyle{remarkstyle}{.5\baselineskip\@plus.2\baselineskip\@minus.2\baselineskip}{.5\baselineskip\@plus.2\baselineskip\@minus.2\baselineskip}{\rm}{}{\bfseries}{.}{.5em}{}
\theoremstyle{remarkstyle}
\newtheorem{rem}[thm]{Remark}

\newtheorem{example}[thm]{Example}

\newcommand{\GL}{\mathbf{GL}}

\newcommand{\PGL}{\mathbf{PGL}}

\newcommand{\PGO}{\mathbf{PGO}}

\newcommand{\Lie}{\mathfrak{Lie}}

\newcommand{\fpgl}{\mathfrak{pgl}}

\newcommand{\fpgo}{\mathfrak{pgo}}

\newcommand{\Sch}{\mathsf{Sch}}

\newcommand{\cHom}{\mathcal{H}\hspace{-0.4ex}\textit{o\hspace{-0.2ex}m}}
\newcommand{\cEnd}{\mathcal{E}\hspace{-0.4ex}\textit{n\hspace{-0.2ex}d}}
\newcommand{\bAut}{\mathbf{Aut}}
\newcommand{\cSym}{\mathcal{S}\hspace{-0.5ex}\textit{y\hspace{-0.3ex}m}}
\newcommand{\cSymd}{\mathcal{S}\hspace{-0.5ex}\textit{y\hspace{-0.3ex}m\hspace{-0.2ex}d}}
\newcommand{\cSkew}{\mathcal{S}\hspace{-0.4ex}\textit{k\hspace{-0.2ex}e\hspace{-0.3ex}w}}
\newcommand{\cAlt}{\mathcal{A}\hspace{-0.1ex}\textit{$\ell$\hspace{-0.3ex}t}}

\newcommand{\cDer}{\mathcal{D}\hspace{-0.35ex}\textit{e\hspace{-0.15ex}r}}

\DeclareMathOperator{\Hom}{Hom}
\DeclareMathOperator{\Spec}{Spec}
\DeclareMathOperator{\Mat}{M}
\DeclareMathOperator{\Id}{Id}
\DeclareMathOperator{\Img}{Img}
\DeclareMathOperator{\Ker}{Ker}
\DeclareMathOperator{\Trd}{Trd}
\DeclareMathOperator{\Der}{Der}
\DeclareMathOperator{\Inn}{Inn}
\DeclareMathOperator{\ad}{ad}
\DeclareMathOperator{\Aut}{Aut}
\DeclareMathOperator{\rH}{\textrm{H}} 

\newcommand{\iso}{\overset{\sim}{\longrightarrow}}
\newcommand{\inj}{\hookrightarrow}
\newcommand{\surj}{\twoheadrightarrow}

\newcommand{\und}{\underline{\hspace{2ex}}}

\DeclareMathOperator{\Ext}{Ext}
\DeclareMathOperator{\Def}{Def}
\DeclareMathOperator{\res}{res}
\newcommand{\pr}{\textrm{pr}}

\newcommand{\dX}{{\tX}}
\newcommand{\dG}{{\widetilde{\bG}}}
\newcommand{\dF}{{\widetilde{\cF}}}
\newcommand{\dA}{{\widetilde{\cA}}}

\newcommand{\Az}{\textrm{Az}}
\newcommand{\Art}{\mathsf{Art}}

\makeatletter
\newcommand*{\relrelbarsep}{.386ex}
\newcommand*{\relrelbar}{%
  \mathrel{%
    \mathpalette\@relrelbar\relrelbarsep
  }%
}
\newcommand*{\@relrelbar}[2]{%
  \raise#2\hbox to 0pt{$\m@th#1\relbar$\hss}%
  \lower#2\hbox{$\m@th#1\relbar$}%
}
\providecommand*{\rightrightarrowsfill@}{%
  \arrowfill@\relrelbar\relrelbar\rightrightarrows
}
\providecommand*{\xrightrightarrows}[2][]{%
  \ext@arrow 0359\rightrightarrowsfill@{#1}{#2}%
}
\makeatother

\tikzset{
  trim node/.default=1cm,
  trim node/.style={
    overlay,
    append after command={
      ([xshift={+#1}]\tikzlastnode.north west)
      ([xshift={+-#1}]\tikzlastnode.south east)}},
  down and trim/.default=1cm,
  down and trim/.style={
    yshift=-(\pgfmatrixcurrentcolumn-1)*1.5\baselineskip,
    trim node={#1}},
  downup and trim/.default=1cm,
  downup and trim/.style={
    yshift=iseven(\pgfmatrixcurrentcolumn) ? -1.5\baselineskip : 0pt,
    trim node={#1}},
  -|/.style={to path={-|(\tikztotarget)\tikztonodes}},
  |-/.style={to path={|-(\tikztotarget)\tikztonodes}},
  -| sl/.style={-|, xslant=-1},
  |- sl/.style={|-, xslant= 1},
  center picture/.style={
    trim left=(current bounding box.center),
    trim right=(current bounding box.center)}}

\tikzset{
    labl/.style={anchor=south, rotate=-90, inner sep=.5mm}
}

\makeatletter
\ifcsname phantomsection\endcsname
    \newcommand*{\qrr@gobblenexttocentry}[5]{}
\else
    \newcommand*{\qrr@gobblenexttocentry}[4]{}
\fi
\newcommand*{\addsubsection}{%
    \addtocontents{toc}{\protect\qrr@gobblenexttocentry}%
    \subsection*}
\makeatother

\title{On deformations of Azumaya algebras with quadratic pair}
\author{Eoin Mackall}
\address{Department of Mathematics, University of California Santa Cruz, Santa Cruz, CA USA}
\email{emackall@ucsc.edu}

\author{Cameron Ruether}
\address{The ``Simion Stoilow" Institute of Mathematics of the Romanian Academy, Bucharest, Romania}
\email{cameronruether@gmail.com}


\thanks{The second author was supported by the project ``Group schemes, root systems, and related representations" founded by the European Union - NextGenerationEU through Romania's National Recovery and Resilience Plan (PNRR) call no. PNRR-III-C9-2023-I8, Project CF159/31.07.2023, and coordinated by the Ministry of Research, Innovation and Digitalization (MCID) of Romania.}

\date{April 8, 2025}

\begin{document}
\begin{abstract}
We construct a tangent-obstruction theory for Azumaya algebras equipped with a quadratic pair. Under the assumption that either $2$ is a global unit or the algebra is of degree $2$, we show how the deformation theory of these objects reduces to the deformation theory of the underlying Azumaya algebra. Namely, if the underlying Azumaya algebra has unobstructed deformations then so does the quadratic pair.

On the other hand, in the purely characteristic $2$ setting, we construct an Azumaya algebra with unobstructed deformations which can be equipped with a quadratic pair such that the associated triple has obstructed deformations. Our example is a biquaternion Azumaya algebra on an Igusa surface.

Independently from the above results, we also introduce a new obstruction for quadratic pairs, existing only in characteristic 2, which is intermediate to both the strong and weak obstructions that were recently introduced by Gille, Neher, and the second named author. This intermediate obstruction characterizes when a canonical extension of the Lie algebra sheaf of the automorphism group scheme of some quadratic triple is split.
\end{abstract}

\maketitle

\tableofcontents

\section*{Introduction}
{%
\renewcommand{\thethm}{\Alph{thm}}
\addsubsection{Background} Constructing moduli spaces and understanding their geometry has been a driving force behind the development of much of modern algebraic geometry. Using Grothendieck's functorial perspective, the deformation theory of objects that are parametrized by a moduli space is a natural subject arising from the study of the local smoothness of that moduli space. A moduli space fails to be smooth at a point when the object which that point parameterizes fails to deform in some way.

A landmark example of such phenomena comes from an Igusa surface \cite{MR74085}. An Igusa surface is a scheme, existing over a field $k$ of characteristic $p>0$, which can be realized as the quotient of a product of two elliptic curves, with at least one being ordinary, by a free action of $\mathbb{Z}/p\mathbb{Z}$. While it's known that the Picard scheme of any variety in characteristic $0$ is smooth, the Picard scheme of an Igusa surface is everywhere nonreduced. As a consequence, given any line bundle on such a surface, we can find a lifting of that line bundle to a thickening of the Igusa surface which has obstructed deformations to some further infinitesimal thickening.

Even when an object is not obviously parametrized by a moduli space, studying the deformation theory of that object can often shed light on how the object varies in families. For a modern example making use of such analysis, we point to \cite{Zdanowicz} where Zdanowicz uses the deformation theory of equivariant bundles to show that certain Fano varieties in characteristic $p>0$ which violate Kodaira vanishing do not lift to any ring in which $p\neq 0$, i.e., any ring which does not contain $\ZZ/p\ZZ$ as a subring. For an example in the theme of this paper, we point to de Jong's proof of the period index conjecture for the function fields of surfaces \cite{MR2060023} which relies crucially on the deformation theory of Azumaya algebras.

The deformation theory of a $\mathbf{G}$-torsor, for a group scheme $\mathbf{G}$ over a base $S$, was worked out in great generality by Illusie, see \cite[Theorem 2.6]{Illusie}. If one restricts themselves to group schemes $\bG$ which are smooth over $S$, then the theory reduces to one of the underlying sheaf of Lie algebras of $\bG$. In particular, for an infinitesimal thickening of schemes $X\subset X'$ defined by a square-zero ideal $\cJ\subset \cO_{X'}$, a $\bG$--torsor $\mathscr{H}$ on $X$ has an obstruction class in $\rH^2(X,\Lie_{\bG}\otimes_{\cO_X} \cJ)$ which vanishes if and only if a deformation $\mathscr{H}'$ of $\mathscr{H}$ exists on $X'$. When such a deformation exists, one further knows that all such deformations form an affine space under $\rH^1(X,\Lie_{\bG}\otimes_{\cO_X} \cJ)$.

An important question is then: for a given $\mathbf{G}$, under what conditions does one have the vanishing, or nonvanishing, of these cohomological obstruction classes? Demonstrations of nonvanishing obstructions are typically rarer and, therefore, of more interest. De Clerq, Florence, and Lucchini Arteche have anlayzed this problem for deformations of vector bundles from a scheme $X$ to schemes $\mathbf{W}_r(X)$ of Witt vectors of $X$, \cite{De-Clercq_Florence_L-A}. They show that, for integers $n,m$ with $2\leq m\leq n-2$, the tautological vector bundle on a Grassmannian $\mathbf{Gr}(m,n)$ over a finite field $\mathbb{F}_p$ has obstructed deformations to $\mathbf{W}_2(\mathbf{Gr}(m,n))$. This work should be compared with that of Zdanowicz \cite[Section 7.3]{Zdanowicz} who uses a similar result to construct torsors with obstructed deformations under semisimple linear algebraic group schemes $\mathbf{G}$ of classical type.

The standard passage from finite rank vector bundles to Azumaya algebras sends a vector bundle to its endomorphism algebra sheaf. If the endomorphism algebra sheaf of a vector bundle has obstructed deformations then the vector bundle is necessarily also obstructed. However, it is easy to see that the property of having obstructed deformations is not always preserved in the other direction. One can take any line bundle $\cL$ on an Igusa surface $X$ which has obstructed deformations relative to some thickening of $X$. Then, the endomorphism algebra $\cEnd_{\cO_X}(\cL) \cong \cO_X$ clearly admits deformations. In examples such as this, we say that $\cL$ is \emph{relatively obstructed} compared to its endomorphism algebra $\cEnd_{\cO_X}(\cL)$. Inversely, we would say that line bundles are \emph{relatively unobstructed} compared to their endomorphism algebras for a given infinitesimal thickening if, whenever $\cEnd_{\cO_X}(\cL)$ deforms then so must $\cL$. We will use this terminology, relatively obstructed or relatively unobstructed, to describe how various functors interact with obstructions to deformations.

\addsubsection{Deformations of Azumaya algebras with quadratic pair} In this paper we investigate whether Azumaya algebras with quadratic pair are relatively obstructed compared to their underlying Azumaya algebra. Quadratic pairs on central simple algebras were introduced in \cite{KMRT} as a tool that allows one to solve problems related to quadratic forms and linear algebraic groups of type $D$ in characteristic $2$. Namely, quadratic pairs provide a way to describe, in any characteristic, any group of type $D_n$, excluding $D_4$ and trialty, as the automorphism group of a central simple algebra with quadratic pair. This construction was recently extended to Azumaya algebras in \cite{CF2015GroupesClassiques} and studied further in \cite{GNR2023AzumayaObstructions}.

In more detail, a quadratic pair $(\sigma,f)$ on an Azumaya algebra $\cA$ on an arbitrary scheme $X$ consists of an $\cO_X$--linear orthogonal involution $\sigma \colon \cA \to \cA$ and a linear map $f\colon \cSym_{\cA,\sigma} \to \cO_X$ on the subsheaf $\cSym_{\cA, \sigma}\subset \cA$ of sections invariant under $\sigma$, which satisfies certain properties (see \Cref{quadratic_triples} for more). The map $f$ is called a semi-trace and we often call such a tuple $(\cA,\sigma,f)$ a quadratic triple.

Quadratic pairs on degree $n$ Azumaya algebras are in correspondence with $\PGO_n$--torsors and, therefore, their deformation theory is covered by the results from \cite{Illusie}. Using this theory, we address the following relative deformation question: is there an Azumaya algebra $\mathcal{A}$ with quadratic pair $(\sigma,f)$ such that the quadratic triple $(\mathcal{A},\sigma,f)$ has obstructed deformations but such that the deformations of $\cA$, simply considered as an Azumaya algebra, are unobstructed?

We work in the following framework. Let $(C,\frm_C)$ and $(C',\frm_{C'})$ be two local Artinian $k$--algebras whose residue fields are also $k$. Let $C' \surj C$ be a surjection given by an ideal $J \subseteq \frm_{C'}$ such that $\frm_{C'}\cdot J = 0$; note that this implies $J^2=0$. Such a surjection is called a small extension. We consider a $k$-scheme $X$ and the associated schemes $X_C = X\times_k C$ and $X_{C'} = X \times_k C'$. The map $X_C \inj X_{C'}$ associated to $C' \surj C$ is then a square-zero (first order) thickening of schemes.

Our first result regarding the relative obstructions of quadratic pairs is a positive statement, showing that vanishing of obstructions for the quadratic pair follows from vanishing of obstructions of the underlying Azumaya algebra in two notable special cases (our \Cref{relatively_unobstructed}). 

\begin{thm}\label{intro_i}
Let $X_C \to X_{C'}$ be a square-zero thickening of schemes associated to a small extension of local Artinian $k$--algebras as above. Let $(\cA,\sigma,f)$ be an Azumaya algebra with quadratic pair on $X_C$. 

Then, there exists a deformation $(\cA',\sigma',f')$ on $X_{C'}$ of the quadratic triple $(\cA,\sigma,f)$ if and only if there exists a deformation $\cA''$ of the underlying Azumaya algebra $\cA$ under either of the following two hypothesis:
\begin{enumerate}
\item\label{intro_hyp_i} assuming either the characteristic of $k$ is not $2$, i.e.\ $\mathrm{char}(k)\neq 2$;
\item\label{intro_hyp_ii} or assuming $\mathrm{deg}(\mathcal{A})=2$.
\end{enumerate}
\end{thm}

Note that, in the above statement, the algebra $\cA'$ appearing in the deformation $(\cA',\sigma',f')$ is not necessarily the same as the deformation $\cA''$ of just the algebra $\cA$.

In the opposite extreme, we construct an example satisfying neither of the above hypothesis \ref{intro_hyp_i} or \ref{intro_hyp_ii} where such relative obstructions really do exist. Our example makes use of the existing obstructions to deforming line bundles on an Igusa surface. Specifically, let $X$ be an Igusa surface defined over an algebraically closed field $k$ of characteristic $2$. Let $\mu_2 = \Spec(k[x]/( x^2 ))$, $\tau_3 = \Spec(k[x]/( x^3 ))$, with the canonical closed embedding $\mu_2 \to \tau_3$. We then have morphisms
\[
\begin{tikzcd}
X_{\mu_2} \ar[dr,"\pi"] \ar[rr,hookrightarrow] & & X_{\tau_3} \ar[dl] \\
 & X & 
\end{tikzcd}
\]
where the top horizontal map is a square-zero infinitesimal thickening and the slanted arrows are the canonical projection maps. We show (in \Cref{thm: main} below) the following:
 
\begin{thm}\label{intro_ii}
There exists an Igusa surface $X$ (gotten as the quotient of any two ordinary elliptic curves) and a quadratic triple $(\cA,\sigma,f)$ on $X_{\mu_2}$ such that
\begin{enumerate}
\item the quadratic triple $(\cA,\sigma,f)$ does not deform as a quadratic triple to $X_{\tau_3}$;
\item the Azumaya algebra $\cA$ on $X_{\mu_2}$ does deform to $X_{\tau_3}$ (the algebra $\cA$ is even unobstructed).
\end{enumerate}
\end{thm}

The quadratic pair appearing in the theorem is a canonically constructed pair on a biquaternion algebra $\mathcal{A}\cong \mathcal{Q}'\otimes \mathcal{P}'$ on $X_{\mu_2}$. The quadratic pair comes from the equivalence of stacks shown in \cite{GNR2024TheNormFunctor} which reflects the exceptional isomorphism between the Dynkin types $A_1\times A_1$ and $D_2$. Because the quadratic pair comes from this equivalence, it will deform if and only if both $\cQ'$ and $\cP'$ deform. However, we construct $\cQ'$ and $\cP'$ in such a way to ensure that at least one of $\cQ'$ or $\cP'$ has obstructed deformations. We can do this by exploiting the fact that the Igusa surface $X$ has line bundles with obstructed deformations.

More precisely, we start with an exact sequence
\[
0 \to \cO_X \to \cV \to \cL \to 0
\]
where $\cL$ is a particular line bundle on $X$ and $\cV$ is a rank $2$ vector bundle such that we have an explicit cocycle description for the quaternion algebra $\cQ = \cEnd_{\cO_X}(\cV)$. The challenge is to find an Azumaya algebra $\mathcal{P}$ on $X$ such that $(\mathcal{Q}\otimes \mathcal{P})(X)=k$. We use this condition critically, along with either \Cref{cor: quartvan} or \Cref{thm: zero} below, to ensure that any nontrivial obstruction we create vanishes once the quadratic pair is forgotten. That is to say, we prove (using that $K_X\cong \mathcal{O}_X)$ that this condition implies that the canonical map on the cohomology of Lie algebra sheaves \[
\mathrm{H}^2(X,\fpgo_{(\cQ\otimes\cP,\sigma,f)}) \to \mathrm{H}^2(X,\fpgl_{\cQ\otimes\cP}),
\]
is identically zero. Although we believe that this condition is most likely generically satisfied, most of the examples that we could produce did not have this condition. The algebra $\mathcal{P}$ with these properties which we use is constructed from a canonical $\mathbf{PGL}_2(\mathbb{F}_2)$-torsor over $X$. Since $\mathbf{GL}_2(\mathbb{F}_2)=\mathbf{PGL}_2(\mathbb{F}_2)$, it is therefore also neutral, so $\mathcal{P}\cong \cEnd(\mathcal{W})$ for some rank $2$ vector bundle $\mathcal{W}$ on $X$ as well.

Next, we argue that at least one of $\mathcal{Q}$ or $\mathcal{P}$ must have a deformation to $X_{\mu_2}$ with obstructed deformations to $X_{\tau_3}$. We do this by showing that simple rank $2$ vector bundles on $X$ are relatively unobstructed compared to their endomorphism algebra sheaves. We then show that collective deformations of the (simple) vector bundles $\mathcal{V}$ and $\mathcal{W}$ underlying $\mathcal{Q}$ and $\mathcal{P}$ have determinant bundles which admit deformations in independent directions. We show that at least one of these directions corresponds to a line bundle with obstructed deformations to $X_{\tau_3}$. All together, this allows us to prove \Cref{intro_ii}.

\addsubsection{An Intermediate Obstruction} Along our way to constructing the example in \Cref{intro_ii}, we develop a new cohomological obstruction related to quadratic triples which only occurs in the purely characteristic $2$ realm and is related to the behaviour of Lie algebra sheaves. In general, on any scheme $X$, given an Azumaya algebra with orthogonal involution $(\cA,\sigma)$ we can consider submodules of $\cA$
\begin{align*}
\cSym_{\cA,\sigma} &= \Ker(\Id_\cA -\sigma) & \cSkew_{\cA,\sigma} &= \Ker(\Id_\cA + \sigma)\\
\cAlt_{\cA,\sigma} &= \Img(\Id_\cA -\sigma) & \cSymd_{\cA,\sigma} &= \Img(\Id_\cA + \sigma)
\end{align*}
called the symmetric, skew-symmetric, alternating, and symmetrized elements, respectively. If $\sigma$ is a locally quadratic involution, meaning that $1\in \cSymd_{\cA,\sigma}(X)$, i.e., the unit of $\cA(X)$ is a symmetrized element, then there exist cohomological obstructions introduced in \cite{GNR2023AzumayaObstructions}
\begin{align*}
\Omega(\cA,\sigma) &\in H^1(X,\cSkew_{\cA,\sigma}) \\
\omega(\cA,\sigma) &\in H^1(X,\cSkew_{\cA,\sigma}/\cAlt_{\cA,\sigma})
\end{align*}
called the strong and weak obstructions, respectively. These are defined as the images of $1\in \cSymd_{\cA,\sigma}(X)$ under the boundary maps of the following long exact cohomology sequences (and so the strong obstruction maps to the weak obstruction under the rightmost vertical arrow).
\[
\begin{tikzcd}[column sep = 2ex]
0 \ar[r] & \cSkew_{\cA,\sigma}(X) \ar[r] \ar[d] & \cA(X) \ar[r] \ar[d] & \cSymd_{\cA,\sigma} \ar[d,equals] \ar[r] & H^1(X,\cSkew_{\cA,\sigma}) \ar[d] \\
0 \ar[r] & (\cSkew_{\cA,\sigma}/\cAlt_{\cA,\sigma})(X) \ar[r] & (\cA/\cAlt_{\cA,\sigma})(X) \ar[r] & \cSymd_{\cA,\sigma} \ar[r] & H^1(X,\cSkew_{\cA,\sigma}/\cAlt_{\cA,\sigma})
\end{tikzcd}
\]
The weak obstruction controls whether $(\cA,\sigma)$ can be equipped with a semi-trace $f\colon \cSym_{\cA,\sigma} \to \cO_X$ and the strong obstruction controls whether there exists a linear map $f' \colon \cA \to \cO_X$ such that its restriction to the symmetric elements is a semi-trace. 

Now, if we specialize to the purely characteristic $2$ case, then
\[
\cSym_{\cA,\sigma} = \cSkew_{\cA,\sigma} \text{ and } \cSymd_{\cA,\sigma} = \cAlt_{\cA,\sigma}.
\]
Furthermore, since we are assuming $1\in \cSymd_{\cA,\sigma}(X) = \cAlt_{\cA,\sigma}(X)$, this means that $\cO_X \subseteq \cAlt_{\cA,\sigma} \subseteq \cSkew_{\cA,\sigma}$. Thus, we may add an intermediate row to the diagram above,
\[
\begin{tikzcd}[column sep = 2ex]
0 \ar[r] & \cSkew_{\cA,\sigma}(X) \ar[r] \ar[d] & \cA(X) \ar[r] \ar[d] & \cSymd_{\cA,\sigma} \ar[d,equals] \ar[r] & H^1(X,\cSkew_{\cA,\sigma}) \ar[d] \\
0 \ar[r] & (\cSkew_{\cA,\sigma}/\cO_X)(X) \ar[r] \ar[d] & (\cA/\cO_X)(X) \ar[d] \ar[r] & \cSymd_{\cA,\sigma} \ar[r] \ar[d,equals] & H^1(X,\cSkew_{\cA,\sigma}/\cO_X) \ar[d] \\
0 \ar[r] & (\cSkew_{\cA,\sigma}/\cAlt_{\cA,\sigma})(X) \ar[r] & (\cA/\cAlt_{\cA,\sigma})(X) \ar[r] & \cSymd_{\cA,\sigma} \ar[r] & H^1(X,\cSkew_{\cA,\sigma}/\cAlt_{\cA,\sigma})
\end{tikzcd}
\]
and we call the image of $1$ in $\mathrm{H}^1(X,\cSkew_{\cA,\sigma}/\cO_X)$ the intermediate obstruction, denoted by $\Omega'(\cA,\sigma)$. We show in \Cref{Lie_extensions} that for any semi-trace $f$ making $(\cA,\sigma,f)$ a quadratic triple, we have an exact sequence
\[
0 \to \cAlt_{\cA,\sigma}/\cO_X \to \fpgo_{(\cA,\sigma,f)} \to \cO_X \to 0
\]
where $\fpgo_{(\cA,\sigma,f)}$ is the Lie algebra sheaf of the automorphism group $\PGO_{(\cA,\sigma,f)}$. The intermediate obstruction, or more precisely its vanishing, controls whether such a sequence exists which is split exact (our \Cref{intermediate_meaning}).
\begin{thm}\label{intro_iii}
Let $X$ be a scheme over a field $k$ of characteristic $2$ and let $(\cA,\sigma)$ be an Azumaya algebra with locally quadratic orthogonal involution. 

Then, the intermediate obstruction $\Omega'(\cA,\sigma)$ is zero if and only if there exists a semi-trace $f$ making $(\cA,\sigma,f)$ a quadratic triple and such that sequence
\[
0 \to \cAlt_{\cA,\sigma}/\cO_X \to \fpgo_{(\cA,\sigma,f)} \to \cO_X \to 0
\]
is split exact.
\end{thm}

While \Cref{intro_iii} is not strictly needed for the proof of \Cref{intro_ii}, we show how such results can naturally be obtained for arbitrary even degree Azumaya algebras by some cohomological arguments related to this sequence (see \Cref{thm: zero}).

\addsubsection{Organization} The contents of this paper are organized as follows. We recall definitions and basic properties of our main objects throughout the preliminaries section. Both Azumaya algebras and quadratic triples are addressed in \Cref{azumaya_algebras}. Their automorphism group schemes, which are linear algebraic groups, as well as the associated Lie algebra sheaves are discussed in \Cref{Lie_Algebras}. The required details about the stack equivalence $A_1\times A_1 \cong D_2$ are covered in \Cref{norm_functor}. The setup for what we consider a deformation problem is explained in \Cref{deformations}, as well as background on canonical cohomological obstruction classes and classification of deformations. These can be neatly combined into a tangent-obstruction theory, discussed in \Cref{tangent_obstruction}. We end the preliminaries by reviewing Igusa surfaces and their properties that are necessary for us in \Cref{Igusa_surfaces}.

Our work on the intermediate obstruction in characteristic $2$ is done in \Cref{intermediate_obstruction}. Throughout this section we also discuss how the extension classes of various short exact sequences relate to the strong, intermediate, and weak obstructions, offering a new lens through which to view the two previously defined obstructions.

Relative obstructions are discussed in \Cref{deformations_of_quadratic_triples}.

Finally, in \Cref{explicit_example} we explicitly construct the example promised in \Cref{intro_ii}. The two quaternion algebras $\cQ$ and $\cP$ on an Igusa surface $X$ come from concrete fppf $1$--cocycles. We use computations with these cocycles to verify that $\cQ\otimes \cP$ has the required properties for our example. To see that the quadratic triple structure $(\cA,\sigma,f)$ on $\mathcal{A}\cong \mathcal{Q}'\otimes \mathcal{P}'$ coming from the equivalence $A_1\times A_1 \cong D_2$ does not deform, we transfer the question back to the $A_1\times A_1$ side where one of $\cQ'$ or $\cP'$ does not deform by construction, therefore obstructing the deformations of $(\cA,\sigma,f)$. All of these details are assembled in \Cref{thm: main}.
}

\section{Preliminaries}
\subsection{Azumaya algebras with additional structure}\label{azumaya_algebras}
Throughout this section we fix a scheme $(X,\mathcal{O}_X)$ to be considered as a base.

\subsubsection{Azumaya algebras} Recall that an $\mathcal{O}_X$-algebra $\mathcal{A}$ is an \textit{Azumaya $\mathcal{O}_X$-algebra} if there exists an fppf-cover $\{X_i \to X\}_{i\in I}$ and finite locally free $\cO_{X_i}$-modules $\cM_i$ of everywhere positive rank such that there are isomorphisms $\cA_{X_i}\cong \cEnd_{\cO_{X_i}}(\cM_i)$ for each $i\in I$. If each of the $\mathcal{O}_{X_i}$-modules $\mathcal{M}_i$ are, moreover, free, then we say that such a cover \textit{splits} the Azumaya $\mathcal{O}_X$-algebra $\mathcal{A}$.

Due to \cite[Theorem 3.9]{MR559531}, for any Azumaya $\mathcal{O}_X$-algebra $\mathcal{A}$ there is an \'etale cover of $X$ that splits $\mathcal{A}$. This fact is useful for setting up a deformation theory for Azumaya $\mathcal{O}_X$-algebras equipped with additional structures. However, we frequently use splittings of Azumaya $\mathcal{O}_X$-algebras along fppf-covers which are not \'etale.

\subsubsection{Involutions on Azumaya algebras} Given an Azumaya $\cO_X$-algebra $\cA$, an involution of the first kind on $\mathcal{A}$ is an $\cO_X$-module morphism $\sigma \colon \cA \to \cA$ satisfying \begin{equation} \sigma^2=\Id_{\cA}\quad \mbox{and} \quad \sigma(ab)=\sigma(b)\sigma(a)\end{equation} for all sections $a,b\in \cA(U)$ for any open $U\subset X$. Since throughout this text we restrict ourselves to involutions of the first kind only (as opposed to involutions of the second kind, which need not be $\cO_X$--linear), we will often implicitly assume that an involution is of the first kind without comment.

Let $(\cA,\sigma)$ be an Azumaya $\mathcal{O}_X$-algebra with involution. There exists a sufficiently fine fppf-cover $\{X_i\to X\}_{i\in I}$ such that $\cA_{X_i} \cong \Mat_{n_i}(\cO_{X_i})$ and the involution $\sigma_{X_i}$ is adjoint to a regular bilinear form $b_i \colon \cO_{X_i}^{n_i} \times \cO_{X_i}^{n_i} \to \cO_{X_i}$. The involution $\sigma$ is called:
\begin{enumerate}
    \item \emph{orthogonal} if all $b_i$ are symmetric, i.e., $b_i(x_1,x_2)=b_i(x_2,x_1)$,
    \item \emph{weakly symplectic} if all $b_i$ are skew-symmetric, i.e., $b_i(x_1,x_2)=-b_i(x_1,x_2)$,
    \item and \emph{symplectic} if all $b_i$ are alternating, i.e., $b_i(x,x)=0$.
\end{enumerate}
These notions are independent of the chosen cover and are thus well-defined globally. For an Azumaya algebra $(\cA, \sigma)$ with involution, we have two $\cO_X$-module morphisms $\Id_\cA \pm \sigma \colon \cA \to \cA$ and we define the following notable submodules of $\cA$,
\begin{align}
\begin{split}
\cSym_{\cA,\sigma} &= \Ker(\Id_\cA -\sigma)\\
\cSkew_{\cA,\sigma} &= \Ker(\Id_\cA + \sigma)\\
\cAlt_{\cA,\sigma} &= \Img(\Id_\cA -\sigma)\\
\cSymd_{\cA,\sigma} &= \Img(\Id_\cA + \sigma)
\end{split}
\end{align}
called the submodules of symmetric, skew-symmetric, alternating, and symmetrized elements respectively. For any open $U\subset X$, we can say that
\begin{align}
\begin{split}
\cSym_{\cA,\sigma}(U) &= \{ a \in \cA(U) \mid \sigma(a)=a\}, \text{ and}\\
\cSkew_{\cA,\sigma}(U) &= \{a \in \cA(U) \mid \sigma(a)=-a\}.
\end{split}
\end{align}
However, $\cAlt_{\cA,\sigma}$ and $\cSymd_{\cA,\sigma}$ are image sheaves and in general do not have such convenient descriptions because their construction involves sheafification.

\subsubsection{Quadratic triples}\label{quadratic_triples}
Orthogonal involutions considered on their own can be poorly behaved if $2$ is not invertible over our base scheme $X$, so it is often advantageous to work with \emph{quadratic pairs}. Taking our definition from \cite[2.7.0.30]{CF2015GroupesClassiques} and our terminology from \cite[3.1]{GNR2023AzumayaObstructions}, for an Azumaya $\cO_X$-algebra $\cA$ we call $(\sigma,f)$ a \emph{quadratic pair on $\cA$} or call $(\cA,\sigma,f)$ a \emph{quadratic triple} if:
\begin{enumerate}
    \item $\sigma$ is an orthogonal involution on $\cA$, and
    \item $f\colon \cSym_{\cA,\sigma} \to \cO_X$ is an $\cO_X$-module morphism satisfying $$f(a+\sigma(a))=\Trd_{\cA}(a)$$ for all sections $a\in \cA(U)$ for all open $U\subset X$.
\end{enumerate}
Here $\Trd_\cA \colon \cA \to \cO_X$ is the \emph{reduced trace} of $\cA$, which is the unique $\cO_X$-module morphism fppf locally isomorphic to the usual trace of matrix algebras. In the case that $2$ is invertible over $X$, the second condition above forces $f=\frac{1}{2}\Trd_\cA$. If this is the case, then we may simply ignore $f$ as the theory of quadratic pairs coincides with the usual theory of orthogonal involutions. By \cite[2.7.0.32]{CF2015GroupesClassiques}, all quadratic triples are \'etale locally isomorphic, i.e., they are all twisted forms of one another.

Given an Azumaya algebra $\mathcal{A}$ with an orthogonal involution $\sigma$ on $\mathcal{A}$, it may not be possible to equip $(\mathcal{A},\sigma)$ with a semi-trace $f$ to form a quadratic triple $(\mathcal{A},\sigma,f)$. By \cite[Lemma 3.14]{GNR2023AzumayaObstructions}, one necessary condition for there to exist such an $f$ is that the unit in $\cA(X)$ must be a symmetrized element, i.e., $1\in \cSymd_{\cA,\sigma}(X)$. Orthogonal involutions which have this property are called \emph{locally quadratic} since, by \cite[Lemma 3.15]{GNR2023AzumayaObstructions}, they may at least be equipped with semi-traces locally. 

If $(\cA,\sigma)$ is an Azumaya algebra with locally quadratic involution, then we may describe all possible semi-traces $f$ which make $(\cA,\sigma,f)$ a quadratic triple. First, since $\cAlt_{\cA,\sigma}$ lies in the kernel of the map $1+\sigma \colon \cA \to \cSymd_{\cA,\sigma}$, there is a commutative diagram
\[
\begin{tikzcd}
\cA \ar[r,"1+\sigma"] \ar[d,twoheadrightarrow] & \cSymd_{\cA,\sigma} \ar[d,equals] \\
\cA/\cAlt_{\cA,\sigma} \ar[r,"\xi"] & \cSymd_{\cA,\sigma}
\end{tikzcd}
\]
where $\xi$ is the map induced by $1+\sigma$. Then, \cite[Theorem 3.18]{GNR2023AzumayaObstructions} states that there is a bijection between the set of possible semi-traces $f$ and the set of global sections $\lambda \in (\cA/\cAlt_{\cA,\sigma})(X)$ such that $\xi(\lambda) = 1 \in \cSymd_{\cA,\sigma}(X)$. For such a section $\lambda \in (\cA/\cAlt_{\cA,\sigma})(X)$, there will exist a cover $\{X_i \to X\}_{i \in I}$ over which $\lambda|_{X_i}$ is the image of some section $\ell_i \in \cA(X_i)$ with $\ell_i+\sigma(\ell_i)=1$, and then the semi-trace $f$ associated to $\lambda$ is uniquely defined by the fact that
\[
f|_{X_i}(s) = \Trd_{\cA}(\ell_i \cdot s)
\]
locally for any section $s\in \cSym_{\cA,\sigma}(X_i)$.

\subsection{Group schemes and Lie algebra sheaves}
In the study of the deformation theory of Azumaya algebras, possibly with additional structure, both the automorphism group schemes associated to such objects and their Lie algebra sheaves play an important role. Here we review the definitions of these objects.

\subsubsection{Algebraic groups}
Given any sheaf $\cF$ on $X$, possibly with some additional structure (such as being an $\cO_X$--module, an $\cO_X$--algebra, etc.), the \emph{automorphism sheaf} $\bAut(\cF)$ is the sheaf of groups on $X$ is defined by
\[
\bAut(\cF)(U) = \Aut(\cF|_U),
\]
for $U\subseteq X$ open, i.e., this is the sheaf of internal automorphisms which respect any present additional structure. For Azumaya $\cO_X$--algebras and involutions we obtain the following familiar group sheaves.
\begin{enumerate}
\item The \emph{projective general linear group} of $\cA$ is the group sheaf of $\cO_X$--algebra automorphisms $\PGL_\cA = \bAut(\cA)$. We write $\PGL_n = \PGL_{\Mat_n(\cO_X)}$.
\item If $(\mathcal{A},\sigma,f)$ is a quadratic triple, then the \emph{projective general orthogonal group of $(\mathcal{A},\sigma,f)$}, denoted by $\PGO_{(\cA,\sigma,f)} = \bAut(\cA,\sigma,f)$, is the subgroup sheaf of $\PGL_\cA$ consisting of those automorphisms $\varphi$ such that $\varphi \circ \sigma = \sigma \circ \varphi$ and $f\circ \varphi = f$. 

In the case that the quadratic triple is of the form $(\Mat_{2n}(\cO_X),\sigma_{2n},f_{2n})$, where $\sigma_{2n}$ is the orthogonal involution
\[
\sigma_{2n}(B) = \begin{bmatrix} & & 1 \\ & \iddots & \\ 1 & & \end{bmatrix} B^T \begin{bmatrix} & & 1 \\ & \iddots & \\ 1 & & \end{bmatrix},
\] which has the effect of reflecting a matrix $B \in \Mat_{2n}(\cO_X)$ across the second diagonal, and the semi-trace $f_{2n}$ is given by
\[
f_{2n}\left(\begin{bmatrix} a_1 & \hdots & * & * & \hdots & * \\ \vdots & \ddots & \vdots & \vdots & & \vdots \\ * & \hdots & a_n & * & \hdots & * \\ * & \hdots & * & a_n & \hdots & * \\ \vdots & & \vdots & \vdots & \ddots & \vdots \\ * & \hdots & * & * & \hdots & a_1 \end{bmatrix}\right) = \sum_{i=1}^n a_i,
\] then we write $\PGO_{2n}$ for $\PGO_{(\Mat_{2n}(\cO_X),\sigma_{2n},f_{2n})}$.

\end{enumerate}
By \cite[2.4.4.1]{CF2015GroupesClassiques} and \cite[4.4.0.32]{CF2015GroupesClassiques} (which are stated more generally over the fppf-site), these two group sheaves are represented by group schemes affine over $X$. Further, the representing group schemes are all smooth, and thus also flat, over $X$ as this is true for their associated split forms. We use the same notation to refer to the representing group scheme as the sheaf, thus effectively extending the sheaf to the category $\Sch_X$ via the usual convention that for $T\in \Sch_X$, we set
\[
\PGL_\cA(T) = \Hom_{\Sch_X}(T,\PGL_\cA)
\]
and likewise for the other groups. In general, if $\cF$ is an $\cO_X$--module, $\cO_X$--algebra, etc., and $\bG=\bAut(\cF)$ is represented by a group scheme, then for each scheme map $g\colon T \to X$ there is a natural identification
\[
\bG(T) = \Aut_{\cO_T}(g^*(\cF)).
\]
We denote by $g^*(\bG)$, or $\bG|_T$ when $g$ is clear, the restriction of $\bG$ to $T$, i.e., the $T$--group scheme represented by $\bG\times_X T$. It has the property that for a scheme $T' \in \Sch_T$ with $X$-scheme structure $T' \to T \to X$, we have
\[
g^*(\bG)(T') = \bG(T').
\]
One of the benefits of using quadratic pairs is that we obtain smooth representing group schemes. In characteristic $2$, if $\sigma$ is an orthogonal involution which is not symplectic on an Azumaya algebra $\mathcal{A}$ and, if we simply consider $\bAut(\cA,\sigma)$ without a semi-trace, then, in general, the group sheaf $\bAut(\cA,\sigma)$ is not representable by a smooth group scheme.

\subsubsection{Lie algebra sheaves}\label{Lie_Algebras}
Here we follow \cite[2.8]{CF2015GroupesClassiques}, adapting their approach using the fppf-site to our setting. A similar approach appears in \cite[\S 21.A]{KMRT}.

Denote by $X[\varepsilon]$ the \emph{scheme of dual numbers} over $X$. By definition, $X[\varepsilon]$ is the scheme with the same underlying topological space as $X$ and with structure sheaf $\cO_X[\varepsilon]\cong \cO_X[x]/(x^2)$, i.e., $\varepsilon^2 = 0$. The canonical projection $\cO_X[\varepsilon] \surj \cO_X$ and its section $\cO_X \inj \cO_X[\varepsilon]$ define a closed immersion $i_X \colon X \to X[\varepsilon]$ and a morphism $p_X \colon X[\varepsilon] \to X$ respectively. We consider $X[\varepsilon]\in \Sch_X$ with structure map $p_X$. 

Note that, by construction, we have that $p_{X,*}(\cO_{X[\varepsilon]}) = \cO_X[\varepsilon]$. Also, for each open subscheme $U\subseteq X$, there is a corresponding open subscheme $U[\varepsilon]\subseteq X[\varepsilon]$ with the same topological space as $U$ and with structure sheaf $\cO_U[\varepsilon]$. The maps $i_X$ and $p_X$ restrict naturally to maps we call $i_U$ and $p_U$.

Now let $\bG$ be a sheaf of groups on $X$ represented by a group scheme over $X$. For each $U\subseteq X$ open, we obtain a split exact sequence of groups
\begin{equation}
\begin{tikzcd}
1 \ar[r] & \Ker(\bG(i_U)) \ar[r,hookrightarrow] & \bG(U[\varepsilon]) \ar[r,yshift=0.5ex,"\bG(i_U)"] & \bG(U) \ar[l,yshift=-0.5ex,"\bG(p_U)"] \ar[r] & 1
\end{tikzcd}
\end{equation}
where $\bG(i_U)$ is the restriction map
\[
\bG(U[\varepsilon]) = \Hom_{\Sch_X}(U[\varepsilon],\bG) \to \Hom_{\Sch_X}(U,\bG) = \bG(U)
\]
induced by $i_U$, and likewise for $\bG(p_U) \colon \bG(U) \to \bG(U[\varepsilon])$. The group $\Ker(\bG(i_U))$ has the structure of a Lie algebra over $\cO_X(U)$. According to \cite{CF2015GroupesClassiques}, these are the $U$ points of the Lie algebra $\Lie_\bG$ of $\bG$, i.e., $\Lie_\bG$ is an $\cO_X$--module such that for all open $U\subset X$ we have $\Lie_\bG(U) = \ker(\bG(i_U))$. Alternatively, there is an exact sequence of group sheaves over $X$,
\begin{equation}\label{Lie_exact_sequence}
\begin{tikzcd}
1 \ar[r] & \Lie_\bG \ar[r,hookrightarrow] & p_{X*}(\bG|_{X[\varepsilon]}) \ar[r,yshift=0.5ex,"i'"] & \bG \ar[r] \ar[l,yshift=-0.5ex,"p'"] & 1
\end{tikzcd}
\end{equation}
where the maps $i'$ and $p'$ are induced by the various $\bG(i_U)$ and $\bG(p_U)$.

There is an equivalent definition of the Lie algebra sheaf of $\mathbf{G}$ in terms of the module of differentials of the scheme $\mathbf{G}$ over $X$. Let $p\colon \bG \to X$ be the structure map and $e\colon X \to \bG$ be the unit section. The module of relative differentials of $\bG$ over $X$ is the $\cO_\bG$--module $\Omega_{\bG/X} = \Omega_{\cO_\bG/p^{-1}(\cO_X)}$ (see \cite[\href{https://stacks.math.columbia.edu/tag/08RT}{Tag 08RT(2)}]{Stacks} for the definition). In \cite{CF2015GroupesClassiques}, Calm{\`e}s and Fasel use the notation $\omega_{\bG/X}^1 = e^*(\Omega_{\bG/X})$ for this module pulled back to $X$ along $e$. They then state in \cite[2.8.0.38]{CF2015GroupesClassiques} that the Lie algebra of $\bG$ is $\bW\big((\omega_{\bG/X}^1)^\vee\big)$, using the $\bW$ functor of \cite[Exp I, 4.6.1]{SGA3} to extend to the fppf-site. Working on a scheme as we do, we may simply write
\begin{equation}\label{Lie_algebra_dual}
\Lie_\bG = (e^*(\Omega_{\bG/X})\big)^\vee.
\end{equation}
So, for $U\subseteq X$ open, we have that
\begin{align*}
\Lie_\bG(U) \cong \Hom_{\cO_U}(e^*(\Omega_{\bG/X})|_U,\cO_U).
\end{align*}

We now give concrete realizations of the Lie algebra sheaves for the group schemes $\PGL_\cA$ and $\PGO_{(\cA,\sigma,f)}$. These are given over fields in \cite[23.A, 23.B]{KMRT} and analogous descriptions hold here. If $\cA$ is an Azumaya $\cO_X$--algebra, then $\PGL_\cA|_{X[\varepsilon]}$ is the group of $\cO_{X[\varepsilon]}$--algebra automorphisms of $\cA[\varepsilon]=\cA\otimes_{\cO_X} \cO_X[\varepsilon]$. The Lie algebra sheaf of $\PGL_\cA$ is then
\begin{align*}
\fpgl_\cA \cong \cA/\cO_X
\end{align*}
and it fits into \eqref{Lie_exact_sequence} over some $U\subseteq X$ open as in the following diagram.
\[
\begin{tikzcd}
0 \arrow{r} & \fpgl_\cA(U) \arrow[hookrightarrow]{r}\arrow["\sim" labl]{d} & p_{X*}(\PGL_\cA|_{X[\varepsilon]})(U) \arrow{r}\arrow[equals]{d} & \PGL_\cA(U) \arrow{r} & 1 \\
 &(\cA/\cO_X)(U)\arrow[hookrightarrow]{r} &  \Aut_{\cO_{U[\varepsilon]}}(\cA|_U[\varepsilon]) & & \\[-1.8 em]
 & \overline{a} \arrow[mapsto]{r} & \Inn(1 + \varepsilon a) & & 
\end{tikzcd}\]
Here we view $\fpgl_\cA$ additively and the other groups multiplicatively. The Lie bracket is given by the commutator $[\overline{a},\overline{b}]=\overline{ab-ba}$. Since $\sigma(\cO_X)=\cO_X$, the involution descends to an involution $\overline{\sigma}$ of $\fpgl_\cA$, i.e., an order two $\cO_X$--module isomorphism which satisfies $\overline{\sigma}([\overline{a},\overline{b}])=[\overline{\sigma}(\overline{b}),\overline{\sigma}(\overline{a})]$.
\begin{rem}\label{lem: liepgl}
Alternatively, since $\cA/\cO_X \cong \cDer_{\cO_X}(\cA,\cA)$ by the map $x\mapsto \ad(x)=(a\mapsto xa-ax)$ sending a section $x$ of $\cA/\cO_X$ to the associated adjoint derivation, one may view the Lie algebra sheaf $\fpgl_\cA$ in terms of derivations. In this case, the map $\cDer_{\cO_X}(\cA,\cA)(U) \to \Aut_{\cO_{U[\varepsilon]}}(\cA|_U[\varepsilon])$ sends a derivation $\Delta$ to the map
\[
a+\varepsilon b \mapsto a+\varepsilon b + \varepsilon\Delta(a).
\]
\end{rem}

If $(\cA,\sigma,f)$ is now a quadratic triple, then the Lie algebra sheaf of $\PGO_{(\cA,\sigma,f)}$ is the preimage of the subgroup sheaf $p_{X*}(\PGO_{(\cA,\sigma,f)}|_{X[\varepsilon]}) \subset p_{X*}(\PGL_\cA|_{X[\varepsilon]})$ under the map $\fpgl_\cA \to p_{X*}(\PGL_\cA|_{X[\varepsilon]})$ given above. Namely,
\[
\fpgo_{(\cA,\sigma,f)}(U) = \{x \in \fpgl_\cA(U) \mid x+\overline{\sigma}(x) = 0,\; f\circ \ad(x)|_{\cSym_{(\cA,\sigma)}} = 0 \}
\]
for $U\subseteq X$ open.

\subsection{$\mathfrak{A}_1\times \mathfrak{A}_1$, $\mathfrak{D}_2$, and the norm functor}\label{norm_functor}
Here we review the result \cite[Theorem 5.13]{GNR2024TheNormFunctor}, which determines that there is an equivalence of stacks between the stack of quaternion algebras over a degree $2$ \'etale cover of schemes and the stack of degree $4$ quadratic triples. We have made some slight modifications to the notation since, in \cite{GNR2024TheNormFunctor}, the authors work with sheaves on the fppf-site. In this section, all schemes are considered over a fixed base field $k$.

First, the stack of quaternion algebras over degree $2$ \'etale covers, denoted $\fA_1^2$ or $\fA_1\times\fA_1$ since it is related to algebraic groups of Dynkin type $A_1\times A_1$, has:
\begin{enumerate}
    \item objects which are pairs $(T'\to T,\cQ)$ where $T'\to T$ is a degree $2$ \'etale cover of $k$--schemes and $\cQ$ is a degree $2$ (i.e., quaternion) Azumaya $\cO_{T'}$--algebra on $T'$,
    \item morphisms which are triples $(f,g,\varphi) \colon (T_1' \to T_1,\cQ_1) \to (T_2' \to T_2,\cQ_2)$ where $f\colon T_1' \to T_2'$ and $g\colon T_1 \to T_2$ are $k$--scheme morphisms making
\[
\begin{tikzcd}
T_1' \ar[d] \ar[r,"f"] & T_2' \ar[d] \\
T_1 \ar[r,"g"] & T_2
\end{tikzcd}
\]
a Cartesian diagram of schemes (so that $T_1'$ shares the universal property of $T_1\times_{T_2} T_2'$) and where $\varphi \colon \cQ_1 \iso f^*(\cQ_2)$ is an isomorphism of $\cO_{T_1'}$--Azumaya algebras, and
    \item structure functor $p\colon \fA_1^2 \to \Sch_k$ defined by setting $p(T'\to T,\cQ) = T$ and $p(f,g,\varphi)=g$.
\end{enumerate}
Here by \emph{a degree $2$ \'etale cover}, we mean the same as in \cite[2.5.2]{CF2015GroupesClassiques}: a $k$-morphism $f\colon T' \to T$ is an \'etale cover of degree $n$ if $f$ is an affine morphism and the sheaf $f_*(\cO_{T'})$ is \'etale locally isomorphic to $\cO_T^n$ as $\cO_T$--algebras (this description also follows from \cite[\href{https://stacks.math.columbia.edu/tag/04HN}{Tag 04HN}]{Stacks}). Pullbacks in $\fA_1^2$ can be taken as follows. For an arbitrary morphism $g\colon X\to T$ in $\Sch_k$ and an object $(T'\to T,\cQ) \in \fA_1^2(T)$ in the fiber over $T$, we get the fiber product diagram
\[
\begin{tikzcd}
X\times_T T' \ar[r,"g'"] \ar[d] & T' \ar[d] \\
X \ar[r,"g"] & T,
\end{tikzcd}
\]
where $X\times_T T' \to X$ is also a degree $2$ \'etale cover since such maps are stable under base change. The morphism $(g',g,\Id)\colon (X\times_T T' \to X,g'^*(\cQ)) \to (T' \to T,\cQ)$ is cartesian, and thus $(X\times_T T' \to X,g'^*(\cQ))$ is a pullback of $(T'\to T,\cQ)$ to $X$.

Second, the stack of degree $4$ quadratic triples, denoted $\fD_2$ because of its relation to groups of Dynkin type $D_2$, has
\begin{enumerate}
    \item objects which are pairs $(T,(\cA,\sigma,f))$ of a $k$-scheme $T$ and a quadratic triple $(\cA,\sigma,f)$ on $T$ such that the $\cO_T$--Azumaya algebra $\cA$ has degree $4$,
    \item morphisms which are pairs $(g,\varphi)\colon (T_1,(\cA_1,\sigma_1,f_1)) \to (T_2,(\cA_2,\sigma_2,f_2))$ of a $k$-scheme morphism $g\colon T_1 \to T_2$ and an isomorphism  $\varphi \colon (\cA_1,\sigma_1,f_1)\iso g^*(\cA_2,\sigma_2,f_2)$ of quadratic triples on $T_1$,
    \item and structure functor $p\colon \fD_2 \to \Sch_k$ defined by setting $p(T,(\cA,\sigma,f))=T$ and $p(g,\varphi) = g$.
\end{enumerate}

\begin{thm}[{\cite[5.13]{GNR2024TheNormFunctor}}]\label{norm_functor_equiv}
There is an equivalence of stacks
\begin{align*}
N \colon \fA_1^2 &\to \fD_2 \\
(T'\to T,\cQ) &\mapsto (T,(N_{T'/T}(\cQ),\sigma_N,f_N))
\end{align*}
induced by norm functors $N_{T'/T}$ associated to the finite \'etale extensions $T'/T$.
\end{thm}

We now review the explicit description of $(N_{T'/T}(\cQ),\sigma_N,f_N)$ occuring above in the case when the \'etale extension $T'\to T$ is the split \'etale extension, i.e.\ $T\sqcup T \to T$. In this case, a quaternion $\cO_{T\sqcup T}$--algebra is the data of two quaternion $\cO_T$--algebras $\cQ_1$ and $\cQ_2$. Precisely, an open subset of $T\sqcup T$ is of the form $U\sqcup V$ for $U,V\subset T$ open and
\[
\cQ(U\sqcup V) = \cQ_1(U)\times \cQ_2(V).
\]
By \cite[3.11]{GNR2024TheNormFunctor}, the norm of such an algebra is simply the $\cO_T$--algebra
\[
N_{(T\sqcup T)/T}(\cQ) = \cQ_1\otimes_{\cO_T} \cQ_2,
\]
which is clearly degree $4$ since each $\cQ_i$ is degree $2$. To describe the involution $\sigma_N$, we first recall from \cite[Example 2]{MR3908769} the fact
that every quaternion algebra $\cQ$ has a canonical (unique up to isomorphism) symplectic involution defined by
\begin{equation}
\psi \colon \cQ \to \cQ, \quad \quad 
q \mapsto \Trd_\cQ(q)\cdot 1 - q.\end{equation}
Thus, the quaternion algebras $\cQ_i$ have their canonical symplectic involutions $\psi_i$ and the tensor product 
$\sigma_N = \psi_1 \otimes \psi_2$ of these involutions is an orthogonal involution on $\cQ_1\otimes_{\cO_T}\cQ_2$. Finally, the semi-trace $f_N$ is described by \cite[Proposition 4.6]{GNR2023AzumayaObstructions} to be the unique semi-trace $f_N \colon \cSym_{\cQ_1\otimes \cQ_2,\psi_1\otimes\psi_2}\rightarrow \mathcal{O}_T$ such that $f_N(s_1\otimes s_2) = 0$ for all sections $s_i \in \cSkew_{\cQ_i,\psi_i}(U)$ for all open $U\subset X$.

\subsection{Deformations and obstructions}\label{deformations}
Throughout this text we will consider a number of different ``deformation situations". The basic set-up is that we will have an object, a scheme, or a morphism of schemes and we will ask when such objects admit infinitesimal extensions. Such extensions will exist, for example, whenever our object is genuinely defined as the fiber of some family.

\subsubsection{Relative deformations}\label{relative_deformations}
Throughout this section we consider two Artinian local $k$-algebras $(C,\mathfrak{m}_C)$ and $(C',\mathfrak{m}_{C'})$ both having residue field $k$, i.e.\ $C/\mathfrak{m}_C\cong k$ and $C'/\mathfrak{m}_{C'}\cong k$. We also fix a surjective $k$-algebra homomorphism $C'\rightarrow C$ with kernel an ideal $J\subset C'$.

We are going to work in the following setting. We assume we are given
\begin{enumerate}
\item a scheme $X$ over $k$,
\item a scheme $\dX$ flat over $C$ together with a closed immersion $i:X\rightarrow \dX$ inducing an isomorphism $X\cong \dX\times_C k$, and
\item a scheme $\dX'$ flat over $C'$ together with a closed immersion $\ti:\dX\rightarrow \dX'$ inducing an isomorphism $\dX\cong \dX'\times_{C'} C$.
\end{enumerate}
In a diagram:
\[
\begin{tikzcd}
\dX \ar[rr,hook,"\ti"] \ar[dd] & & \dX' \ar[dd] \\
 & X \ar[ul,hook',swap,"i"] \ar[ur,hook,"i_0"] & \\[-15pt]
\Spec(C) \ar[rr] & & \Spec(C') \\
 & \Spec(k) \ar[ul] \ar[ur] \ar[from=uu,crossing over]&
\end{tikzcd}
\]
where the vertical square faces are fiber product diagrams and we set $i_0 = \ti \circ i$.

\begin{defn}\label{defn: defsheaves}
Let $\mathcal{F}$ be a coherent sheaf on $\dX$ flat over $C$. A deformation of $\mathcal{F}$ to $\dX'$ is a coherent sheaf $\mathcal{F}'$ on $\dX'$ flat over $C'$ together with an isomorphism $g\colon \ti^*\mathcal{F}'\iso \mathcal{F}$.

If $\mathcal{F}$ is an invertible sheaf (or a finite locally free sheaf), then we say that $\mathcal{F}'$ is a deformation of an invertible sheaf (or of a finite locally free sheaf) if $\mathcal{F}'$ has the same property as well.
\end{defn}

\begin{defn}\label{defn: defaz}
Let $\mathcal{A}$ be an Azumaya algebra on $\dX$. A deformation of $\mathcal{A}$ to $\dX'$ is an Azumaya $\mathcal{O}_{\dX'}$-algebra $\mathcal{A}'$ together with an isomorphism $g \colon \ti^*\mathcal{A}'\iso \mathcal{A}$ of Azumaya $\mathcal{O}_{\dX}$-algebras.
\end{defn}

\begin{defn}\label{defn: definvs}
Let $\mathcal{A}$ be an Azumaya algebra on $\dX$ and suppose that $\sigma$ is an involution on $\mathcal{A}$. A deformation of the Azumaya algebra with involution $(\mathcal{A},\sigma)$ to $\dX'$ is a pair $((\mathcal{A}',g),\sigma')$ such that:
\begin{enumerate}
\item $(\mathcal{A}',g:\ti^*\mathcal{A}'\iso \mathcal{A})$ is a deformation of $\mathcal{A}$ to $\dX'$ as an Azumaya algebra,
\item $\sigma'$ is an involution on $\mathcal{A}'$ making the diagram below
\begin{equation*}
\begin{tikzcd}
\ti^*\mathcal{A}'\arrow["g"]{d}\arrow["\ti^*\sigma'"]{r} & \ti^*\mathcal{A}'\arrow["g"]{d}\\ \mathcal{A}\arrow["\sigma"]{r} & \mathcal{A}
\end{tikzcd} 
\end{equation*} a commutative square. In particular, $\ti^*\sigma'=g^{-1}\circ \sigma \circ g$.
\end{enumerate}
\end{defn}

\begin{defn}\label{defn: defpairs}
 Again, let $\mathcal{A}$ be an Azumaya algebra on $\dX$. Let $\sigma$ be an involution on $\mathcal{A}$ and suppose also that $\sigma$ is part of a quadratic pair $(\sigma,f)$ on $\mathcal{A}$. We say that a triple $((\mathcal{A}',g),\sigma',f')$ is a deformation to $\dX'$ of the Azumaya algebra with quadratic pair $(\mathcal{A},\sigma,f)$ if:
 \begin{enumerate}
 \item the pair $((\mathcal{A}',g),\sigma')$ is a deformation of the Azumaya algebra with involution $(\mathcal{A},\sigma)$ and
 \item $f':\cSym_{\mathcal{A}',\sigma'}\rightarrow \mathcal{O}_{\dX'}$ is a semi-trace such that the associated diagram
\begin{equation*}
\begin{tikzcd}
\ti^*\cSym_{\mathcal{A}',\sigma'}\arrow["\ti^*f'"]{r}\arrow[d,"g"] & \ti^*\mathcal{O}_{\dX'}\ar[d,equals] \\
\cSym_{\mathcal{A},\sigma}\arrow["f"]{r} & \mathcal{O}_\dX
\end{tikzcd}
\end{equation*} is commutative, where $g$ restricts to a map between the respective symmetric elements since $\ti^*\sigma'=g^{-1}\circ \sigma \circ g$.
\end{enumerate}
\end{defn}

Most of the time, we will be interested in deformations of sheaves with additional structures such as those listed in \Crefrange{defn: defsheaves}{defn: defpairs}. Occasionally though, we will also need a reference for deformations of certain types of morphisms of schemes. For this, let $Y$ be any scheme flat over $C$.
\begin{defn}\label{defn: defetmap}
Suppose that we are given an arbitrary $C$-morphism $f:Y\rightarrow \dX$. We say that a pair $((Y',j:Y\rightarrow Y'),f':Y'\rightarrow \dX')$ is a deformation of the morphism $f$ with target $\dX'$ if:
\begin{enumerate}
\item $(Y',j)$ is a deformation of $Y$, i.e.\ $Y'$ is a flat $C'$-scheme and $j$ is a closed immersion inducing an isomorphism $Y\cong Y'\times_{C'} C$;
\item $f':Y'\rightarrow \dX'$ is a morphism giving an equality $\ti\circ f'=f\circ j$.
\end{enumerate}
\end{defn}

\subsubsection{Infinitesimal automorphisms}
Here, we work with a deformation setup as above. In particular, since the surjection $C' \surj C$ is defined by the ideal $J\subset C'$, the closed embedding $\dX \inj \dX'$ is defined by the quasi-coherent ideal $\cJ = \cO_{\dX'}\otimes_{C'} J$. Throughout this subsection, we make the additional assumption that $J\cdot \mathfrak{m}_{C'}=0$. In particular, this implies that $J^2=0$ and thus also $\cJ^2 = 0$.

\begin{rem}\label{rem: smallext}
Fix a field $k$. Let $(C',\mathfrak{m}_{C'})$ and $(C,\mathfrak{m}_C)$ be two local Artinian $k$-algebras both with residue field $k$. We say that a surjection $C'\twoheadrightarrow C$ of $k$-algebras with kernel $J\subset C'$ is a \textit{small extension} if $J\cdot \mathfrak{m}_{C'}=0$. This agrees with the definition of a small extension given in \cite[Definition 6.1.9]{MR2223408}.

There is another definition of a small extension, that Hartshorne gives in \cite[\S 16]{MR2583634}, which requires $\dim_k J=1$.

Any surjection of local Artinian $k$-algebras with residue field $k$ can be factored into a sequence of small extensions by Nakayama's lemma. Further, any small extension of local Artinian $k$-algebras $(J\subset C', C)$, as above, can be factored into a sequence of small extensions with $1$-dimensional kernel by choosing a basis for $J$. So, there is not much loss of generality choosing between the two definitions.
\end{rem}

Note that, since $\dX \to \dX'$ is an infinitesimal thickening, both schemes have the same underlying topological space. Thus, for each open subscheme $U \subseteq \dX$, there is a corresponding open subscheme $U'\subseteq \dX'$ and an infinitesimal thickening $i_U \colon U \inj U'$ defined by the ideal $\cJ|_{U'} \subset \cO_{U'}$.

We assume we have a sheaf $\dF'$ on $\dX'$ such that $\dG' = \bAut(\dF')$ is represented by a group scheme over $\dX'$. For example, $\dF'$ may be a finite locally free $\cO_{\dX'}$--module, an Azumaya $\cO_{\dX'}$-algebra, or an Azumaya $\cO_{\dX'}$--algebra with additional structure. We set $\dF = \ti^*(\dF')$ and $\cF = i^*(\ti^*(\dF'))$ and denote their automorphism groups $\dG$ and $\bG$ respectively. Thus, $\dF'$ is a deformation of $\dF$. Note that for $U\subseteq \dX$ open, we have that $\dG'(U) = \dG(U)$, as well as similar relations for open sets in $X$.

Since the pullbacks we consider are functorial, for each open subscheme $U \subseteq \dX$ and automorphism $\varphi \colon \dF'|_{U'} \iso \dF'|_{U'}$, there is a pullback of this morphism
\[
\ti_U^*(\varphi) \colon \ti_U^*(\dF'|_U) \iso \ti_U^*(\dF'|_U)
\]
which defines an automorphism of $\dF|_U$. Thus, we obtain a restriction map \[\dG'(U') \to \dG'(U) = \dG(U).\] Alternatively, this restriction map is simply the morphism $\dG'(\ti_U)$ induced by the closed embedding $\ti_U \colon U \to U'$. An \emph{infinitesimal automorphism of $\dF'$ over $U'$} is an element in the kernel of this restriction map. We also have a restriction morphism of group sheaves on $\dX'$,
\begin{equation}\label{eq: restriction_morphism}
\res \colon \dG' \to \ti_*(\ti^*(\dG'))
\end{equation}
where over any $U' \subseteq \dX'$, we identify
\[
\ti_*(\ti^*(\dG'))(U') = \ti^*(\dG')(U) = \dG'(U)
\]
and $\res(U) \colon \dG'(U') \to \dG'(U)$ is the restriction map $\dG'(\ti_U)$. We wish to describe the kernel of this restriction morphism, which will be a group sheaf of infinitesimal automorphisms.

Let $\cR$ be a commutative $\cO_X$--algebra and let $\cM$ be an $\cR$--module. Consider the sheaf $\cDer_{\cO_X}(\cR,\cM)$ on $X$ defined over $U\subseteq X$ by
\[
\cDer_{\cO_X}(\cR,\cM)(U) = \Der_{\cO_U}(\cR|_U,\cM|_U)
\]
where a derivation $d \colon \cR|_U \to \cM|_U$ is an $\cO_U$--module morphism satisfying the Liebniz rule $d(ab)=ad(b)+d(a)b$ for appropriate sections $a,b$ of  $\cR|_U$. The sheaf $\cDer_{\cO_X}(\cR,\cM)$ is an $\cO_X$--module.

\begin{prop}\label{infinitesimal_autos_general}
Let $\ti \colon \dX \inj \dX'$ be a closed immersion defined by a square-zero quasi-coherent ideal $\cJ \subset \cO_{\dX'}$. Let $\dG'$ be a group scheme over $\dX'$. Let $p \colon \dG' \to \dX'$ and $e\colon \dX' \to \dG'$ be the structure morphism and identity section of $\dG'$, respectively. Consider $e^{-1}(\cO_{\dG'})$ as an $\cO_{\dX'}$--algebra via the map $e^{-1}(p^\sharp)\colon \cO_{\dX'} \to e^{-1}(\cO_{\dG'})$ and consider $\cJ$ as an $e^{-1}(\cO_{\dG'})$--module via $e^\sharp \colon e^{-1}(\cO_{\dG'}) \to \cO_{\dX'}$. 

Then, the kernel of the restriction map $\dG'(i) \colon \dG'(\dX') \to \dG'(\dX)$ is isomorphic to the group of derivations $\Der_{\cO_{\dX'}}(e^{-1}(\cO_{\dG'}),\cJ)$
\end{prop}

\begin{proof}
Because $\cJ$ is square-zero, the inclusion $i\colon \dX \to \dX'$ is a homeomorphism on the underlying topological spaces and we can consider $\cO_{\dX} = \cO_{\dX'}/\cJ$. The identity in $\dG'(\dX)$ is the morphism $\dX \xrightarrow{e\circ i} \dG'$, so an element in the kernel of $\dG'(i)$ is a morphism $X'\rightarrow \dG'$ in $\Sch_{\dX'}$, which must be a section of $p$, making the diagram
\[
\begin{tikzcd}
 & \dG' \\
\dX \ar[r,"i"] \ar[ur,"e\circ i"] & \dX' \ar[u,dashed]
\end{tikzcd}
\]
commute. Because $i$ is a homeomophism on the underlying spaces, the topological data of such a morphism is uniquely determined and so we only need to care about the morphisms of sheaves on $\dX'$ making the following diagram commute.
\begin{equation}\label{eq: dashedarrow}
\begin{tikzcd}
 & e^{-1}(\cO_{\dG'}) \ar[d,dashed] \ar[dl,swap,"(e\circ i)^\sharp"] \\
\cO_{\dX'}/\cJ & \cO_{\dX'} \ar[l,twoheadrightarrow,swap,near start,"i^\sharp"]
\end{tikzcd}
\end{equation}
Of course, $e^\sharp \colon e^{-1}(\cO_{\dG'}) \to \cO_{\dX'}$ is one possible such morphism. Let $\psi \colon e^{-1}(\cO_{\dG'}) \to \cO_{\dX'}$ be another such morphism. Since $\psi$ must define a section of $p$, we have that $\psi \circ e^{-1}(p^\sharp) = \Id_{\cO_{\dX'}}$. Letting $x$ and $y$ be sections of $\cO_{\dX'}$ and $e^{-1}(\cO_{\dG'})$ respectively, and recalling that the module structure is $x\cdot y = e^{-1}(p^\sharp)(x)y$ by definition, we have
\[
\psi(x\cdot y) = \psi(e^{-1}(p^\sharp)(x)y) = x\psi(y), 
\]
which shows that $\psi$ is $\cO_{\dX'}$--linear. 

We claim that $e^\sharp - \psi$ is a section of $\Der_{\cO_{\dX'}}(e^{-1}(\cO_{\dG'}),\cJ)$. It is clear that $e^\sharp - \psi$ is $\cO_{\dX'}$--linear and that $(e^\sharp - \psi)(y)$ is a section of $\cJ$ for all sections $y$ of $e^{-1}(\cO_{\dG'})$. We also note that this means $e^\sharp(y)j = \psi(y)j$ for any $j$ of $\cJ$ since $\cJ$ is square-zero. Now, we compute for sections $y_1,y_2$ of $\cJ$:
\begin{align*}
(e^\sharp-\psi)(y_1y_2) &= e^\sharp(y_1y_2)-\psi(y_1y_2) \\
&= e^\sharp(y_1)e^\sharp(y_2) - \psi(y_1)\psi(y_2) \\
&= e^\sharp(y_1)e^\sharp(y_2) -e^\sharp(y_1)\psi(y_2) + e^\sharp(y_1)\psi(y_2) - \psi(y_1)\psi(y_2) \\
&= e^\sharp(y_1)(e^\sharp-\psi)(y_2) + (e^\sharp-\psi)(y_1)\psi(y_2) \\
&= e^\sharp(y_1)(e^\sharp-\psi)(y_2) + (e^\sharp-\psi)(y_1)e^\sharp(y_2) \\
&= y_1 \cdot (e^\sharp-\psi)(y_2) + (e^\sharp-\psi)(y_1)\cdot y_2
\end{align*}
in the $e^{-1}(\cO_{\dG'})$--module $\cJ$. This show that $e^\sharp-\psi \colon e^{-1}(\cO_{\dG'}) \to \cJ$ is a derivation.

Conversely, let $d$ be a derivation from $\Der_{\cO_{\dX'}}(e^{-1}(\cO_{\dG'}),\cJ)$. If $e^\sharp + d$ is a homomorphism of sheaves of rings, then it is clear that it may take the place of the dashed morphism in \eqref{eq: dashedarrow} above. Hence, we compute
\begin{align*}
(e^\sharp+d)(y_1)(e^\sharp+d)(y_2) &= (e^\sharp(y_1)+d(y_1))(e^\sharp(y_2)+d(y_2)) \\
&= e^\sharp(y_1)e^\sharp(y_2) + e^\sharp(y_1)d(y_2) + d(y_1)e^\sharp(y_2) + d(y_1)d(y_2) \\
&= e^\sharp(y_1y_2) + y_1\cdot d(y_2) + d(y_1)\cdot y_2 + 0 \\
&= e^\sharp(y_1y_2) + d(y_1y_2) \\
&= (e^\sharp+d)(y_1y_2)
\end{align*}
which shows that $e^\sharp+d$ is a ring sheaf homomorphism as required.

So far, we have shown that we have a bijection of sets
\[
\Ker(\dG'(i)) \leftrightarrow \Der_{\cO_{\dX'}}(e^{-1}(\cO_{\dG'}),\cJ).
\]
It remains to show that this is a group isomorphism. In particular, we argue that multiplication in $\dG'(\dX')$ between elements in the kernel of $\dG'(i)$ corresponds to addition of derivations. We note that the following is inspired by the approach in \cite[Expos\'e 2]{SGA3} where they treat the case when $\cJ$ is a direct summand of $\cO_{\dX'}$.

We consider an additional scheme with the same underlying topological space as $|\dX|=|\dX'|$. First, we define a ring structure on the $\cO_{\dX'}$--module $\cO_{\dX'}\oplus \cJ$ by
\[
(x_1,j_1)(x_2,j_2) = (x_1x_2, x_1j_2 + x_2j_1),
\]
i.e., the direct summand $\cJ$ is a second square-zero copy of $\cJ$ in addition to the one which is a submodule $\cJ\subseteq \cO_{\dX'}$. We denote by $\dX'[\cJ]$ the space $|\dX|$ equipped with this sheaf of rings. We have a closed embedding $i_\cJ' \colon \dX' \inj \dX'[\cJ]$ defined by the square-zero ideal $0\oplus\cJ \subset \cO_{\dX'}\oplus \cJ$. Furthermore, this embedding is split by the map $\pi_\cJ'\colon \dX'[\cJ] \to \dX'$ defined by the canonical inclusion $\cO_{\dX'} \inj \cO_{\dX'}\oplus \cJ$. Both $i_\cJ'$ and $\pi_\cJ'$ are homeomorphisms by construction. The identity section in $\dG'(\dX'[\cJ])$ is the map $e\circ \pi_\cJ' \colon \dX'[\cJ] \to \dG'$.

We now consider the diagram
\begin{equation}\label{eq: pyrimad}
\begin{tikzcd}
 & \dG' \ar[d,xshift=0.5ex,"p"] & \\
\dX \ar[ur,"e\circ i"] \ar[r,"i"] & \dX' \ar[r,"i_\cJ'"] \ar[u,xshift=-0.5ex,"e"] & \dX'[\cJ] \ar[ul,swap,dashed,"\varphi"]
\end{tikzcd}
\end{equation}
and seek to identify which dashed morphisms $\varphi$ belong to $\Ker(\dG'(i_\cJ'\circ i))$, i.e., such that $\varphi\circ i_\cJ' \circ i = e\circ i$. Because both horizontal maps along the bottom are homeomorphisms on the underlying topological spaces, given such a $\varphi$ we would have $\varphi^{-1}(\cO_{\dG'}) = e^{-1}(\cO_{\dG'})$ on $|\dX|$. Thus, a suitable morphism $\varphi$ is determined by a map of sheaves of rings $\varphi^\sharp$ making
\begin{equation}\label{eq: pyrimad2}
\begin{tikzcd}
 & e^{-1}(\cO_{\dG'}) \ar[dr,dashed,"\varphi^\sharp"] \ar[dl,swap,"(e\circ i)^\sharp"]& \\
\cO_{\dX'}/\cJ & \cO_{\dX'} \ar[l,swap,"i^\sharp"] & \cO_{\dX'}\oplus \cJ \ar[l,swap,"i_\cJ'^\sharp"]
\end{tikzcd}
\end{equation}
commute. In particular, $i_\cJ'^\sharp\circ \varphi^\sharp$ must belong to $\Ker(\dG'(i)) = \Der_{\cO_{\dX'}}(e^{-1}(\cO_{\dG'}),\cJ)$ identified above. Therefore, we may write
\[
\varphi^\sharp(y) = (\psi(y),f(y))
\]
for $y$ a section of $e^{-1}(\cO_{\dG'})$ where $\psi = e^\sharp + d$ for some $d$ of $\Der_{\cO_{\dX'}}(e^{-1}(\cO_{\dG'}),\cJ)$. We claim that $f$ is also such a derivation. Since $\varphi^\sharp$ is a ring homomorphism, we can compute, for sections $y_1,y_2$ of $e^{-1}(\mathcal{O}_{\dG'})$,
\begin{align*}
\varphi^\sharp(y_1y_2) &= (\psi(y_1),f(y_1))\cdot(\psi(y_2),f(y_2)) \\
&= (\psi(y_1)\psi(y_2),\psi(y_1)f(y_2)+\psi(y_2)f(y_1)) \\
&= (\psi(y_1y_2), e^\sharp(y_1)f(y_2) + e^\sharp(y_2)f(y_1))
\end{align*}
since the $f(y_i)$ are contained in $\cJ$. Thus, $f(y_1y_2) = e^\sharp(y_1)f(y_2) + e^\sharp(y_2)f(y_1)$, which shows that $f\in \Der_{\cO_{\dX'}}(e^{-1}(\cO_{\dG'}),\cJ)$. Conversely, it is clear that any such derivation can occur in the second factor. Therefore, we have a bijection of sets
\[
\Ker(\dG'(i_\cJ'\circ i)) = \Der_{\cO_{\dX'}}(e^{-1}(\cO_{\dG'}),\cJ) \times \Der_{\cO_{\dX'}}(e^{-1}(\cO_{\dG'}),\cJ).
\]

Next, we consider the morphism of schemes $\sigma \colon \dX' \to \dX'[\cJ]$ defined by the map
\begin{align*}
\sigma^\sharp\colon \cO_{\dX'}\oplus \cJ &\mapsto \cO_{\dX'} \\
(x,j) &\mapsto x+j
\end{align*}
which is a homomorphism of sheaves of rings since the copy of $\cJ$ in $\cO_{\dX'}$ is square-zero. The induced morphism $\dG'(\sigma)\colon \dG'(\dX'[\cJ]) \to \dG'(\dX')$ restricts to the kernels discussed above, and takes the form
\begin{align*}
\dG'(\sigma) \colon \Ker(\dG'(i_\cJ'\circ i)) &\to \Ker(\dG'(i)) \\
(e^\sharp + d, f) &\mapsto e^\sharp + (d+f)
\end{align*}
since on the square-zero ideals, $\sigma^\sharp$ is simply the module addition $\cJ\oplus \cJ \to \cJ$. In other words, we have a commutative diagram
\[
\begin{tikzcd}
\Ker(\dG'(i_\cJ'\circ i)) \ar[d,"\dG'(\sigma)"] \ar[r,equals] & \Der_{\cO_{\dX'}}(e^{-1}(\cO_{\dG'}),\cJ)\times \Der_{\cO_{\dX'}}(e^{-1}(\cO_{\dG'}),\cJ) \ar[d,"+"] \\
\Ker(\dG'(i)) \ar[r,equals] & \Der_{\cO_{\dX'}}(e^{-1}(\cO_{\dG'}),\cJ).
\end{tikzcd}
\]
where we highlight the fact that $\dG'(\sigma)$ is a homomorphism with respect to the group structure within $\dG'$ since, as a sheaf, $\dG'$ is a group functor. In particular, this equips $\Der_{\cO_{\dX'}}(e^{-1}(\cO_{\dG'}),\cJ)$ with two group structures, one of which is defined by a homomorphism with respect to the other. Thus, \cite[Exp 2, 3.10]{SGA3} applies and we conclude that the group structures coincide (and furthermore are commutative, but this is apparent for addition of derivations). This finishes the proof.
\end{proof}

\begin{cor}
Working in the same setting as \Cref{infinitesimal_autos_general}, the kernel of the restriction morphism
\[
\res \colon \dG' \to \ti_*(\ti^*(\dG'))
\]
of \eqref{eq: restriction_morphism} is the sheaf $\cDer_{\cO_{\dX'}}(e^{-1}(\cO_{\dG'}),\cJ)$.
\end{cor}
\begin{proof}
Denote the kernel of the morphism $\res \colon \dG' \to \ti_*(\ti^*(\dG'))$ by $\cK$. Let $U' \subseteq \dX'$ be an open subset. We have a fiber product diagram
\[
\begin{tikzcd}
\dG'\times_{\dX'} U' \ar[d,xshift=0.5ex,"p_U"] \ar[r] & \dG' \ar[d,xshift=0.5ex,"p"] \\
U' \ar[r] \ar[u,xshift=-0.5ex,"e_U"] & \dX' \ar[u,xshift=-0.5ex,"e"]
\end{tikzcd}
\]
where $\dG'\times_{\dX'} U'$ represents the restriction of the sheaf $\dG'$ to $U'$. Additionally, the closed embedding $\ti_U \colon U \to U'$ from the reduction $U= U' \times_{\dX} \dX'$ is defined by the square-zero ideal $\cJ|_{U'}$. Therefore, applying \Cref{infinitesimal_autos_general} we see that
\[
\cK(U') = \Der_{\cO_{U'}}(e_U^{-1}(\cO_{\dG'\times_{\dX'} U'}),\cJ|_{U'}).
\]
However, since $U' \to \dX'$ is an open immersion, so is $\dG'\times_{\dX'} U' \to \dG'$, and therefore $\cO_{\dG'\times_{\dX'} U'} = \cO_{\dG'}|_{\dG'\times_{\dX'} U'}$. Thus, we have that $e_U^{-1}(\cO_{\dG'\times_{\dX'} U'}) = e^{-1}(\cO_\dG)|_{U'}$, and therefore
\begin{align*}
\cK(U') &= \Der_{\cO_{U'}}(e^{-1}(\cO_\dG)|_{U'},\cJ|_{U'}) \\
&= \cDer_{\cO_{\dX'}}(e^{-1}(\cO_{\dG'}),\cJ)(U').
\end{align*}
This means that $\cK = \cDer_{\cO_{\dX'}}(e^{-1}(\cO_{\dG'}),\cJ)$, as desired.
\end{proof}

\begin{lem}\label{Der_reduction_to_X}
Let $i_0 \colon X \to \dX'$ be a closed embedding of schemes defined by a nilpotent ideal $\cM \subseteq \cO_{\dX'}$. Additionally, we consider
\begin{enumerate}
\item an ideal $\cJ \subseteq \cM$ such that $\cM \cdot \cJ = 0$, and 
\item a group scheme $\dG'$ over $\dX'$ such that the structure morphism $p'\colon \dG' \to \dX'$ is flat.
\end{enumerate}
Let $e'\colon \dX' \to \dG'$ be the identity section. Set $\bG = \dG' \times_{\dX'} X$ to be the restriction of $\dG'$ to $X$, with identity section $e\colon X \to \bG$. Then, there is a canonical isomorphism
\[
\cDer_{\cO_{\dX'}}(e'^{-1}(\cO_{\dG'}),\cJ) \cong \cDer_{\cO_X}(e^{-1}(\cO_\bG),\cJ)
\]
on the topological space $|\dX'|=|X|$.
\end{lem}
\begin{proof}
We begin by noting that we have an exact sequence of sheaves on $|X|$, 
\[
0 \to \cM \to \cO_{\dX'} \to \cO_X \to 0.
\]
which we may view as an exact sequence of $\cO_{\dX'}$--modules. Since $p' \colon \dG' \to \dX'$ is flat, for each point $g \in \dG'$ the morphism of local rings
\[
p'_g \colon \cO_{\dX',p'(g)} \to \cO_{\dG',g}
\]
is flat. In particular, for a point $x \in \dX'$, the map of stalks at $e'(x)$, namely
\[
p'_{e'(x)} \colon \cO_{\dX',p'(e'(x))} = \cO_{\dX',x} \to \cO_{\dG',e'(x)},
\]
is flat. However, these are the stalks of the morphism $e'^{-1}(p'^\sharp) \colon \cO_{\dX'} \to e'^{-1}(\cO_{\dG'})$. Hence, using $e'^{-1}(p'^\sharp)$, the functor $\und \otimes_{\cO_{\dX'}} e'^{-1}(\cO_{\dG'})$ is exact on the category of $\cO_{\dX'}$--modules. Thus, we also have an exact sequence
\[
0 \to \cM \otimes_{\cO_{\dX'}} e'^{-1}(\cO_{\dG'}) \to e'^{-1}(\cO_{\dG'}) \to \cO_X \otimes_{\cO_{\dX'}} e'^{-1}(\cO_{\dG'}) \to 0.
\]
Next, we claim that $\cO_X \otimes_{\cO_{\dX'}} e'^{-1}(\cO_{\dG'}) \cong e^{-1}(\cO_{\bG})$. Indeed, since we have a fiber product diagram
\[
\begin{tikzcd}
\bG \ar[r] \ar[d] & \dG' \ar[d,"p'"] \\
X \ar[r,"i_0"] & \dX'
\end{tikzcd}
\]
and $|X| = |\dX'|$, the pullback will share the same property, namely $|\bG| = |\dG'|$. Furthermore, $|e|=|e'|$ and $|p|=|p'|$, by which we mean the maps agree on the underlying topological spaces. To identify the sheaf $\cO_\bG$ on $|\dG'|$, note that we have a diagram of sheaves on $|\dG'|$
\[
\begin{tikzcd}
\cO_\bG & \cO_{\dG'} \ar[l] \\
p^{-1}(\cO_X) \ar[u] & p'^{-1}(\cO_{\dX'}) \ar[u,swap,"p'^\sharp"] \ar[l,swap,"p^{-1}(i_0^\sharp)"]
\end{tikzcd}
\]
which is now cocartesian. Passing this through $e^{-1}=e'^{-1}$ to get a diagram on $|X|$, we obtain
\[
\begin{tikzcd}
e^{-1}(\cO_\bG) & e'^{-1}(\cO_{\dG'}) \ar[l] \\
\cO_X \ar[u] & \cO_{\dX'} \ar[u,swap,"e'^{-1}(p'^\sharp)"] \ar[l,swap,"i_0^\sharp"]
\end{tikzcd}
\]
which shows that $e^{-1}(\cO_{\bG}) \cong \cO_X \otimes_{\cO_{\dX'}} e'^{-1}(\cO_{\dG'})$ as claimed.

Now, let $D \in \cDer_{\cO_{\dX'}}(e'^{-1}(\cO_{\dG'}),\cJ)$ be any derivation. For any $U\subset \dX'$ open and sections $m\in \cM(U)$ and $x\in e'^{-1}(\cO_{\dG'})(U)$, we have that
\[
D(mx) = m\cdot D(x) = 0
\]
since $\cM\cdot \cJ = 0$. Hence, the image of $\cM \otimes_{\cO_{\dX'}} e'^{-1}(\cO_{\dG'}) \inj e'^{-1}(\cO_{\dG'})$ is always in the kernel of $D$. As such, it defines a unique morphism
\[
D_0 \colon e^{-1}(\cO_\bG) \to \cJ
\]
which is $\cO_X$--linear and also a derivation with respect to the appropriate module structures. Conversely, given a derivation $C_0 \colon e^{-1}(\cO_\bG) \to \cJ$ it is clear that the composition
\[
C \colon e'^{-1}(\cO_{\dG'}) \surj e^{-1}(\cO_\bG) \xrightarrow{C_0} \cJ
\]
is an $\cO_{\dX'}$--linear derivation. It is then straightforward to check that these two maps are mutually inverse, i.e., we have a canonical isomorphism
\begin{align*}
\cDer_{\cO_{\dX'}}(e'^{-1}(\cO_{\dG'}),\cJ) &\iso \cDer_{\cO_X}(e^{-1}(\cO_\bG),\cJ) \\
D &\mapsto D_0.
\end{align*}
This completes the proof.
\end{proof}

\begin{lem}\label{lem: infautlie}
Consider schemes $X$, $\dX$, and $\dX'$ defined over $k$, $C$, and $C'$ as in \Cref{relative_deformations}. Let $\bG'$ be a group scheme over $\dX'$ such that the structure morphism $p'\colon \dG' \to \dX'$ is flat. Let $\bG$ be the restriction of $\dG'$ to $X$. Then, the kernel of the restriction morphism \eqref{eq: restriction_morphism} is isomorphic to
\[
\Lie_\bG\otimes_{\cO_X} \cJ
\]
as a sheaf on $|X|=|X'|$.
\end{lem}
\begin{proof}
Recall that the infinitesimal thickening $i_0 \colon X \to \dX'$ is defined by a nilpotent ideal $\cM \subset \cO_{\dX'}$ and the infinitesimal thickening $\ti \colon \dX \to \dX'$ is defined by $\cJ\subseteq \cM$ such that $\cM\cdot \cJ=0$, in particular $\cJ^2=0$. Additionally, $\cJ$ is finite locally free since it comes from the kernel $J$ of the map $C' \surj C$, and $J$ is a finite dimensional $k$--vector space.

Due to our assumptions, we can invoke \Cref{infinitesimal_autos_general} and \Cref{Der_reduction_to_X} to see that the kernel of the restriction morphism is isomorphic to
\[
\cDer_{\cO_X}(e^{-1}(\cO_\bG),\cJ)
\]
as a sheaf on $X$. By definition, the group of derivations $\Der_{\cO_X}(e^{-1}(\cO_\bG),\cJ)$ is isomorphic to the group of $e^{-1}(\cO_\bG)$--linear morphisms $\Hom_{e^{-1}(\cO_\bG)}(\Omega_{e^{-1}(\cO_\bG)/\cO_X},\cJ)$ between the module of relative differentials and $\cJ$. Since $e$ is a section of $p$, we have $p\circ e = \Id_X$ and so $\cO_X = e^{-1}(p^{-1}(\cO_X))$. Now, applying \cite[\href{https://stacks.math.columbia.edu/tag/08RR}{Tag 08RR}]{Stacks}, we have a canonical identification
\[
\Omega_{e^{-1}(\cO_\bG)/\cO_X} = e^{-1}(\Omega_{\cO_\bG/p^{-1}(\cO_X)}).
\]
Notationally, $\Omega_{\cO_\bG/p^{-1}(\cO_X)} = \Omega_{\bG/X}$ is the module of relative differentials, which is a sheaf on $\bG$. Since $e^*(\Omega_{\bG/X}) = e^{-1}(\Omega_{\bG/X})\otimes_{e^{-1}(\cO_\bG)}\cO_X$ by definition, we then have that
\[
\Hom_{e^{-1}(\cO_\bG)}(\Omega_{e^{-1}(\cO_\bG)/\cO_X},\cJ) \cong \Hom_{\cO_X}(e^*(\Omega_{\bG/X}),\cJ).
\]
Using \cite[\href{https://stacks.math.columbia.edu/tag/01V0}{Tag 01V0}]{Stacks} to restrict the module of relative differentials, we also have that for any $U \subseteq X$ open
\[
\Der_{\cO_U}(e^{-1}(\cO_\bG)|_U,\cJ|_U) \cong \Hom_{\cO_U}(e^*(\Omega_{\bG/X})|_U,\cJ|_U),
\]
and so globally we have an isomorphism of sheaves
\[
\cDer_{\cO_X}(e^{-1}(\cO_\bG),\cJ) \cong \cHom_{\cO_X}(e^*(\Omega_{\bG/X}),\cJ).
\]
Since $\cJ$ is finite locally free, we can write
\[
\cHom_{\cO_X}(e^*(\Omega_{\bG/X}),\cJ) \cong \cHom_{\cO_X}(e^*(\Omega_{\bG/X}),\cO_X)\otimes_{\cO_X} \cJ
\]
where the first factor is $\Lie_\bG$ by \eqref{Lie_algebra_dual}, and so we are done.
\end{proof}

\subsubsection{Obstructions and classifying deformations} Given a deformation situation as in \Cref{relative_deformations}, there exists a canonical obstruction class in a corresponding cohomology group which vanishes if and only if such a deformation exists. The set of all deformations is then naturally a torsor for an appropriate vector space. Here we make this precise in each of the cases we are interested in. We note that these results are instances of the very general statement \cite[2.6]{Illusie}, which classifies deformations of torsors. Throughout this subsection we maintain the assumption that we are working with a small extension, i.e., that $J\cdot \mathfrak{m}_{C'}=0$ and so $J^2=0$. 

\begin{lem}[{\cite[Theorem 7.1]{MR2583634}}]\label{lem: deflf}
Let $\dF$ be a locally free sheaf of finite rank on $\dX$. Set $\cF=i^*\dF$. Then:
\begin{enumerate}
\item There is a canonical obstruction class $ob(\dF;\dX'/\dX)\in \mathrm{H}^2(X, \cEnd(\cF)\otimes_{\cO_X} \cJ)$ whose vanishing is a necessary and sufficient condition for the existence of a deformation $\dF'$ of $\dF$ to $\dX'$ as a finite locally free sheaf.
\item If a deformation $\dF'$ of $\dF$ to $\dX'$ exists as a finite locally free sheaf, then the set of all such deformations is a torsor under $\mathrm{H}^1(X, \cEnd(\cF)\otimes_{\cO_X} \cJ)$.
\item If a deformation $\dF'$ of $\dF$ to $\dX'$ exists as a finite locally free sheaf, then the group of automorphisms of $\dF'$ inducing the identity on $\ti^*\dF'\cong \dF$ is isomorphic with $\mathrm{H}^0(X, \cEnd(\cF)\otimes_{\cO_X} \cJ)$.
\end{enumerate}
\end{lem}

Similarly standard is the following analogous statement for Azumaya algebras. In the following we make use of the fact that for any Azumaya algebra $\mathcal{A}$ on a scheme $X$, there is an identification $\mathfrak{pgl}_{\mathcal{A}}\cong \mathcal{A}/\mathcal{O}_X$ (cf.\ \Cref{lem: liepgl} and above).

\begin{lem}[cf.\ {\cite[Lemma 3.1]{MR2060023}}]\label{lem: Azdef}
Assume that $X$ is separated over $k$. Let $\dA$ be an Azumaya algebra on $\dX$ and set $\cA=i^*\dA$. Then the following hold:
\begin{enumerate}
\item\label{lem: aoba} There exists a canonical obstruction $ob(\dA;\dX'/\dX)\in \mathrm{H}^2(X, (\mathcal{A}/\mathcal{O}_X)\otimes_{\cO_X} \cJ)$ whose vanishing is a necessary and sufficient condition for the existence of a deformation $\dA'$ of $\dA$ to $\dX'$ as an Azumaya algebra.
\item\label{lem: atora} If a deformation $\dA'$ of $\dA$ to $\dX'$ exists as an Azumaya algebra, then the set of all such deformations is a torsor under $\mathrm{H}^1(X, (\mathcal{A}/\mathcal{O}_X)\otimes_{\cO_X} \cJ)$.
\item\label{lem: aauta} If a deformation $\dA'$ of $\dA$ to $\dX'$ exists as an Azumaya algebra, then the group of automorphisms of $\dA'$ inducing the identity on $\ti^*\dA'\cong \dA$ is isomorphic with $\mathrm{H}^0(X, (\mathcal{A}/\mathcal{O}_X)\otimes_{\cO_X} \cJ)$.
\end{enumerate}
\end{lem}

\begin{proof}
We first prove \ref{lem: aauta}. Note that an $\mathcal{O}_{\dX'}$-Azumaya algebra automorphism $\phi$ of $\dA'$ that induces the identity on $\ti^*\dA'$ is exactly an element of the kernel of the restriction morphism
\[
K:=\Ker\left(\mathbf{PGL}_{\dA'}(\dX')\rightarrow \mathbf{PGL}_{\dA}(\dX)\right).
\]
Since we are in the setting of \Cref{lem: infautlie}, we know that
\[
K = \mathrm{H}^0(X,\fpgl_{\cA}\otimes_{\cO_X}\cJ) = \mathrm{H}^0(X,(\cA/\cO_X)\otimes_{\cO_X}\cJ)
\]
as claimed.

The proofs of \ref{lem: aoba} and \ref{lem: atora} are nearly identical to the proof of \Cref{lem: deflf} with the sole change that one splits $\dA$ by \'etale covers rather than Zariski open covers. The canonical obstruction class then naturally sits in the \'etale \v{C}ech cohomology group $\check{\mathrm{H}}_{\acute{e}t}^2(X, (\cA/\mathcal{O}_X)\otimes_{\cO_X} \cJ)$ but, since we are assuming that $X$ is separated, this agrees with the second \'etale cohomology group by \cite[III Remark 2.16]{MR559531}. Lastly, since $(\cA/\mathcal{O}_X)\otimes_{\cO_X} \cJ$ is quasi-coherent, we may replace the \'etale cohomology $\mathrm{H}^i_{\acute{e}t}(X,(\cA/\mathcal{O}_X)\otimes_{\cO_X} \cJ)$ with their Zariski counterparts, \cite[III Remark 3.8]{MR559531}.
\end{proof}

An analogous result holds for Azumaya algebras with quadratic pair:

\begin{lem}\label{lem: defqp}
Assume that $X$ is separated over $k$. Let $\dA$ be an Azumaya algebra on $\dX$ with quadratic pair $(\sigma,f)$. Set $\cA=i^*\dA$, $\sigma_0=i^*\sigma$ and $f_0=i^*f$. Then the following hold:
\begin{enumerate}
\item\label{lem: aobq} There exists an obstruction $ob(\dA,\sigma,f;\dX'/\dX)\in \mathrm{H}^2(X, \mathfrak{pgo}_{(\cA,\sigma_0,f_0)}\otimes_{\cO_X} \cJ)$ whose vanishing is a necessary and sufficient condition for the existence of a deformation $(\dA',\sigma',f')$ of $(\dA,\sigma,f)$ to $\dX'$.
\item\label{lem: atorq} If a deformation $(\dA',\sigma',f')$ of $(\dA,\sigma,f)$ to $\dX'$ exists, then the set of all such deformations is a torsor under $\mathrm{H}^1(X,\mathfrak{pgo}_{(\cA,\sigma_0,f_0)}\otimes_{\cO_X} \cJ)$.
\item\label{lem: aautq} If a deformation $(\dA',\sigma',f')$ of $(\dA,\sigma,f)$ to $\dX'$ exists, then the group of automorphisms of $(\dA',\sigma',f')$ which induce the identity on $(\ti^*\dA',\ti^*\sigma',\ti^*f')\cong (\dA,\sigma,f)$ is isomorphic with $\mathrm{H}^0(X, \mathfrak{pgo}_{(\cA,\sigma_0,f_0)}\otimes_{\cO_X} \cJ)$.
\end{enumerate}
\end{lem}

\begin{proof}
Every quadratic triple is \'etale locally split by \cite[Corollaire 2.7.0.32]{CF2015GroupesClassiques}. Moreover, given a quadratic triple like $(\mathcal{A},\sigma_0,f_0)$, the functor of automorphisms of this triple is representable by a smooth group scheme $\mathbf{PGO}_{(\mathcal{A},\sigma_0,f_0)}$. Hence the Lie algebra sheaf $\mathfrak{pgo}_{(\mathcal{A},\sigma_0,f_0)}$ is a coherent $\mathcal{O}_X$-module. The proof is then identical to that of \Cref{lem: Azdef}.
\end{proof}

\subsubsection{Tangent-obstruction theories}\label{tangent_obstruction}
We are going to package some key concepts from deformation theory into a convenient functorial framework using so-called deformation functors. For us, a deformation functor $D:\mathsf{Art}_k\rightarrow \mathsf{Set}$ is a covariant functor from the category $\mathsf{Art}_k$ of Artinian local $k$-algebras such that $D(k)$ consists of a single point. Because $D(k)$ is a single point, we may also view $D$ as a functor into pointed sets, where for each $A\in \mathsf{Art}_k$ the base point of $D(A)$ is the image of the induced map $D(k)\to D(A)$. We tacitly make this identification below.

Here are the deformation functors we will be interested in:

\begin{defn} Let $X$ be a fixed $k$-scheme.
\begin{enumerate}[leftmargin = *]
\item Given a coherent sheaf $\mathcal{F}$ on $X$, the deformation functor associated to $\mathcal{F}$ is
\begin{align*}
\mathrm{Def}_{\mathcal{F}}&:\mathsf{Art}_k\rightarrow \mathsf{Set}, & A&\mapsto \left\{\begin{gathered} \text{Isom.\ classes of deformations}\\ \text{$\mathcal{F}'$ of $\mathcal{F}$ to $X\times_k A$} \end{gathered} \right\}
\end{align*}
with canonical morphisms.
\item If $\mathcal{F}$ is finite locally free, then we write $\mathrm{Def}_{\mathcal{F}, \mathrm{lf}}$ to mean the subfunctor of $\mathrm{Def}_{\mathcal{F}}$ consisting only of those deformations of $\mathcal{F}$ which are locally free.
\item Given an Azumaya algebra $\mathcal{A}$ on $X$, the deformation functor of deformations of $\mathcal{A}$ as an Azumaya algebra is
\begin{align*}
\mathrm{Def}_{\mathcal{A},\mathrm{Az}}&:\mathsf{Art}_k\rightarrow \mathsf{Set}, & A&\mapsto \left\{\begin{gathered} \text{Isom.\ classes of deformations}\\ \text{$(\mathcal{A}',g)$ of $\mathcal{\mathcal{A}}$ to $X\times_k A$} \end{gathered} \right\}
\end{align*}
also with canonical morphisms.
\item Given either an Azumaya algebra with involution $(\mathcal{A},\sigma)$ on $X$, or an Azumaya algebra with quadratic pair $(\mathcal{A},\sigma,f)$ on $X$, we write $\mathrm{Def}_{(\mathcal{A},\sigma)}$ or $\mathrm{Def}_{(\mathcal{A},\sigma,f)}$ for the functor of deformations of the Azumaya algebra with additional structure defined similarly.
\item Lastly, given a smooth morphism $f:Y\rightarrow X$, the functor of deformations of the smooth morphism $f$ is
\begin{align*}
\mathrm{Def}_{f,sm}&:\mathsf{Art}_k\rightarrow \mathsf{Set}, & A&\mapsto \left\{\begin{gathered} \text{Isom.\ classes of deformations}\\ \text{$((Y',j),f')$ of $f$ with target $X\times_k A$}\end{gathered} \right\}
\end{align*}
with canonical morphisms as well.
\end{enumerate}
\end{defn}

A \textit{tangent-obstruction theory} for a deformation functor $\mathrm{D}:\mathsf{Art}_k\rightarrow \mathsf{Set}$ is a pair of finite dimensional $k$-vector spaces $(T_1,T_2)$ such that for all small extensions $(J\subset C',C)$, i.e.\ for all surjections of local Artinian $k$-algebras $C'\rightarrow C$ with kernel $J\subset C'$ satisfying $J\cdot \mathfrak{m}_{C'}=0$, we have an exact sequence of pointed-sets
\[
T_1\otimes_k J\rightarrow D(C')\rightarrow D(C)\rightarrow T_2\otimes_k J.
\]
For every element $\xi\in D(C)$ admitting a lift $\xi'\in D(C')$, the space $T_1\otimes_k J$ is also required to act transitively on the preimage of $\xi$ in $D(C')$. Moreover, in the case $C=k$, this action is required to give the preimage the structure of a $T_1\otimes_k J$-torsor. This data should be compatible with morphisms of small-extensions in a natural way, cf.\ \cite[Remark 6.1.20]{MR2223408}.

Under certain hypotheses, each of the deformation functors above is equipped with a canonical tangent-obstruction theory. Here are the ones we will need:

\begin{thm}\label{thm: tanobs} Let $X$ be a fixed proper $k$-scheme.
\begin{enumerate}[leftmargin = *]
\item\label{thm: lftot} Let $\mathcal{F}$ be a finite locally free $\mathcal{O}_X$-module. Then a tangent-obstruction theory for $\mathrm{Def}_{\mathcal{F},\mathrm{lf}}$ is given by $(\mathrm{H}^1(X,\cEnd(\mathcal{F})),\mathrm{H}^2(X,\cEnd(\mathcal{F})))$.
\item\label{thm: aztot} Let $\mathcal{A}$ be an Azumaya $\mathcal{O}_X$-algebra. Then $(\mathrm{H}^1(X,\mathcal{A}/\mathcal{O}_X),\mathrm{H}^2(X,\mathcal{A}/\mathcal{O}_X))$ is a tangent-obstruction theory for $\mathrm{Def}_{\mathcal{A},\mathrm{Az}}$.
\item\label{thm: qttot} Let $(\mathcal{A},\sigma,f)$ be a quadratic triple on $X$. Then a tangent-obstruction theory for $\mathrm{Def}_{(\mathcal{A},\sigma,f)}$ is given by $(\mathrm{H}^1(X,\mathfrak{pgo}_{(\mathcal{A},\sigma,f)}),\mathrm{H}^2(X,\mathfrak{pgo}_{(\mathcal{A},\sigma,f)}))$.
\item\label{thm: smtot} Let $Y$ be a projective $k$-scheme and $f:Y\rightarrow X$ a smooth morphism. Then a tangent-obstruction theory for $\mathrm{Def}_{f,sm}$ is given by $(\mathrm{H}^1(Y,\mathcal{T}_{Y/X}),\mathrm{H}^2(Y,\mathcal{T}_{Y/X}))$. Here $\mathcal{T}_{Y/X}$ is the relative tangent bundle of $f$.
\end{enumerate}
\end{thm}

\begin{proof}
We provide a proof for \ref{thm: lftot}, following from \Cref{lem: deflf}. The proofs of \ref{thm: aztot} and \ref{thm: qttot} are the similar but following from \Cref{lem: Azdef} and \Cref{lem: defqp}, respectively. The proof for \ref{thm: smtot} is given in \cite[Theorem 3.4.8]{MR2247603}.

Let $C' \surj C$ be a surjection of local Artinian $k$-algebras with kernel $J$ such that $\frm_{C'}\cdot J = 0$. From \Cref{lem: deflf}, we know that we have an exact sequence of pointed sets
\[
\mathrm{H}^1(X,\cEnd_{\cO_X}(\cF)\otimes_{\cO_X}\cJ) \to D(C') \to D(C) \to \mathrm{H}^2(X,\cEnd_{\cO_X}(\cF)\otimes_{\cO_X} \cJ).
\]
In order to have a tangent obstruction theory as claimed, we need to ``bring out" the ideal $\cJ$. This can be done using the K\"unneth formula.

We have a diagram where both squares are fiber product diagrams
\[
\begin{tikzcd}
X \ar[r,"g'"] \ar[d,"f'"] & X_{C'} \ar[d,"f"] \ar[r,"\pi'"] & X \ar[d,"f'"] \\
\Spec(k) \ar[r,"g"] & \Spec(C') \ar[r,"\pi"] & \Spec(k)
\end{tikzcd}
\]
and both $\pi'\circ g' = \Id_X$ and $\pi\circ g = \Id_k$. The ideal $J\subseteq C'$ is a $C'$--module, and so can be viewed as an $\cO_{\Spec(C')}$--module. By definition, we have $f^*(J) = \cJ\subseteq \cO_{X_{C'}}$. Furthermore, $g^*(J) = J$ viewed as a $k$--vector space with its natural action coming from the fact that $C'/\frm_{C'} \cong k$. Similarly, $g'^*(\cJ) = \cJ$ viewed naturally as an $\cO_X$--module. Because the diagram commutes, we also have $g'^*(\cJ) = f'^*(J)$. Next, we set $\cF' = \pi'^*(\cF)$. We then have
\[
\cEnd_{\cO_X}(\cF) \cong g'^*(\cEnd_{\cO_{X_{C'}}}(\cF'))
\]
by construction. Thus, we are considering cohomology sets of the form
\[
\mathrm{H}^n(X_{C'} \times_{\Spec(C')} \Spec(k), g'^*(\cEnd_{\cO_{X_{C'}}}(\cF')) \otimes_{\cO_X} f'^*(J))
\]
Now, applying the K\"unneth formula and noting that $H^n(\Spec(k),J)=0$ for $n\geq 1$ since $\Spec(k)$ is affine and $J$ is quasi-coherent (even coherent), we find that this cohomology group is isomorphic to
\begin{align*}
&\mathrm{H}^n(X_{C'}, \cEnd_{\cO_{X_{C'}}}(\cF'))\otimes_{C'} \mathrm{H}^0(\Spec(k),J) \\
= &\mathrm{H}^n(X_{C'}, \cEnd_{\cO_{X_{C'}}}(\cF'))\otimes_{C'} J.
\end{align*}
Lastly, since $\pi \colon \Spec(C') \to \Spec(k)$ is flat, we have by \cite[\href{https://stacks.math.columbia.edu/tag/02KH}{Tag 02KH (2)}]{Stacks} that
\[
\mathrm{H}^n(X_{C'}, \cEnd_{\cO_{X_{C'}}}(\cF')) \cong \mathrm{H}^n(X,\cEnd_{\cO_X}(\cF))\otimes_k C'.
\]
Therefore, over all we have that
\begin{align*}
\mathrm{H}^n(X,\cEnd_{\cO_X}(\cF)\otimes_{\cO_X}\cJ) &\cong \mathrm{H}^n(X_{C'}, \cEnd_{\cO_{X_{C'}}}(\cF'))\otimes_{C'} J \\
&\cong \mathrm{H}^n(X,\cEnd_{\cO_X}(\cF))\otimes_k C' \otimes_{C'} J \\
&\cong \mathrm{H}^n(X,\cEnd_{\cO_X}(\cF))\otimes_k J
\end{align*}
as required, finishing the proof.
\end{proof}

\begin{cor}\label{etale_map_deformations}
Let $f:Y\rightarrow X$ be a finite \'etale map with $X$ a projective $k$-scheme. Let $(A,\mathfrak{m}_A)$ be a local Artinian $k$-algebra with residue field $A/\mathfrak{m}_A\cong k$. 

Then, there is a unique (up to isomorphism) deformation $f':Y_A\rightarrow X_A$ of $f$ as a smooth morphism with target $X_A$.
\end{cor}

\begin{proof}
The relative tangent bundle of $f$ is trivial for an \'etale morphism $f$.
\end{proof}

Some natural transformations between deformation functors are compatible with the cohomological realization of their tangent-obstruction theories. The following is an example of such compatibility, for natural transformations induced by forgetting structure, whose proof is immediate from the constructions of \Crefrange{lem: deflf}{lem: defqp}.

\begin{cor}\label{cor: forg}
Let $\mathcal{A}$ be an Azumaya $\mathcal{O}_X$-algebra. Then there is a natural transformation induced by forgetting the algebra structure on $\mathcal{A}$ \begin{equation}
\mathrm{Def}_{\mathcal{A},\mathrm{Az}}\rightarrow \mathrm{Def}_{\mathcal{A},\mathrm{lf}}.\end{equation}
This natural transformation is compatible with the morphism of tangent-obstruction theories induced by the canonical embedding $\mathcal{A}/\mathcal{O}_X\cong \cDer_{\mathcal{O}_X}(\mathcal{A},\mathcal{A})\subset \cEnd(\mathcal{A})$. So, in particular, the obstruction to deforming an Azumaya algebra maps to the obstruction to deforming the underlying locally free sheaf under these maps.

Similarly, if $\mathcal{A}$ is equipped with a quadratic pair $(\sigma,f)$, then there is also a natural transformation induced by forgetting $(\sigma,f)$, \begin{equation}\mathrm{Def}_{(\mathcal{A},\sigma,f)}\rightarrow \mathrm{Def}_{\mathcal{A},\mathrm{Az}}.
\end{equation} This natural transformation is compatible with the morphism of tangent-obstruction theories induced by the embedding $\mathfrak{pgo}_{(\mathcal{A},\sigma,f)}\subset \mathcal{A}/\mathcal{O}_X$. Similarly, the obstruction to deforming a quadratic triple maps to the obstruction to deforming the underlying Azumaya algebra under these maps.$\hfill\square$
\end{cor}

The following are other well-known examples of such compatibility.

\begin{lem}\label{lem: detob}
Let $\mathcal{F}$ be a finite locally free sheaf on $X$. Then there is a natural transformation \[\mathrm{Def}_{\mathcal{F},\mathrm{lf}}\rightarrow \mathrm{Def}_{\det(\mathcal{F}),\mathrm{lf}}\] induced by sending a locally free sheaf to its determinant sheaf. This natural transformation is compatible with the morphism on tangent-obstruction theories induced from the trace $tr:\cEnd(\mathcal{F})\rightarrow \mathcal{O}_X$.

There is also a natural transformation \[\mathrm{Def}_{\mathcal{F},\mathrm{lf}}\rightarrow \mathrm{Def}_{\cEnd(\mathcal{F}),\mathrm{Az}}\] induced by sending a locally free sheaf to its endomorphism algebra. This natural transformation is compatible with the morphism on tangent-obstruction theories induced from the quotient $\cEnd(\mathcal{F})\rightarrow \cEnd(\mathcal{F})/\mathcal{O}_X$ by the inclusion of $\mathcal{O}_X\subset \cEnd(\mathcal{F})$ at the identity.
$\hfill\square$
\end{lem}

Less obvious compatibilities come from transformations which can be thought of as being induced by a reduction of structure group.

\begin{prop}\label{prop: quartvan}
Let $\mathcal{A}$ and $\mathcal{B}$ be two Azumaya algebras on $X$. Consider the deformation functor $\mathrm{D}_{\mathcal{A},\mathcal{B}}:=\mathrm{Def}_{\mathcal{A},Az}\times \mathrm{Def}_{\mathcal{B},Az}$. Let \[0\rightarrow J \rightarrow C'\rightarrow C\rightarrow 0\] be a small extension of local Artinian $k$-algebras $(C',\mathfrak{m}_{C'})$ and $(C,\mathfrak{m}_C)$ each having residue field $k$. Then the following diagram is commutative with exact rows
{\footnotesize
\[
\begin{tikzcd}[column sep = tiny, center picture]
\mathrm{H}^1(X,\mathcal{A}/\mathcal{O}_X \oplus \mathcal{B}/\mathcal{O}_X)\arrow{d}\otimes_k J\arrow{r} &[-1.5ex] \mathrm{D}_{\mathcal{A},\mathcal{B}}(C')\arrow{d}\arrow{r} & \mathrm{D}_{\mathcal{A},\mathcal{B}}(C) \arrow{d}\arrow{r} & \mathrm{H}^2(X,\mathcal{A}/\mathcal{O}_X \oplus \mathcal{B}/\mathcal{O}_X)\otimes_k J\arrow{d}\\
 \mathrm{H}^1(X,\mathcal({A}\otimes\mathcal{B})/\mathcal{O}_X)\otimes_k J \arrow{r} & \mathrm{Def}_{\mathcal{A}\otimes\mathcal{B},Az}(C')\arrow{r} & \mathrm{Def}_{\mathcal{A}\otimes\mathcal{B},Az}(C)\arrow{r} & \mathrm{H}^2(X,(\mathcal{A}\otimes \mathcal{B})/\mathcal{O}_X)\otimes_k J
\end{tikzcd}
\]
}
where the vertical arrows on cohomology are induced by the morphism \[\mathcal{A}/\mathcal{O}_X\oplus \mathcal{B}/\mathcal{O}_X\rightarrow (\mathcal{A}\otimes\mathcal{B})/\mathcal{O}_X\quad \quad (a,b)\mapsto a\otimes 1+ 1\otimes b\] and the vertical arrows of deformation functors are canonical.
\end{prop}

\begin{proof}
On the right, one writes the obstruction to deforming $\mathcal{A}\otimes \mathcal{B}$ as an Azumaya algebra in terms of the obstructions to deforming both $\mathcal{A}$ and $\mathcal{B}$ and uses the fact that $J^2=0$. On the left, set $\pi_{C}:X_{C}\rightarrow X$ and $\pi_{C'}:X_{C'}\rightarrow X$ to be the projections and pick deformations $\mathcal{D}_1$ of $\pi_C^*\mathcal{A}$ and $\mathcal{D}_2$ of $\pi_C^*\mathcal{B}$ to $X_{C'}$. 

Simultaneously split $\mathcal{D}_1,\mathcal{D}_2,\pi_{C'}^*\mathcal{A},\pi_{C'}^*\mathcal{B}$ by an \'etale cover of $X_{C'}$. One gets two 1-cocycles on the left by choosing local isomorphisms, one between $\mathcal{D}_1$ and $\pi_{C'}^*\mathcal{A}$ and another between $\mathcal{D}_2$ and $\pi_{C'}^*\mathcal{B}$, on this splitting \'etale cover. Commutativity of the given map can now be checked directly as giving a local isomorphism between $\mathcal{D}_1\otimes \mathcal{D}_2$ and $\pi_{C'}^*(\mathcal{A}\otimes \mathcal{B})$ on this cover, using the fact that $J^2=0$.
\end{proof}

\begin{cor}\label{cor: quartvan}
Let $k$ be a field of characteristic $2$. Suppose that $X$ is a smooth and projective $k$-surface with $K_X\cong \mathcal{O}_X$. Let $\mathcal{A}$ and $\mathcal{B}$ be two Azumaya $\mathcal{O}_X$-algebras of even degree and suppose that $(\mathcal{A}\otimes \mathcal{B})(X)=k$. Then the canonical map \begin{equation}\label{eq: tensorvanish}
\mathrm{H}^2(X,\mathcal{A}/\mathcal{O}_X\oplus \mathcal{B}/\mathcal{O}_X)\rightarrow \mathrm{H}^2(X,(\mathcal{A}\otimes \mathcal{B})/\mathcal{O}_X)\end{equation} is identically zero.
\end{cor}

\begin{proof}
Since we are assuming $K_X$ is trivial, the map \eqref{eq: tensorvanish} is dual to the map of global sections \begin{equation}\label{eq: redtrtens}((\mathcal{A}\otimes \mathcal{B})/\mathcal{O}_X)^\vee(X)\rightarrow (\mathcal{A}/\mathcal{O}_X)^\vee(X)\oplus (\mathcal{B}/\mathcal{O}_X)^\vee(X)\end{equation} by Serre duality. The surjection $\mathcal{A}\otimes \mathcal{B}\surj (\mathcal{A}\otimes\mathcal{B})/\mathcal{O}_X$ gives an injection on global sections \[((\mathcal{A}\otimes \mathcal{B})/\cO_X)^\vee(X)\subset (\mathcal{A}\otimes \mathcal{B})^\vee(X).\] The reduced trace on the tensor product induces an isomorphism $\mathcal{A}\otimes \mathcal{B}\cong (\mathcal{A}\otimes \mathcal{B})^\vee$. Thus, by our assumption that $(\mathcal{A}\otimes\mathcal{B})(X)=k$, we know $((\cA\otimes\cB)/\cO_X)^\vee(X) \subset k$. Furthermore, the degree of $\cA\otimes\cB$ is also even and so $\cO_X$ lies in the kernel of the reduced trace $\mathrm{trd}_{\cA\otimes\cB}$ since we are in characteristic $2$. Therefore, $((\cA\otimes\cB)/\cO_X)^\vee(X)$ contains a map $T$ induced by the reduced trace. This map is non-zero and therefore generates $((\cA\otimes\cB)/\cO_X)^\vee(X)\cong k$ as a $k$--vector space. 

Now the morphism of \eqref{eq: redtrtens} is defined so that the element $T$ goes to the map $t:\mathcal{A}/\mathcal{O}_X\oplus \mathcal{B}/\mathcal{O}_X\rightarrow \mathcal{O}_X$ defined by \begin{align*}t((a,b))& =\mathrm{trd}_{\mathcal{A}\otimes \mathcal{B}}(a\otimes 1+1\otimes b)\\ &= \mathrm{trd}_{\mathcal{A}}(a)\mathrm{trd}_{\mathcal{B}}(1)+\mathrm{trd}_{\mathcal{A}}(1)\mathrm{trd}_{\mathcal{B}}(b).\end{align*} However, since $2$ divides $\mathrm{gcd}(\mathrm{deg}(\mathcal{A}),\mathrm{deg}(\mathcal{B}))$, we have $\mathrm{trd}_{\mathcal{A}}(1)=\mathrm{trd}_{\mathcal{B}}(1)=0$. Therefore, $T \mapsto 0$ in \eqref{eq: redtrtens} and so the entire map is identically $0$ as claimed.
\end{proof}

We can combine some of the above results with the norm functor from \Cref{norm_functor} to give a description of the deformation functor for some degree $4$ quadratic triples in terms of the obstructions for quaternion algebras.

\begin{lem}\label{deformations_through_stack_morphism}
Let $X$ be a fixed $k$--scheme and let $\cQ_1$ and $\cQ_2$ be two quaternion $\cO_X$--algebras, i.e.\ $\mathrm{deg}(\cQ_1)=\mathrm{deg}(\cQ_2)=2$. Then there is an isomorphism of deformation functors
\begin{equation}\label{lem: normdef}
\Def_{\cQ_1,\Az}\times \Def_{\cQ_2,\Az} \iso \Def_{(\cQ_1\otimes_{\cO_X}\cQ_2,\sigma_N,f_N)} \\
\end{equation}
given by the norm functor of \Cref{norm_functor_equiv}. In particular, for an Artinian $k$-algebra $(A,\frm_A)$ of $\Art_k$, the isomorphism is given by
\begin{align*}
\Def_{\cQ_1,\Az}(A)\times \Def_{\cQ_2,\Az}(A) &\iso \Def_{(\cQ_1\otimes_{\cO_X}\cQ_2,\sigma_N,f_N)}(A) \\
([\cQ_1'],[\cQ_2']) &\mapsto [(\cQ_1'\otimes_{\cO_{X_A}}\cQ_2',\psi_1'\otimes\psi_2',f_N')]
\end{align*}
where $\psi_i'$ is the canonical symplectic involution on $\cQ_i'$ and $f_N'$ is the semi-trace produced by the norm functor.
\end{lem}
\begin{proof} Let $(A,\frm_A)$ be an object from $\Art_k$ and consider two deformations $\cQ_1'$ and $\cQ_2'$ on $X_A = X\times_k A$ of $\cQ_1$ and $\cQ_2$ respectively. Let $\cQ'$ be the associated quaternion algebra on $X_A \sqcup X_A$. Because the norm is a morphism of stacks and thus respects base change, it is clear that $N(X_A\sqcup X_A \to X_A,\cQ')$ on $X_A$ restricts down to
\[
N(X\sqcup X \to X,\cQ'|_X) = (\cQ_1\otimes_{\cO_X}\cQ_2,\psi_1\otimes \psi_2,f_N).
\]
Hence
\[
N(X_A\sqcup X_A \to X_A,\cQ') = (\cQ_1'\otimes_{\cO_{X_A}}\cQ_2',\psi_1'\otimes\psi_2',f_N')
\]
is an element of $\Def_{(\cQ_1\otimes_{\cO_X}\cQ_2,\sigma_N,f_N)}(A)$.

The natural transformation \eqref{lem: normdef} is injective. Indeed, because the norm is an equivalence of categories, in particular it is full, any isomorphism between the resulting quadratic pairs would correspond to an isomorphism between the starting quaternion algebras. Thus our proposed map is injective on isomorphism classes. 

To see that \eqref{lem: normdef} is surjective, we use the essential surjectivity of the norm. Given any deformation $(\cA,\sigma,f)$ on $X_A$ of $(\cQ_1\otimes_{\cO_X}\cQ_2,\psi_1\otimes \psi_2,f_N)$, the object $(X_A,(\cA,\sigma,f))$ in $\fD_2(X_A)$ must be isomorphic to the norm of some $(T \to X_A,\cB)$ from $\fA_1^2(X_A)$, i.e.\ $T \to X_A$ is some degree $2$ \'etale cover and $\cB$ is a quaternion algebra on $T$. Since our equivalence of stacks respects base change, the pullback of $(T \to X_A,\cB)$ to $X$ must be isomorphic to $(X\sqcup X \to X,\cQ)$. 

In particular, this means that the morphism $T \to X_A$ is a deformation of the split degree $2$ \'etale cover $X\sqcup X \to X$. However, by \Cref{etale_map_deformations} such deformations are unique up to isomorphism and clearly $X_A\sqcup X_A \to X_A$ is such a deformation. Therefore, $\cB$ is a quaternion algebra on $X_A\sqcup X_A$ and thus corresponds to two quaternion algebras $\cB_1$ and $\cB_2$ on $X_A$. Lastly, to respect base change these must each be deformations of the algebras $\cQ_1$ and $\cQ_2$ respectively. Hence, we have that
\[
(X_A,(\cA,\sigma,f)) \cong N(T\to X_A,\cB) \cong (X_A,(\cB_1\otimes_{\cO_{X_A}}\cB_2,\psi_{\cB_1}\otimes\psi_{\cB_2},f_\cB))
\]
meaning that $(\cA,\sigma,f)$ is (up to isomorphism) of the correct form, and thus our map on deformation functors is surjective. This completes the proof.
\end{proof}

\subsection{Igusa surfaces}\label{Igusa_surfaces}
Our main example in this paper exists on an Igusa surface. An Igusa surface is the quotient of a product of two elliptic curves, at least one of which is ordinary, over a field of characteristic $2$ by a certain free action of $\mathbb{Z}/2\mathbb{Z}$. Throughout this subsection we work over a fixed algebraically closed base field $k$ of characteristic $2$. The assumption that $k$ is algebraically closed is not strictly necessary, but it simplifies some of the discussion.

Let $E_1$ be an elliptic curve over $k$. Let $E_2$ be an ordinary elliptic curve over $k$, i.e.,\ the $2$-torsion subscheme $E_2[2]$ of $E_2$ admits an isomorphism $E_2[2]\cong \mu_2\times \mathbb{Z}/2\mathbb{Z}$. Let $Y=E_1\times_k E_2$ and consider the action of $\mathbb{Z}/2\mathbb{Z}$ on $Y$ defined by sending a point $(a,b)$ to $(-a,b+t)$ where $t\in E_2[2](k)$ is the unique nontrivial $2$ torsion $k$-point. This is a free action of $\mathbb{Z}/2\mathbb{Z}$ on $Y$, so there is a smooth and projective quotient $\pi:Y\rightarrow X$. The surface $X$ is called an \textit{Igusa surface}, named after \cite{MR74085}.

\subsubsection{Cohomological invariants}
The cotangent bundle of the Igusa surface $X$ admits an isomorphism $\Omega_X^1\cong \mathcal{O}_X\oplus \mathcal{O}_X$. Actually, any global section of the cotangent bundle $\Omega_{Y}^1\cong \Omega_{E_1}^1\boxtimes \Omega_{E_2}^1$ is invariant for the action of $\mathbb{Z}/2\mathbb{Z}$ and so descends to a global section of $\Omega_X^1$. It follows that $K_X\cong \mathcal{O}_X$.

The Hodge numbers of $X$ are easy to determine, cf.\ \cite[Example 6, p.\ 22]{MR217090}. They are:
\[ h^{0,0}=h^{2,2}=1,\quad h^{0,1}=h^{1,0}=h^{1,2}=h^{2,1}=2, \quad h^{0,2}=h^{2,0}=1,\quad h^{1,1}=4.\]
Of these, we point out $h^{0,1}=\dim_k\mathrm{H}^1(X,\mathcal{O}_X)=2$ and $h^{0,2}=\dim_k\mathrm{H}^2(X,\mathcal{O}_X)=1$.

\subsubsection{Nonreduced Picard scheme}
Since $X$ is smooth, irreducible, projective, and defined over an algebraically closed field $k$, the Picard scheme $\mathbf{Pic}_{X/k}$ exists and represents the relative Picard functor \cite[Corollary 4.18.3]{MR3287693}. The Picard scheme $\mathbf{Pic}_{X/k}$ is locally of finite type and the connected component at the identity $\mathbf{Pic}_{X/k}^0$ is an open and closed projective subgroup scheme \cite[Proposition 5.3]{MR3287693}. 

In \cite{MR74085}, Igusa shows that $\mathbf{Pic}^0_{X/k}$ is one dimensional; since $\dim \mathrm{H}^1(X,\mathcal{O}_X)=2$ it follows that $\mathbf{Pic}^0_{X/k}$ is not reduced. For an explicit determination of the Picard scheme of $X$ we refer to \cite[Example 2.2]{MR512270}. The result depends on the curve $E_1$. When $E_1$ is ordinary, one has that $\mathbf{Pic}^0_{X/k}\cong \mu_2\times E_2/\langle t \rangle$.

As a consequence of the nonreduced structure of the Picard scheme of $X$, we get a nonvanishing statement on obstructions to deforming certain line bundles on $X$.

\begin{prop}\label{prop: smallext}
Let $X$ be an Igusa surface over an algebraically closed field $k$ of characteristic $2$. Then for any line bundle $\mathcal{L}$ on $X$, there exists a small extension \[0\rightarrow J\rightarrow C'\rightarrow C\rightarrow 0,\] of local Artinian $k$-algebras $(C',\mathfrak{m}_{C'})$ and $(C,\mathfrak{m}_C)$ both with residue field $k$, and a deformation $\mathcal{L}'$ of $\mathcal{L}$ to $X_C$ so that $ob(\mathcal{L}'; X_{C'}/X_C)\neq 0$.
\end{prop}

\begin{proof}
Since $\mathbf{Pic}_{X/k}$ exists and represents the relative Picard functor $\mathrm{Pic}_{X/k}$, we have an inclusion of subfunctors $\mathrm{Def}_{\mathcal{L},\mathrm{lf}}\subset \mathrm{Pic}_{X/k}$ with the latter restricted to $\mathsf{Art}_k$. All of this is to say that $\mathrm{Def}_{\mathcal{L},\mathrm{lf}}$ is pro-representable by the completion $\hat{R}$ of the local ring $R$ of $\mathbf{Pic}_{X/k}$ at the point $[\mathcal{L}]\in\mathbf{Pic}_{X/k}(k)$.

By \cite[Corollary 6.2.5]{MR2223408}, for any tangent-obstruction theory $(T_1,T_2)$ for $\mathrm{Def}_{\mathcal{L},\mathrm{lf}}$ there is an inequality on the Krull dimension of $\hat{R}$ as follows \[1=\mathrm{dim}(\hat{R})\geq \dim_k T_1-\dim_k T_2.\] Since the space $T_1$ is uniquely determined by $\mathrm{Def}_{\mathcal{L},\mathrm{lf}}$ by \cite[Proposition 6.1.23]{MR2223408}, and since $\mathrm{dim}_k(T_1)=2$, there is no tangent-obstruction theory with $T_2$ a proper subset of $\mathrm{H}^2(X,\mathcal{O}_X)$. The claim of the proposition follows.
\end{proof}

\begin{rem}\label{rem: explicitob}
Here is an explicit example of the setting that is stated to exist in \Cref{prop: smallext} for any particular line bundle $\mathcal{L}$ on $X$. Suppose that $Y=E_1\times E_2$ with both $E_1$ and $E_2$ ordinary so that $\mathbf{Pic}_{X/k}^0\cong \mu_2\times E_2/\langle t\rangle$ for the unique nontrivial point $t\in E_2[2](k)$. Note that the point $s\in \mathbf{Pic}_{X/k}(k)$ which represents $\mathcal{L}$ defines an isomorphism between the connected component $\mathcal{P}_s\subset \mathbf{Pic}_{X/k}$ containing $s$ and the connected component of the identity $\mathbf{Pic}^0_{X/k}$.

We may then view $s\in \mathbf{Pic}_{X/k}(k)$ as a map $s\colon \Spec(k) \to \mu_2\times E_2/\langle t \rangle$ which, since $\Spec(k)$ is reduced and $\mu_2\cong \mathrm{Spec}(k[x]/(x-1)^2)$ is non-reduced in characteristic $2$, factors as
\[
s\colon \Spec(k) \xrightarrow{\tilde{s}} E_2/\langle t \rangle \to \mu_2 \times E_2/\langle t \rangle
\]
for a map $\tilde{s}:\mathrm{Spec}(k)\rightarrow E_2/\langle t\rangle$. Consider the $\mu_2$-point
\[
s_2:\mu_2\cong \mu_2\times_k \Spec(k) \xrightarrow{\Id \times \tilde{s}} \mu_2\times E_2/\langle t\rangle.
\]
If we identify $s_2\in \mathbf{Pic}_{X/k}(\mu_2)=\mathrm{Pic}(X\times_k \mu_2)$ with the corresponding line bundle $\mathcal{L}_2$ that it defines, then we claim that $\mathcal{L}_2$ has obstructed deformations to the thickening $\tau_3=\mathrm{Spec}(k[x]/(x-1)^3)$ under the inclusion $i:\mu_2\subset \tau_3$. 

Indeed, any $\tau_3$-point $s_3\in\mathbf{Pic}_{X/k}(\tau_3)=\mathrm{Pic}(X\times_k \tau_3)$ lifting the point $s_2$ induces an endomorphism (on projection to $\mu_2$) \[\mu_2\xrightarrow{i} \tau_3\rightarrow \mu_2\] which is required to be the identity. But, of course, this is impossible and so the obstruction $ob(\mathcal{L}_2;X_{\tau_3}/X_{\mu_2})\neq 0$ is nonvanishing.

More generally, since $\mu_2\cong k[\epsilon]$ in characteristic $2$, any point $s'\in \mathbf{Pic}_{X/k}(\mu_2)$ lifting $s$ naturally defines a vector in the tangent space \[\mathrm{H}^1(X,\mathcal{O}_X)\cong \mathrm{T}_s\mathbf{Pic}_{X/k}\cong \mathrm{T}_{x}\mu_2\oplus \mathrm{T}_{y} E_2/\langle t\rangle\] where $s=(x,y)$. The same reasoning shows that any line bundle on $X_{\mu_2}$ defined by a $\mu_2$-point corresponding to a vector with nonzero $T_x\mu_2$ component must therefore also have obstructed deformations to $\tau_3$.
\end{rem}

\section{An intermediate obstruction and Lie algebra sheaves}\label{intermediate_obstruction}
 Fix a scheme $X$. Let $\mathcal{A}$ be an Azumaya $\mathcal{O}_X$-algebra. Assume that there exists a locally quadratic orthogonal involution $\sigma$ on $\mathcal{A}$. In \cite{GNR2023AzumayaObstructions}, the authors introduce \textit{strong} and \textit{weak} obstructions for the existence of a semi-trace $f$ making $(\sigma,f)$ a quadratic pair on $\mathcal{A}$. In this section, we reinterpret these obstructions and introduce a new obstruction, the \textit{intermediate obstruction}, which exists only when $X$ is purely of characteristic $2$.

\subsection{Extensions and their associated obstructions} The starting point for our observations is the extension \begin{equation}\label{eq: skewsymd} 0\rightarrow \cSkew_{\mathcal{A},\sigma}\rightarrow \mathcal{A}\xrightarrow{1+\sigma} \cSymd_{\mathcal{A},\sigma}\rightarrow 0.\end{equation} The extension \eqref{eq: skewsymd} defines a canonical extension class in $\mathrm{Ext}^1(\cSymd_{\mathcal{A},\sigma},\cSkew_{\mathcal{A},\sigma})$. We will show that this class is responsible for both the strong and weak obstructions introduced in \cite{GNR2023AzumayaObstructions}.

\begin{prop}\label{prop: strong}
For any scheme $X$ and for any Azumaya algebra $(\mathcal{A},\sigma)$ with locally quadratic orthogonal involution on $X$, there exists a canonical locally free subsheaf $\mathcal{E}\subset \mathcal{A}$ with the following properties:
\begin{enumerate}
\item\label{it: quo} $\mathcal{A}/\mathcal{E}\cong \cSymd_{\mathcal{A},\sigma}/\mathcal{O}_X$;
\item\label{it: ext} $\mathcal{E}$ fits into an extension \[
0\rightarrow \cSkew_{\mathcal{A},\sigma}\rightarrow \mathcal{E}\rightarrow \mathcal{O}_X\rightarrow 0;\]
\item\label{it: strong} the class $\alpha(\mathcal{E})\in \mathrm{Ext}^1(\mathcal{O}_X, \cSkew_{\mathcal{A},\sigma})$ determined by the extension above is mapped to the strong obstruction $\Omega(\mathcal{A},\sigma)$ under the canonical isomorphism $ \mathrm{Ext}^1(\mathcal{O}_X, \cSkew_{\mathcal{A},\sigma})\cong \mathrm{H}^1(X,\cSkew_{\mathcal{A},\sigma})$.
\end{enumerate}
\end{prop}

\begin{proof}
Since $\sigma$ is locally quadratic, there exists a morphism $i:\mathcal{O}_X\rightarrow \cSymd_{\mathcal{A},\sigma}$ defined by sending $1\mapsto 1_{\cA}$. Let $\mathcal{E}$ be the fiber product sheaf associated to the pair of morphisms $(i,1+\sigma)$ where $1+\sigma:\mathcal{A}\rightarrow \cSymd_{\mathcal{A},\sigma}$ is the induced quotient. Then, there is a commutative diagram with exact rows as below
\[
\begin{tikzcd}
0\arrow{r} & \cSkew_{\mathcal{A},\sigma}\arrow[equals]{d}\arrow{r} & \mathcal{E} \arrow{d}\arrow{r} & \mathcal{O}_X\arrow["i"]{d}\arrow{r} & 0\\
0\arrow{r} & \cSkew_{\mathcal{A},\sigma}\arrow{r} & \mathcal{A}\arrow["1+\sigma"]{r} & \cSymd_{\mathcal{A},\sigma}\arrow{r} & 0.
\end{tikzcd}
\]
Property \ref{it: quo} now follows from the snake lemma and property \ref{it: ext} is clear. To check property \ref{it: strong}, one applies $\mathrm{Hom}(\mathcal{O}_X,-)$ and uses the fundamental constructions of homological algebra.
\end{proof}

Similarly, passing to the quotient $\cSkew_{\mathcal{A},\sigma}/\cAlt_{\mathcal{A},\sigma}$ in the second component of $\mathrm{Ext}^1(\mathcal{O}_X,\cSkew_{\mathcal{A},\sigma})$ determines the weak obstruction $\omega(\mathcal{A},\sigma)$.

\subsection{An intermediate obstruction in characteristic 2} Assume now that $k$ is a field of characteristic $2$ and that $X$ is a scheme over $k$. In this case, we have $\cO_X \subset \cSkew_{\cA,\sigma} = \cSym_{\cA,\sigma}$ and that $1\in \cSymd_{\mathcal{A},\sigma}(X)=\cAlt_{\mathcal{A},\sigma}(X)$ since $\sigma$ is locally quadratic. This allows us to consider an intermediary quotient \[\cSkew_{\mathcal{A},\sigma}\twoheadrightarrow \cSkew_{\mathcal{A},\sigma}/\mathcal{O}_X\twoheadrightarrow \cSkew_{\mathcal{A},\sigma}/\cAlt_{\mathcal{A},\sigma}.\] Correspondingly, we introduce:

\begin{defn}
Consider the short exact sequence of sheaves \[0\rightarrow \cSkew_{\mathcal{A},\sigma}/\mathcal{O}_X\rightarrow \mathcal{A}/\mathcal{O}_X\xrightarrow{\Id +\sigma} \cSymd_{\mathcal{A},\sigma}\rightarrow 0.\] This gives the following portion of the long exact cohomology sequence \[\mathcal{A}/\mathcal{O}_X(X)\rightarrow \cSymd_{\mathcal{A},\sigma}(X)\xrightarrow{\delta} \mathrm{H}^1(X,\cSkew_{\mathcal{A},\sigma}/\mathcal{O}_X).\] We call $\Omega'(\mathcal{A},\sigma)=\delta(1)$ the intermediate obstruction associated to the pair $(\mathcal{A},\sigma)$.
\end{defn}

In a similar spirit to \Cref{prop: strong}, we have the following characterization of the intermediate obstruction:
\begin{lem}\label{lem: inter}
Fix a base field $k$ of characteristic $2$. Let $X$ be a scheme over $k$. Suppose that $(\mathcal{A},\sigma)$ is an Azumaya algebra with locally quadratic orthogonal involution over $X$. Then there exists a canonical locally free subsheaf $\mathcal{E}'\subset \mathcal{A}/\mathcal{O}_X$ with the following properties: 

\begin{enumerate}
\item $\mathcal{E}'\cong \mathcal{E}/\mathcal{O}_X$;
\item $(\mathcal{A}/\mathcal{O}_X)/\mathcal{E}'\cong \cAlt_{\mathcal{A},\sigma}/\mathcal{O}_X$;
\item $\mathcal{E}'$ fits into a canonical extension \[
0\rightarrow \cSkew_{\mathcal{A},\sigma}/\mathcal{O}_X\rightarrow \mathcal{E}'\rightarrow \mathcal{O}_X\rightarrow 0;\]
\item the class $\alpha(\mathcal{E}')\in \mathrm{Ext}^1(\mathcal{O}_X, \cSkew_{\mathcal{A},\sigma}/\mathcal{O}_X)$ determined by the extension above is mapped to the intermediate obstruction $\Omega'(\mathcal{A},\sigma)$ under the isomorphism $ \mathrm{Ext}^1(\mathcal{O}_X,\cSkew_{\mathcal{A},\sigma}/\mathcal{O}_X)\cong \mathrm{H}^1(X,\cSkew_{\mathcal{A},\sigma}/\mathcal{O}_X)$.
\end{enumerate}
\end{lem}

\begin{proof}
The proof is identical to that of \Cref{prop: strong}, but one starts with the exact sequence \[0\rightarrow \cSkew_{\mathcal{A},\sigma}/\mathcal{O}_X\rightarrow \mathcal{A}/\mathcal{O}_X\xrightarrow{1+\sigma} \cAlt_{\mathcal{A},\sigma}\rightarrow 0\] and one sets $\mathcal{E}'$ to be the fiber product sheaf associated to the pair $(i,1+\sigma)$ where $i: \mathcal{O}_X\rightarrow \cSymd_{\mathcal{A},\sigma}=\cAlt_{\mathcal{A},\sigma}$ is the canonical inclusion.
\end{proof}

\begin{cor}\label{cor: quaternion}
Fix a base field $k$ of characteristic $2$. Let $X$ be a scheme over $k$. Let $(\mathcal{A},\sigma,f)$ be a quadratic triple over $X$ and assume that $\mathrm{deg}(\mathcal{A})=2$. Then there are isomorphisms of vector bundles \[\mathcal{A}\cong \mathcal{E}\oplus \cSkew_{\mathcal{A},\sigma}/\mathcal{O}_X\quad \mbox{and} \quad \cAlt_{\mathcal{A},\sigma}\cong \mathcal{O}_X\] for some vector bundle $\mathcal{E}$ determined by an extension class $\alpha(\mathcal{E})\in\mathrm{Ext}^1(\mathcal{O}_X,\mathcal{O}_X)$. Moreover, $\alpha(\mathcal{E})=0$ if and only if $\Omega(\mathcal{A},\sigma)=0$.
\end{cor}

\begin{proof}
Since $\sigma$ is a locally quadratic involution, there is a morphism $\mathcal{O}_X\rightarrow \cAlt_{\mathcal{A},\sigma}$ which one can check is an isomorphism after splitting $(\mathcal{A},\sigma,f)$. Since we assume that $(\mathcal{A},\sigma)$ is equipped with a semi-trace $f$, there is a section $\ell\in \mathcal{A}/\mathcal{O}_X(X)$ with $\sigma(\ell)+\ell=1$ and $f=\mathrm{trd}_{\mathcal{A}}(\ell\cdot -)$ locally. The exact sequence \[0\rightarrow \cSkew_{\mathcal{A},\sigma}/\mathcal{O}_X\rightarrow \mathcal{A}/\mathcal{O}_X\xrightarrow{\mathrm{Id}+\sigma} \mathcal{O}_X\rightarrow 0\] is then right-split by the map $\mathcal{O}_X\rightarrow \mathcal{A}/\mathcal{O}_X$ gotten from the section $\ell\in \mathcal{A}/\mathcal{O}_X(X)$. By \Cref{lem: inter}, this is exactly the claim that the intermediate obstruction is trivial (under the assumption $\mathrm{deg}(\mathcal{A})=2$, the intermediate and weak obstructions agree).

Now consider the canonical sequence \begin{equation}\label{eq: can} 0\rightarrow \cSkew_{\mathcal{A},\sigma}\rightarrow \mathcal{A}\rightarrow \cAlt_{\mathcal{A},\sigma}\rightarrow 0.\end{equation} This extension determines the strong obstruction $\Omega(\mathcal{A},\sigma)$ by \Cref{prop: strong} (here we have $\mathcal{A}=\mathcal{E}$). Since we are assuming $\mathrm{deg}(\mathcal{A})=2$, the exact sequence \[0\rightarrow \mathcal{O}_X\rightarrow \cSkew_{\mathcal{A},\sigma}\rightarrow \cSkew_{\mathcal{A},\sigma}/\mathcal{O}_X\rightarrow 0\] is left-split by the semi-trace $f$, cf.\ \cite[Corollary 3.8]{GNR2023AzumayaObstructions}. Thus, the extension class associated to the exact sequence \eqref{eq: can} decomposes $\alpha(\mathcal{A})=(a,b)$ as an element of the group \[\mathrm{Ext}^1(\mathcal{O}_X,\mathcal{O}_X\oplus \cSkew_{\mathcal{A},\sigma}/\mathcal{O}_X)=\mathrm{Ext}^1(\mathcal{O}_X,\mathcal{O}_X)\oplus\mathrm{Ext}^1(\mathcal{O}_X,\cSkew_{\mathcal{A},\sigma}/\mathcal{O}_X).\] However, by the above, we have $b=0$. This gives the direct sum decomposition $\mathcal{A}\cong \mathcal{E}\oplus \cSkew_{\mathcal{A},\sigma}/\mathcal{O}_X$ as described in the corollary statement.
\end{proof}

\begin{example}\label{exmp: e2q} Let $k$ be a field of characteristic $2$. Let $E$ be an elliptic curve over $k$ such that $E[2]\cong \mu_2\times \mathbb{Z}/2\mathbb{Z}$ where $E[2]\subset E$ is the $2$-torsion subgroup scheme. Gille, Neher, and the second author of this paper have constructed a quaternion Azumaya algebra $\mathcal{Q}$ over $E$ with a locally quadratic orthogonal involution $\sigma$ such that $\Omega(\mathcal{Q},\sigma)\neq 0$ but $\omega(\mathcal{Q},\sigma)=0$, see \cite[Example 6.1]{GNR2023AzumayaObstructions}. The nonvanishing of the strong obstruction $\Omega(\mathcal{Q},\sigma)\neq 0$ is verified directly, making use of the fact that $\mathcal{Q}(E)=k$. In this example, we show how one can deduce this nonvanishing of the strong obstruction by determining explicitly the underlying module structure on $\mathcal{Q}$ and appealing to \Cref{cor: quaternion}.

Write $E'$ for another copy of $E$ and let $E'\rightarrow E$ be the multiplication-by-2 endomorphism. The Azumaya algebra $\mathcal{Q}$ is gotten by gluing two copies of $\Mat_2(\mathcal{O}_{E'})$ over $E'\times_E E'$ using the isomorphism
\[
\phi=\Inn\left(\begin{bmatrix} 0 & 1 \\ x & 0 \end{bmatrix}, \begin{bmatrix} 1 & 0 \\ 0 & x \end{bmatrix} \right)\in \PGL_2(E'\times_E E')
\]
where $x^2=1$ in $\mathcal{O}_{E'\times_E E'}(E'\times_E E')\cong k[x]/(x^2-1)\otimes (k\times k) \cong \left(k[x]/(x^2-1)\right)^2$.

The Azumaya algebra $\mathcal{Q}\otimes \mathcal{Q}\cong \cEnd(\mathcal{Q})$ can then be described by a similar gluing procedure along the same fppf-cover using the isomorphism \[\phi\otimes \phi=\Inn\left(\begin{bmatrix} 0 & 0 & 0 & 1 \\ 0 & 0 & x & 0\\ 0 & x & 0 & 0 \\ 1 & 0 & 0 & 0  \end{bmatrix}, \begin{bmatrix} 1 & 0 & 0 & 0 \\ 0 & x & 0 & 0\\ 0 & 0 & x & 0 \\ 0 & 0 & 0 & 1  \end{bmatrix}\right)\in \PGL_4(E'\times_E E').\] A simple calculation using the equalizer sequence defining the sheaf $\mathcal{Q}\otimes \mathcal{Q}$ shows that the global sections of $\mathcal{Q}\otimes \mathcal{Q}$ can be identified with \[(\mathcal{Q}\otimes \mathcal{Q})(E)=\left\{\begin{bmatrix} a & 0 & 0 & d \\ 0 & b & c & 0 \\ 0 & c & b & 0 \\ d & 0 & 0 & a \end{bmatrix}: a,b,c,d\in k\right\}\subset \Mat_2(\mathcal{O}_{E'}).\] One can then check directly that $\cEnd(\mathcal{Q})$ contains $4$ idempotent sections, namely $0$, the identity $1$, and the idempotents
\[
T=\begin{bmatrix}0 & 0 & 0 & 0\\ 0 & 1 & 0 & 0\\ 0 & 0 & 1 & 0\\ 0 & 0 & 0 & 0 \end{bmatrix} \quad \mbox{and}\quad 1-T=\begin{bmatrix}1 & 0 & 0 & 0\\ 0 & 0 & 0 & 0\\ 0 & 0 & 0 & 0\\ 0 & 0 & 0 & 1 \end{bmatrix}.
\]
It follows that $\mathcal{Q}\cong \mathcal{E}\oplus \mathcal{F}$ can be written as a sum of two indecomposable vector bundles $\mathcal{E}$ and $\mathcal{F}$ each of rank $2$. Because of the Krull-Schmidt theorem, it follows immediately from \Cref{cor: quaternion} that $\Omega(\mathcal{Q},\sigma)\neq 0$.

We can say more about the module structure of $\mathcal{Q}$. By the considerations above, we may assume $\mathcal{E}$ is the unique nontrivial extension of $\mathcal{O}_E$ by $\mathcal{O}_E$. Now, $\mathcal{Q}\cong \mathcal{Q}^\vee$ by the reduced trace, so $\mathrm{deg}(Q)=0$. Since $\mathrm{deg}(\mathcal{E})=0$ too, the degree of $\mathcal{F}$ must also be zero. From the classification of vector bundles on $E$, we must have $\mathcal{F}\cong \mathcal{E}\otimes \mathcal{L}$ for a line bundle $\mathcal{L}$ on $E$ with $\mathrm{deg}(\mathcal{L})=0$ also. But under the isomorphism $\mathcal{Q}\cong \mathcal{Q}^\vee$ we have $\mathcal{F}\cong \mathcal{F}^\vee$ from which it follows that $\mathcal{L}\cong \mathcal{L}^\vee$. As $\mathcal{Q}(E)=k$, we must have $\mathcal{L}\neq \mathcal{O}_E$ and, by our assumption that $E$ is ordinary, there is a unique such $\mathcal{L}$. Altogether this implies that there is an isomorphism of $\mathcal{O}_E$-modules $\mathcal{Q}\cong \mathcal{E}\oplus (\mathcal{E}\otimes \mathcal{L})$ where $\mathcal{E}$ is the unique nontrivial extension of $\mathcal{O}_E$ by itself and where $\mathcal{L}\neq \mathcal{O}_E$ is the unique line bundle such that $\mathcal{L}^{\otimes 2}\cong \mathcal{O}_E$.

One might wonder if $\mathcal{Q}$ is a neutral Azumaya algebra, i.e.\ $\mathcal{Q}\cong \cEnd(\mathcal{V})$ for some vector bundle $\mathcal{V}$. Under the assumption that $k$ is algebraically closed, this is the case since $\mathrm{Br}(E)=0$ by Tsen's theorem. We now give a description of such a $\cV$ for which $\mathcal{Q}\cong \cEnd(\mathcal{V})$ which will suffice for our purposes later.

Clearly such a bundle $\mathcal{V}$ is necessarily indecomposable, since $\mathcal{Q}$ is a direct sum of two indecomposable vector bundles. Up to replacing $\mathcal{V}$ with $\mathcal{V}\otimes \mathcal{L}$ for some line bundle $\mathcal{L}$, we may assume that $\deg(\mathcal{V})=0$ or $\deg(\mathcal{V})=1$. Up to replacing $\mathcal{V}$ again, we can assume that $\mathcal{V}$ is the unique extension \[0\rightarrow \mathcal{O}_E\rightarrow \mathcal{V}\rightarrow \mathcal{L}\rightarrow 0\] for either $\mathcal{L}=\mathcal{O}_E$ or $\mathcal{L}=\mathcal{O}_E(p)$ where $p\in E(k)$ is the origin, cf.\ \cite[Ch.\ V, Theorem 2.15]{MR2583634}. However, the former case necessarily does not occur since there is an isomorphism $\mathcal{E}\otimes \mathcal{E}\cong \mathcal{E}\oplus \mathcal{E}$ for $\mathcal{E}$ the unique nonsplit extension of $\mathcal{O}_E$ by itself, see \cite[Proposition 2.8]{MR318151}. Hence, $\mathcal{V}$ is isomorphic to the unique nonsplit extension of $\mathcal{O}_E$ and $\mathcal{O}_E(p)$ (cf.\ \cite[Corollary 2.7]{MR318151}). 

The first author thanks Igor Dolgachev for pointing out that the Severi--Brauer scheme associated to $\mathcal{Q}$ should be the symmetric square $S^2E$ over $E$. This direction led to the above description of $\mathcal{V}$.\end{example}

\begin{rem}\label{rem: torsquat}
Let $k$ be a field of characteristic $2$ with $\#k\geq q=2^r$ for some $r\geq 1$. Let $\Gamma_q=\mathbf{PGL}_2(\mathbb{F}_q)$ be the constant group scheme over $k$. Let $X_q'\rightarrow X_q$ be a $\Gamma_q$-torsor with both $X_q'$ and $X_q$ connected projective $k$-varieties. In \cite[Example 6.4]{GNR2023AzumayaObstructions}, Gille, Neher, and the second named author construct (for $q=4$, but the process generalizes in a straightforward way to all $q$) another Azumaya algebra $\mathcal{Q}_q$ on $X_q$ by gluing two copies of $\Mat_2(\mathcal{O}_{X_q'})$ along the cocycle
\[
\gamma=\left(g\right)_{g\in \Gamma_q}\in \PGL_2(X_q'\times_{X_q} X_q').
\]
Here, we have that $X'_q \times_{X_q} X'_q \cong X'_q \times_k \Gamma_q$ and so $\cO_{X_q'\times_{X_q} X_q'}(X_q'\times_{X_q} X_q') \cong k^{|\Gamma_q|}$. Therefore, 
\[
\PGL_2(X_q'\times_{X_q} X_q') = \PGL_2(k)^{|\Gamma_q|}
\]
where we index the factors with the elements of $\Gamma_q$, thus writing $\gamma = \left(g\right)_{g\in \Gamma_q}$ makes sense. In \cite{GNR2023AzumayaObstructions}, they show that $\mathcal{Q}_q$ admits a locally quadratic orthogonal involution $\sigma_q$ and that, when $4\mid q$, we have $\omega(\mathcal{Q}_q,\sigma_q)\neq 0$. Their proof of the nonvanishing of the weak invariant relies crucially on the existence of an element $\lambda \neq 0,1 \in \mathbb{F}_4$ and their same argument does not apply to show $\omega(\mathcal{Q}_2,\sigma_2)\neq 0$.

Applying the results of this section, we find a new proof that $\omega(\mathcal{Q}_q,\sigma_q)\neq 0$ for all powers $q=2^r$ with $r\geq 2$. To start, note that the Azumaya algebra $\mathcal{Q}_q\otimes \mathcal{Q}_q\cong \cEnd(\mathcal{Q}_q)$ on $X_q$ is described by the cocycle \[\gamma\otimes \gamma=(g\otimes g)_{g\in \Gamma_q}\in \PGL_4(X_q'\times_{X_q} X_q').\]
A calculation with the sheaf exact sequence shows that when $4\mid q$ \[(\mathcal{Q}_q\otimes \mathcal{Q}_q)(X_q)=\left\{ \begin{bmatrix} a+b & 0 & 0 & 0 \\ 0 & b & a & 0 \\ 0 & a & b & 0 \\ 0 & 0 & 0 & a+b\end{bmatrix}: a,b\in k\right\}\subset \Mat_2(\mathcal{O}_{X_q'})(X_q').\] Since there are only the two trivial idempotents $0$ and $1$ in this set, it follows that $\mathcal{Q}_q$ is indecomposable as an $\mathcal{O}_{X_q}$-module. If there was a map $f_q$ making $(\sigma_q,f_q)$ a quadratic pair on $\mathcal{Q}_q$ then, by \Cref{cor: quaternion}, we would find $\mathcal{Q}_q$ to be decomposable as a nontrivial direct sum of proper subsheaves. Hence, there is no such $f_q$.

On the other hand, when $q=2$, a similar computation shows that \[(\mathcal{Q}_2\otimes \mathcal{Q}_2)(X_2)=\left\{ \begin{bmatrix} a & a+b+c & a+b+c & 0 \\a+b+c & b & c & a+b+c \\ a+b+c & c & b & a+b+c \\ 0 & a+b+c & a+b+c & a\end{bmatrix}: a,b,c\in k\right\}\] as a subset of $\Mat_2(\mathcal{O}_{X_2'})(X_2')$. Here there are four idempotents, $0$ and $1$, as well as
\[
T=\begin{bmatrix}1 & 1 & 1 & 0\\ 1 & 0 & 0 & 1\\ 1 & 0 & 0 &1 \\ 0 & 1 & 1 & 1 \end{bmatrix} \quad \mbox{and} \quad 1-T=\begin{bmatrix}0 & 1 & 1 & 0\\ 1 & 1 & 0 & 1\\ 1 & 0 & 1 & 1 \\ 0 & 1 & 1 &0 \end{bmatrix}.\]
The nontrivial idempotents $T$ and $1-T$ induce a direct sum decomposition $\mathcal{Q}_2\cong \mathcal{F}\oplus \mathcal{G}$ with $\mathcal{F}$ and $\mathcal{G}$ each of rank $2$. Moreover, the inclusion $\mathcal{O}_X\subset \mathcal{Q}_2$ factors through one of the direct summands, say $\mathcal{F}$, and the composition \[\mathcal{F}\cong \mathcal{F}\oplus 0\subset \mathcal{Q}_2\xrightarrow{\mathrm{trd}_{\mathcal{Q}_2}} \mathcal{O}_X\] realizes $\mathcal{F}/\mathcal{O}_X\cong \mathcal{O}_X$.

One can explicitly construct a semi-trace in this case. Indeed, the composition \[f_2:\cSkew_{\mathcal{Q}_2,\sigma_2}\subset \mathcal{Q}_2\cong \mathcal{F}\oplus \mathcal{G}\twoheadrightarrow \mathcal{F}\oplus 0\subset \mathcal{Q}_2\] satisfies $\mathrm{trd}_{\mathcal{Q}_2}\circ f_2=0$ and so $f_2$ has target contained in $\mathcal{O}_X\subset \mathcal{F}$. After splitting, it is clear that $f_2(a+\sigma_2(a))=\mathrm{trd}_{Q_2}(a)$, which implies the same relation on $X_2$. Hence $\omega(X_2,\sigma_2)=0$ and $\Omega(X_2,\sigma_2)\neq 0$.
\end{rem}

\subsection{Lie algebra sheaves from quadratic triples}\label{Lie_extensions}
We maintain the assumption that the characteristic of $k$ is 2. Let $X$ and $(\mathcal{A},\sigma)$ be, as before, a scheme over $k$ and an Azumaya algebra with locally quadratic orthogonal involution on $X$. Assume that there is a semi-trace $f$ making $(\cA,\sigma,f)$ a quadratic triple. Then, the inclusion of algebraic group $X$-schemes $\mathbf{PGO}_{(\mathcal{A},\sigma,f)}\subset \mathbf{PGL}_{\mathcal{A}}$ induces an inclusion of Lie algebra sheaves $\mathfrak{pgo}_{(\mathcal{A},\sigma,f)}\subset \mathcal{A}/\mathcal{O}_X$. It turns out that this fact is completely determined by the vanishing of the weak obstruction:

\begin{thm}\label{thm: extensions_and_weak_obstructions}
Fix a base field $k$ of characteristic $2$. Let $X$ be any scheme over $k$. Suppose that $(\mathcal{A},\sigma)$ is an Azumaya algebra on $X$ with locally quadratic orthogonal involution. Then $\omega(\mathcal{A},\sigma)=0$ if and only if there exists an extension \begin{equation}\label{eq: extf} 0\rightarrow \cAlt_{\mathcal{A},\sigma}/\mathcal{O}_X\rightarrow \mathcal{F}\rightarrow \mathcal{O}_X\rightarrow 0\end{equation} whose pushout along the inclusion $\cAlt_{\mathcal{A},\sigma}/\mathcal{O}_X\subset \cSkew_{\mathcal{A},\sigma}/\mathcal{O}_X$ yields the extension $\alpha(\mathcal{E}')\in \mathrm{Ext}^1(\mathcal{O}_X,\cSkew_{\mathcal{A},\sigma}/\mathcal{O}_X)$ of \Cref{lem: inter}. 

Moreover, if $\omega(\mathcal{A},\sigma)=0$ so that an extension of the form \eqref{eq: extf} exists, then $\mathcal{F}$ is a locally free subsheaf of $\mathcal{A}/\mathcal{O}_X$ and $\mathcal{F}=\mathfrak{pgo}_{(\mathcal{A},\sigma,f)}$ for some quadratic pair $(\sigma,f)$ containing the involution $\sigma$.
\end{thm}

\begin{proof}
From the exact sequence \[0\rightarrow \cAlt_{\mathcal{A},\sigma}/\mathcal{O}_X\rightarrow \cSkew_{\mathcal{A},\sigma}/\mathcal{O}_X\rightarrow \cSkew_{\mathcal{A},\sigma}/\cAlt_{\mathcal{A},\sigma}\rightarrow 0\] there is an exact sequence
\begin{equation*}
\resizebox{\textwidth}{!}{$
\mathrm{Ext}^1(\mathcal{O}_X,\cAlt_{\mathcal{A},\sigma}/\mathcal{O}_X)\rightarrow \mathrm{Ext}^1(\mathcal{O}_X,\cSkew_{\mathcal{A},\sigma}/\mathcal{O}_X)\xrightarrow{\varphi} \mathrm{Ext}^1(\mathcal{O}_X, \cSkew_{\mathcal{A},\sigma}/\cAlt_{\mathcal{A},\sigma}).
$
}
\end{equation*} By construction, we have that $\varphi(\alpha(\mathcal{E}'))$ identifies with the weak obstruction $\omega(\mathcal{A},\sigma)$ under the isomorphism \[\mathrm{Ext}^1(\mathcal{O}_X, \cSkew_{\mathcal{A},\sigma}/\cAlt_{\mathcal{A},\sigma})\cong\mathrm{H}^1(X, \cSkew_{\mathcal{A},\sigma}/\cAlt_{\mathcal{A},\sigma}),\] hence the first claim.

To see the second claim, we note that any extension of the form \eqref{eq: extf} fits into a commutative diagram with exact rows as below.\\

\begin{adjustbox}{max width = \textwidth, center}
\begin{tikzcd}[column sep = small]
& 0\arrow{rr} &&[-4ex] \cAlt_{\mathcal{A},\sigma}/\mathcal{O}_X \arrow[equals]{dd}\arrow{dl}\arrow{rr} && \mathcal{F}\arrow{dd} \arrow{dl}\arrow{rr} &&[-3ex] \mathcal{O}_X \arrow[equals]{dl}\arrow[near start, dotted, "i'"]{dd}\arrow{rr} && 0\\
0\arrow{rr} && \cSkew_{\mathcal{A},\sigma}/\mathcal{O}_X\arrow[crossing over]{rr} && \mathcal{E}'\arrow[crossing over]{rr} && \mathcal{O}_X\arrow[crossing over]{rr}&& 0 &\\
& 0\arrow{rr} &&\cAlt_{\mathcal{A},\sigma}/\mathcal{O}_X \arrow{rr}\arrow{dl} && \mathcal{A}/\mathcal{O}_X\arrow[near start, "\pi"]{rr}\arrow[equals]{dl} && \mathcal{A}/\cAlt_{\mathcal{A},\sigma}\arrow{rr}\arrow["1+\sigma"]{dl} && 0\\
0\arrow{rr} && \cSkew_{\mathcal{A},\sigma}/\mathcal{O}_X\arrow{rr}\arrow[from=uu, equals, crossing over] && \mathcal{A}/\mathcal{O}_X\arrow["1+\sigma"]{rr}\arrow[from=uu, crossing over] && \cAlt_{\mathcal{A},\sigma}\arrow{rr}\arrow[from=uu, near start, "i", crossing over] && 0 &\\
\end{tikzcd}
\end{adjustbox} Here, $\pi$ is the canonical surjection. The existence of $i'$ follows. This gives an element $\ell\in \mathcal{A}/\cAlt_{\mathcal{A},\sigma}(X)$ such that $\sigma(\ell)+\ell=1$. The inclusion $\mathcal{F}\subset \mathcal{A}/\mathcal{O}_X$ is clear and the equality $\mathcal{F}=\mathfrak{pgo}_{(\mathcal{A},\sigma,f)}$ with $f=\mathrm{trd}_{\mathcal{A}}(\ell\cdot -)$ can be checked locally.\end{proof}

The vanishing of the intermediate obstruction then controls whether $(\cA,\sigma)$ can be equipped with a semi-trace $f$ such that the sequence \eqref{eq: extf} determining the Lie algebra sheaf $\mathfrak{pgo}_{(\mathcal{A},\sigma,f)}$ is split.

\begin{thm}\label{thm: liesections}
Fix a base field $k$ of characteristic $2$. Let $X$ be any scheme over $k$. Suppose that $(\mathcal{A},\sigma,f)$ is a quadratic triple over $X$. Let $\ell\in (\mathcal{A}/\cAlt_{\mathcal{A},\sigma})(X)$ be the unique section such that $\ell+\sigma(\ell)=1$ and $f=\mathrm{trd}_{\mathcal{A}}(\ell\cdot -)$ locally. 

Then the Lie algebra sheaf $\mathfrak{pgo}_{(\mathcal{A},\sigma,f)}$ of the group scheme $\mathbf{PGO}_{(\mathcal{A},\sigma,f)}$ over $X$ fits into a canonical extension \begin{equation}\tag{Lie}\label{eq: lie}
0\rightarrow \cAlt_{\mathcal{A},\sigma}/\mathcal{O}_X\rightarrow \mathfrak{pgo}_{(\mathcal{A},\sigma,f)}\rightarrow \mathcal{O}_X\rightarrow 0.\end{equation} The extension class $\alpha(\mathfrak{pgo}_{(\mathcal{A},\sigma,f)})\in \mathrm{Ext}^1(\mathcal{O}_X,\cAlt_{\mathcal{A},\sigma}/\mathcal{O}_X)$ that is associated to the sequence \eqref{eq: lie} maps to the intermediate obstruction $\Omega'(\mathcal{A},\sigma)$ under the morphism \[\mathrm{Ext}^1(\mathcal{O}_X,\cAlt_{\mathcal{A},\sigma}/\mathcal{O}_X)=\mathrm{H}^1(X,\cAlt_{\mathcal{A},\sigma}/\mathcal{O}_X)\rightarrow \mathrm{H}^1(X,\cSkew_{\mathcal{A},\sigma}/\mathcal{O}_X)\] induced by the inclusion $\cAlt_{\mathcal{A},\sigma}/\mathcal{O}_X\subset \cSkew_{\mathcal{A},\sigma}/\mathcal{O}_X$.

Moreover, the set of right-splittings of the sequence \eqref{eq: lie} is in bijection with the set of sections $\bar{\ell}\in (\mathcal{A}/\mathcal{O}_X)(X)$ lifting $\ell\in (\mathcal{A}/\cAlt_{\mathcal{A},\sigma})(X)$. 
\end{thm}

\begin{proof}
Let $\mathcal{E}'\subset \mathcal{A}/\mathcal{O}_X$ be the sheaf constructed in the proof of \Cref{lem: inter} and let $i':\mathcal{O}_X\rightarrow \mathcal{A}/\cAlt_{\mathcal{A},\sigma}$ be the morphism sending $1$ to $\ell$. There is then a commutative diagram with exact rows

\begin{adjustbox}{max width = \textwidth, center}
\begin{tikzcd}[column sep = small]
& 0\arrow{rr} &&[-4ex] \cAlt_{\mathcal{A},\sigma}/\mathcal{O}_X \arrow[equals]{dd}\arrow{dl}\arrow{rr} && \mathcal{F}\arrow{dd} \arrow{dl}\arrow{rr} &&[-3ex] \mathcal{O}_X \arrow[equals]{dl}\arrow[near start, "i'"]{dd}\arrow{rr} && 0\\
0\arrow{rr} && \cSkew_{\mathcal{A},\sigma}/\mathcal{O}_X\arrow[crossing over]{rr} && \mathcal{E}'\arrow[crossing over]{rr} && \mathcal{O}_X\arrow[crossing over]{rr}&& 0 &\\
& 0\arrow{rr} &&\cAlt_{\mathcal{A},\sigma}/\mathcal{O}_X \arrow{rr}\arrow{dl} && \mathcal{A}/\mathcal{O}_X\arrow[near start, "\pi"]{rr}\arrow[equals]{dl} && \mathcal{A}/\cAlt_{\mathcal{A},\sigma}\arrow{rr}\arrow["1+\sigma"]{dl} && 0\\
0\arrow{rr} && \cSkew_{\mathcal{A},\sigma}/\mathcal{O}_X\arrow{rr}\arrow[from=uu, equals, crossing over] && \mathcal{A}/\mathcal{O}_X\arrow["1+\sigma"]{rr}\arrow[from=uu, crossing over] && \cAlt_{\mathcal{A},\sigma}\arrow{rr}\arrow[from=uu, near start, "i", crossing over] && 0 &\\
\end{tikzcd}
\end{adjustbox} where $\pi$ is the canonical surjection and $\mathcal{F}$ is the pullback sheaf of the pair $(\pi,i')$. We claim that $\mathcal{F}$ identifies with the Lie algebra sheaf $\mathfrak{pgo}_{(\mathcal{A},\sigma,f)}$ inside of $\mathcal{A}/\mathcal{O}_X$. To see this, one may work locally where the top row of the above diagram splits. The claim in the case that this extension splits is left to the reader.

Next, applying the functor $\mathrm{Hom}(\mathcal{O}_X,-)$ to the commutative diagram above gives the commutative diagram below.

\begin{adjustbox}{max width = 0.75\textwidth, center}
\begin{tikzcd}[column sep = 0 em]
&[6ex] \mathcal{O}_X\arrow["i'", near end]{dd} \arrow[equals]{dl}\arrow["\delta"]{rr} &&[-10ex] \mathrm{Ext}^1(\mathcal{O}_X,\cAlt_{\mathcal{A},\sigma}/\mathcal{O}_X) \arrow{dl}\arrow[equals]{dd}\\
\mathcal{O}_X\arrow[crossing over]{rr} && \mathrm{Ext}^1(\mathcal{O}_X,\cSkew_{\mathcal{A},\sigma}/\mathcal{O}_X)\\
&\mathcal{A}/\cAlt_{\mathcal{A},\sigma}\arrow{rr}\arrow["1+\sigma"]{dl} &&\mathrm{H}^1(X,\cAlt_{\mathcal{A},\sigma}/\mathcal{O}_X)\arrow{dl}\\
 \cAlt_{\mathcal{A},\sigma}\arrow{rr}\arrow[from=uu, "i", crossing over] && \mathrm{H}^1(X,\cSkew_{\mathcal{A},\sigma}/\mathcal{O}_X)\arrow[from=uu, equals, crossing over]\\
\end{tikzcd}
\end{adjustbox}
A diagram chase then shows that the class $\alpha(\mathcal{F})=\delta(1)$ of the extension $\mathcal{F}$ above is sent to $\Omega'(\mathcal{A},\sigma)$ as claimed.

Now, a section of the surjection $\mathcal{F}\rightarrow \mathcal{O}_X$ determines an element $\bar{\ell}\in (\mathcal{A}/\mathcal{O}_X)(X)$. By commutativity of the diagrams above, it follows $\bar{\ell}$ lifts $\ell$. Conversely, since $\mathcal{F}$ is defined as the fiber product sheaf of the pair $(\pi,i')$, we find that any section $\bar{\ell}\in(\mathcal{A}/\mathcal{O}_X)(X)$ lifting $\ell$ defines a section of $\mathcal{F}$ which splits the surjection $\mathcal{F}\rightarrow \mathcal{O}_X$. These two correspondences are mutually inverse, as one can immediately check.
\end{proof}

\begin{cor}\label{intermediate_meaning}
Fix a base field $k$ of characteristic $2$. Let $X$ be any scheme over $k$ and let $(\cA,\sigma)$ be an Azumaya algebra with locally quadratic orthogonal involution. Then, the intermediate obstruction $\Omega'(\cA,\sigma)$ vanishes if and only if there exists a semi-trace $f$ making $(\cA,\sigma,f)$ a quadratic triple such that the canonical extension
\[
0 \to \cAlt_{\cA,\sigma}/\cO_X \to \fpgo_{(\cA,\sigma,f)} \to \cO_X \to 0
\]
is split.
\end{cor}
\begin{proof}
If such an $f$ exists for which the sequence is split, then the extension class $\alpha(\fpgo_{(\cA,\sigma,f)}) \in \Ext^1(\cO_X,\cAlt_{\cA,\sigma}/\cO_X)$ is zero. By \Cref{thm: liesections}, this class maps to the intermediate obstruction $\Omega'(\cA,\sigma)$, which is therefore also zero.

Conversely, suppose $\Omega'(\cA,\sigma)=0$. This means there is a section $\bar{\ell} \in (\cA/\cO_X)(X)$ which maps to $1\in \cSymd_{\cA,\sigma}(X)$. Hence, we may use the image of $\bar{\ell}$ under the map $(\cA/\cO_X)(X) \to (\cA/\cAlt_{\cA,\sigma})(X)$ to define a semi-trace $f$. Now, invoking \Cref{thm: liesections} again, the section $\bar{\ell}$ corresponds to a right-splitting of the sequence \eqref{eq: lie} for this $f$. This concludes the proof.
\end{proof}

\section{Relative deformations of quadratic triples}\label{deformations_of_quadratic_triples}
Let $X$ be an arbitrary scheme and let $(\mathcal{A},\sigma,f)$ be a quadratic triple on $X$. In this section we analyze the deformations of $(\mathcal{A},\sigma,f)$ in comparison with the deformations of $\mathcal{A}$. We say that quadratic triples are \emph{relatively unobstructed} when they deform if and only if the underlying algebra deforms. We show that, in two extraordinary cases, certain quadratic triples are always relatively unobstructed. 

Specifically, in \Cref{relatively_unobstructed}, we show that any quaternion Azumaya algebra with quadratic pair is relatively unobstructed as are all quadratic triples over a scheme $X$ on which $2$ is a global unit. Both results follow from observing that in these cases, the Lie algebra sheaf $\fpgo_{(\cA,\sigma,f)}$ is a direct summand of $\fpgl_\cA$.

\begin{lem}\label{degree_2_lie_split}
Let $X$ be an arbitrary scheme. Let $(\cA,\sigma,f)$ be a quadratic triple on $X$ and assume that $\deg(\cA)=2$. Then $\fpgo_{(\cA,\sigma,f)}$ is a direct summand of $\fpgl_\cA$.
\end{lem}
\begin{proof}
From the descriptions of the Lie algebra sheaves given in \Cref{Lie_Algebras}, we know that
\[
\fpgl_2 = \Mat_{2n}(\cO_X)/\cO_X
\]
where for a general matrix we have
\[
\begin{bmatrix} a & b \\ c & d \end{bmatrix} \equiv \begin{bmatrix} 0 & b \\ c & d-a \end{bmatrix} \text{ as sections in } \Mat_{2n}(\cO_X)/\cO_X.
\]
Therefore, we can write for $U\subseteq X$ open
\[
\fpgl_2(U) = \left\{ \begin{bmatrix} 0 & a \\ b & c \end{bmatrix} \middle\vert\, a,b,c \in \cO_X(U) \right\} \\
\]
and then compute that
\[
\fpgo_2(U) = \left\{ \begin{bmatrix} 0 & 0 \\ 0 & c \end{bmatrix} \middle\vert\, c \in \cO_X(U) \right\}
\]
Thus there is a splitting $\fpgl_2 = \fpgo_2 \oplus \mathcal{N}$ where
\[
\mathcal{N}(U) = \left\{ \begin{bmatrix} 0 & a \\ b & 0 \end{bmatrix} \middle\vert\, a,b \in \cO_X(U) \right\}. \\
\]

Now, any automorphism $\varphi \in \PGO_2(U)$ induces an automorphism $\overline{\varphi}$ on $\fpgl_2(U)$ which stabilizes $\fpgo_2(U)$. We argue that, in addition, any $\overline{\varphi}$ also stabilizes $\mathcal{N}(U)$. Since any element of $\mathcal{N}(U)$ is (represented by) a symmetric matrix $s$ with $f_2(s)=0$, the image $\varphi(s)$ must have the same properties. Thus we have that
\[
\varphi(s) = \begin{bmatrix} a & b \\ c & a \end{bmatrix}
\]
with $a=0$ since $f(\varphi(s)) = a$. Thus $\varphi(s) \in \mathcal{N}(U)$. The fact that all automorphisms in $\PGO_2$ stabilize $\mathcal{N}$ implies our claim as follows.

Let $\cU = \{U_i \to X\}_{i\in I}$ be an \'etale cover splitting $\cA$ and let $(\varphi_{ij}) \in Z^1(\cU,\PGO_2)$ be a cocycle defining $\cA$. Then $(\overline{\varphi_{ij}}) \in Z^1(\cU,\bAut(\fpgo_2))$ will define $\fpgo_{(\cA,\sigma,f)}$ and the same cocycle will define $\fpgl_\cA$ when viewed as an element of $Z^1(\cU,\bAut(\fpgl_2))$.  Letting $\pi_i \colon \fpgl_\cA|_{U_i} \surj \mathcal{N}|_{U_i}$ be the canonical projections, the following diagrams will commute since $\mathcal{N}$ is stabilized:
\[
\begin{tikzcd}
(\fpgo_2|_{U_i} \oplus \mathcal{N}|_{U_i})|_{U_{ij}} \ar[d,"\pi_i|_{U_{ij}}"] \ar[r,"\overline{\varphi_{ij}}"] & (\fpgo_2|_{U_j} \oplus \mathcal{N}|_{U_j})|_{U_{ij}} \ar[d,"\pi_j|_{U_{ij}}"]\\
(\fpgo_2|_{U_i})|_{U_{ij}} \ar[r,"\overline{\varphi_{ij}}"] & (\fpgo_2|_{U_j})|_{U_{ij}}.
\end{tikzcd}
\]
Thus, the $\pi_i$ glue into a global projection $\pi \colon \fpgl_\cA \surj \fpgo_{(\cA,\sigma,f)}$ which splits the inclusion $\fpgo_{(\cA,\sigma,f)} \inj \fpgl_\cA$, and we are done.
\end{proof}

\begin{lem}\label{thm: n2split}
Let $X$ be a scheme such that $2$ is invertible on $X$, i.e.\ $2\in \mathcal{O}_X^\times(X)$. Let $(\mathcal{A},\sigma,f)$ be a quadratic triple on $X$. Then $\mathfrak{pgo}_{(\mathcal{A},\sigma,f)}\cong \cAlt_{\mathcal{A},\sigma}$. In particular, the inclusion $\mathfrak{pgo}_{(\mathcal{A},\sigma,f)}\subset \mathcal{A}/\mathcal{O}_X$ realizes $\mathfrak{pgo}_{(\mathcal{A},\sigma,f)}$ as a direct summand of $\mathcal{A}/\mathcal{O}_X$.
\end{lem}

\begin{proof}
Consider the composition \[\cAlt_{\mathcal{A},\sigma}\rightarrow \mathcal{A}\rightarrow \mathcal{A}/\mathcal{O}_X\] of the canonical inclusion and the quotient. This map is injective and one can check that $\cAlt_{\mathcal{A},\sigma}\subset\mathfrak{pgo}_{(\mathcal{A},\sigma,f)}$. For any point $x\in X$, the induced map \[(\cAlt_{\mathcal{A},\sigma})_{\kappa(x)}\rightarrow(\mathfrak{pgo}_{(\mathcal{A},\sigma)})_{\kappa(x)}\] is an isomorphism since $\kappa(x)$ is of characteristic not $2$. Thus $\cAlt_{\mathcal{A},\sigma}\cong \mathfrak{pgo}_{(\mathcal{A},\sigma)}$. 

Finally, the exact sequence \[0\rightarrow \cAlt_{\mathcal{A},\sigma}\rightarrow \mathcal{A}/\mathcal{O}_X\rightarrow \mathcal{N}\rightarrow 0\] is left-split by the map $(1/2)(\mathrm{Id}-\sigma)$.
\end{proof}

The following theorem shows how to use \Cref{degree_2_lie_split,thm: n2split} to obtain statements on the relative unobstructedness of quadratic triples.

\begin{thm}\label{relatively_unobstructed}
Let $k$ be any field and let $X$ be any scheme that's separated over $k$. Let $(\mathcal{A},\sigma,f)$ be a quadratic triple on $X$. If either:
\begin{enumerate}
\item\label{relatively_unobstructed_i} $k$ is not characteristic $2$, or
\item\label{relatively_unobstructed_ii} $\cA$ is a quaternion algebra,
\end{enumerate}
then for any small extension \[0\rightarrow J \rightarrow C'\rightarrow C\rightarrow 0\] of Artinian local $k$-algebras $(C',\mathfrak{m}_{C'})$ and $(C,\mathfrak{m}_C)$ with residue fields $k$, the outer vertical arrows of the commutative diagram below
\[
\begin{tikzcd}[column sep = tiny, center picture]
\mathrm{H}^1(X,\mathfrak{pgo}_{(\mathcal{A},\sigma,f)})\otimes_k J \arrow{r}\arrow{d} & \mathrm{Def}_{(\mathcal{A},\sigma,f)}(C')\arrow{r}\arrow{d} & \mathrm{Def}_{(\mathcal{A},\sigma,f)}(C)\arrow{r}\arrow{d} & \mathrm{H}^2(X,\mathfrak{pgo}_{(\mathcal{A},\sigma,f)})\otimes_k J\arrow{d} \\
\mathrm{H}^1(X,\mathcal{A}/\mathcal{O}_X)\otimes_k J \arrow{r} & \mathrm{Def}_{\mathcal{A},Az}(C')\arrow{r} & \mathrm{Def}_{\mathcal{A},Az}(C)\arrow{r} & \mathrm{H}^2(X,\mathcal{A}/\mathcal{O}_X)\otimes_k J
\end{tikzcd}
\] are injective. 

Consequently, if $\xi=(\mathcal{A}_C,\sigma_C,f_C)$ is any deformation of $(\mathcal{A},\sigma,f)$ to $X_C$, then the obstruction $ob(\xi; X_{C'}/X_C)$ to deforming $\xi$ to $X_{C'}$ vanishes if and only if the obstruction $ob(\mathcal{A}_C;X_{C'}/X_{C})$ to deforming $\mathcal{A}_C$ as an Azumaya algebra vanishes.
\end{thm}
\begin{proof}
We obtain the commutative diagram in the statement from \Cref{cor: forg}. The outer vertical arrows are injective since they have a left inverse arising from the splitting\[
\fpgo_{(\cA,\sigma,f)}\to \fpgl_\cA = \cA/\cO_X \to \fpgo_{(\cA,\sigma,f)}
\]
which exists either by \Cref{thm: n2split} in case \ref{relatively_unobstructed_i} or \Cref{degree_2_lie_split} in case \ref{relatively_unobstructed_ii}.
\end{proof}
As a brief warning, when we have a quadratic triple $(\cA_C,\sigma_C,f_C)$ over $X_C$ as in \Cref{relatively_unobstructed} and we assume that there exists a deformation $\cA_{C'}$ on $X_{C'}$ of $\cA_C$ as an Azumaya algebra, we are guaranteed that there exists a deformation $(\cB_{C'},\sigma_{C'},f_{C'})$ on $X_{C'}$ of the quadratic triple, but the argument does not imply that $\cA_{C'} \cong \cB_{C'}$.

For the other extreme, when $X$ is purely of characteristic $2$, the following lemma and subsequent theorem show that the situation is completely different.
\begin{lem}\label{lem: zero}
Let $k$ be a field of characteristic $2$.
Let $X$ be a smooth and projective surface over $k$. Assume that the canonical bundle of $X$ is trivial, i.e.\ $K_X\cong \mathcal{O}_X$. Fix an Azumaya algebra with locally quadratic orthogonal involution $(\mathcal{A},\sigma)$ over $X$ such that the map \[\mathcal{A}(X)\xrightarrow{\mathrm{Id}+\sigma} \cAlt_{\mathcal{A},\sigma}(X)\] is identically zero. 

Then, with these assumptions, the canonical map \begin{equation}\mathrm{H}^2(X,\cAlt_{\mathcal{A},\sigma}/\mathcal{O}_X)\rightarrow \mathrm{H}^2(X,\mathcal{A}/\mathcal{O}_X)\end{equation} is identically zero as well.
\end{lem}

\begin{proof}
Since $\sigma$ is assumed locally quadratic, we have $\mathcal{O}_X\subset \cSymd_{\mathcal{A},\sigma}=\cAlt_{\mathcal{A},\sigma}$. Consider, then, the canonical commutative square of sheaves below.
\begin{equation}\label{eq: altsq}
\begin{tikzcd}
\cAlt_{\mathcal{A},\sigma}\arrow{r}\arrow{d} & \mathcal{A}\arrow{d}\\
\cAlt_{\mathcal{A},\sigma}/\mathcal{O}_X\arrow{r} & \mathcal{A}/\mathcal{O}_X
\end{tikzcd}
\end{equation} Associated to \eqref{eq: altsq}, there is an induced commutative diagram of $k$-vector spaces \begin{equation}\label{eq: h2altsq}
\begin{tikzcd}
\mathrm{H}^2(X,\cAlt_{\mathcal{A},\sigma}) \arrow{r}\arrow{d} & \mathrm{H}^2(X,\mathcal{A})\arrow{d} \\
\mathrm{H}^2(X,\cAlt_{\mathcal{A},\sigma}/\mathcal{O}_X)\arrow{r} & \mathrm{H}^2(X,\mathcal{A}/\mathcal{O}_X).
\end{tikzcd}
\end{equation}

By Serre duality, the square \eqref{eq: h2altsq} is canonically isomorphic with the square below, gotten by dualizing the square of global sections of the dual of \eqref{eq: altsq}:

\begin{equation}
\label{sq: duals}
\begin{tikzcd}
\mathrm{H}^0(X,(\cAlt_{\mathcal{A},\sigma})^\vee)^\vee \arrow{r}\arrow{d} & \mathrm{H}^0(X,\mathcal{A}^\vee)^\vee\arrow{d} \\
\mathrm{H}^0(X,(\cAlt_{\mathcal{A},\sigma}/\mathcal{O}_X)^\vee)^\vee\arrow{r} & \mathrm{H}^0(X,(\mathcal{A}/\mathcal{O}_X)^\vee)^\vee.
\end{tikzcd}
\end{equation}

Now, the dual of the canonical inclusion $\cAlt_{\mathcal{A},\sigma}\subset \mathcal{A}$ can be identified using the isomorphism $\mathcal{A}\cong \mathcal{A}^\vee$ coming from the reduced trace, and the isomorphism $\cAlt_{\mathcal{A},\sigma}\cong \cAlt_{\mathcal{A},\sigma}^\vee$ coming from the existence of a nondegenerate bilinear form $b_{-}$, see \cite[2.7 Lemma (ii)]{GNR2023AzumayaObstructions}. The composition \[\mathcal{A}\cong \mathcal{A}^\vee\rightarrow \cAlt_{\mathcal{A},\sigma}^\vee\cong \cAlt_{\mathcal{A},\sigma}\] is the map defined on sections by $x\mapsto x+\sigma(x)$. Because of our assumptions, we have that the top arrow in \eqref{sq: duals} is therefore zero. Since the left vertical arrow of \eqref{sq: duals} is a surjection, it follows that the bottom horizontal arrow of \eqref{sq: duals} is zero too.
\end{proof}

\begin{thm}\label{thm: zero}
Let $k$ be a field of characteristic $2$.
Let $X$ be a smooth and projective surface over $k$. Assume that the canonical bundle of $X$ is trivial, i.e.\ $K_X\cong \mathcal{O}_X$. Lastly, fix an Azumaya algebra with quadratic pair $(\mathcal{A},\sigma,f)$ over $X$ such that both of the maps \begin{equation}\mathcal{A}(X)\xrightarrow{\mathrm{Id}+\sigma} \cAlt_{\mathcal{A},\sigma}(X)\quad \mbox{and}\quad \cSym_{\mathcal{A},\sigma}(X)\xrightarrow{f} \mathcal{O}_X(X)\end{equation} are identically zero.

Then, with these assumptions, the canonical map \begin{equation}\mathrm{H}^2(X,\mathfrak{pgo}_{(\mathcal{A},\sigma,f)})\rightarrow \mathrm{H}^2(X,\mathcal{A}/\mathcal{O}_X)\end{equation} is identically zero as well.
\end{thm}

\begin{proof}
The Lie algebra sheaf $\mathfrak{pgo}_{(\mathcal{A},\sigma,f)}$ sits in a canonical extension \[0\rightarrow \cAlt_{\mathcal{A},\sigma}/\cO_X \rightarrow \mathfrak{pgo}_{(\mathcal{A},\sigma,f)}\rightarrow\mathcal{O}_X\rightarrow 0.\] Consider the commutative ladder with exact rows whose first column is the above extension. Here we define $\mathcal{N}_{(\mathcal{A},\sigma,f)}$ to be the appropriate locally free cokernel sheaf.
\begin{equation}\label{eq: snake}
\begin{tikzcd}
0 \arrow{r} & \cAlt_{\mathcal{A},\sigma}/\mathcal{O}_X\arrow{r}\arrow{d} & \mathcal{A}/\mathcal{O}_X\arrow{r}\arrow[equals]{d} & \mathcal{A}/\cAlt_{\mathcal{A},\sigma}\arrow{d}\arrow{r} & 0\\
0\arrow{r} & \mathfrak{pgo}_{(\mathcal{A},\sigma,f)}\arrow{r}\arrow{d} & \mathcal{A}/\mathcal{O}_X\arrow{r} & \mathcal{N}_{(\mathcal{A},\sigma,f)}\arrow{r} & 0\\ & \mathcal{O}_X & & &
\end{tikzcd}
\end{equation}
By the snake lemma, the kernel of the map $\mathcal{A}/\cAlt_{\mathcal{A},\sigma}\rightarrow \mathcal{N}_{(\mathcal{A},\sigma,f)}$ is isomorphic with $\mathcal{O}_X$. The inclusion of this kernel $\mathcal{O}_X\subset\mathcal{A}/\cAlt_{\mathcal{A},\sigma}$ is identified with the map sending $1$ to the element $\ell\in \mathcal{A}/\cAlt_{\mathcal{A},\sigma}(X)$ with $\ell+\sigma(\ell)=1$ which defines $f$. There is an induced commutative diagram with exact rows and columns as below, obtained by applying cohomology to \eqref{eq: snake}.
\begin{equation}\label{eq: cohsnake}
\begin{tikzcd}[column sep =small]
  & & \mathrm{H}^2(X,\mathcal{O}_X)\arrow["\tau"]{d}\\
\mathrm{H}^2(X,\cAlt_{\mathcal{A},\sigma}/\mathcal{O}_X)\arrow["\varphi"]{r}\arrow{d} & \mathrm{H}^2(X,\mathcal{A}/\mathcal{O}_X)\arrow["\psi"]{r}\arrow[equals]{d} & \mathrm{H}^2(X,\mathcal{A}/\cAlt_{\mathcal{A},\sigma})\arrow{d}\arrow{r} & 0\\
\mathrm{H}^2(X,\mathfrak{pgo}_{(\mathcal{A},\sigma,f)})\arrow["\rho"]{r}\arrow{d} & \mathrm{H}^2(X,\mathcal{A}/\mathcal{O}_X)\arrow{r} & \mathrm{H}^2(X,\mathcal{N}_{(\mathcal{A},\sigma,f)})\arrow{r}\arrow{d} & 0\\  \mathrm{H}^2(X,\mathcal{O}_X)\arrow{d} & & 0 &\\
0
\end{tikzcd}
\end{equation}

By \Cref{lem: zero}, the map $\varphi$ in \eqref{eq: cohsnake} is zero. Hence $\psi$ in \eqref{eq: cohsnake} is an isomorphism. If we could show that $\tau$ is also zero, it then follows that the map $\rho$ is zero as well, as claimed in the theorem statement.

By Serre duality, the map $\tau$ is dual to the map \[\mathrm{H}^0(X,(\mathcal{A}/\cAlt_{\mathcal{A},\sigma})^\vee)\rightarrow \mathrm{H}^0(X,\mathcal{O}_X)\] defined by sending a functional $g$ to $g(\ell)$. Note also that the inclusion $\cAlt_{\mathcal{A},\sigma}\subset \mathcal{A}$ yields the exact sequence of sheaves \[0\rightarrow (\mathcal{A}/\cAlt_{\mathcal{A},\sigma})^\vee \rightarrow \mathcal{A}^\vee\rightarrow \cAlt_{\mathcal{A},\sigma}^\vee\rightarrow 0.\] This implies that a functional $g$ on $\mathcal{A}/\cAlt_{\mathcal{A},\sigma}$ has the form $g=\mathrm{trd}_\mathcal{A}(a\cdot -)$ for some $a\in \mathcal{A}(X)$. By our assumptions, we have $a+\sigma(a)=0$, i.e.\ $a\in \cSym_{\mathcal{A},\sigma}(X)$. Hence, by the definition of the semi-trace $f$, we have \[g(\ell)=\mathrm{trd}_{\mathcal{A}}(a\ell)=f(a).\] Since we are assuming that $f$ is trivial on global sections, we find that $\tau$ is zero. As explained above, this concludes the proof.
\end{proof}

\begin{rem}\label{rem: assumptions}
Let $k$ be a field of characteristic $2$. Let $X$ be a smooth, projective $k$-surface with $K_X\cong \mathcal{O}_X$. Let $(\mathcal{A},\sigma,f)$ be a quadratic triple on $X$ with $4\mid \mathrm{deg}(\mathcal{A})$ and suppose that $\mathcal{A}(X)=k$. Then $(\mathcal{A},\sigma,f)$ satisfies the assumptions of \Cref{thm: zero}. Indeed, $1+\sigma(1)=0$ trivially and $f(1)=0$ by \cite[Corollary 3.8]{GNR2023AzumayaObstructions}.
\end{rem}

\begin{rem}\label{rem: specialcase}
Let $k$ be a field of characteristic $2$. Let $X$ be a smooth, projective $k$-surface with $K_X\cong \mathcal{O}_X$. Let $\mathcal{A}$ and $\mathcal{B}$ be two quaternion Azumaya algebras on $X$ and suppose that $(\mathcal{A}\otimes\mathcal{B})(X)=k$. In this case, there is a canonical quadratic triple $(\mathcal{A}\otimes\mathcal{B},\sigma,f)$ and an isomorphism $\mathfrak{pgo}_{(\mathcal{A}\otimes \mathcal{B},\sigma,f)}\cong \mathcal{A}/\mathcal{O}_X\oplus \mathcal{B}/\mathcal{O}_X$ which follows from the stack equivalence of \Cref{norm_functor_equiv}. Thus \Cref{cor: quartvan} is a special case of \Cref{thm: zero}.
\end{rem}

\section{A quadratic pair with obstructed deformations}\label{explicit_example}
In this section we prove \Cref{thm: main} which shows that, in characteristic 2, there exist quadratic triples with nonvanishing obstructions to deformation such that when one forgets the quadratic pair, the underlying Azumaya algebra has vanishing obstruction.

\subsection{Quaternion algebras on an Igusa surface}
Throughout the remainder of this text we fix a field $k$ which is both algebraically closed and of characteristic $2$. We also fix a pair of ordinary elliptic curves $(E_1,p_1)$ and $(E_2,p_2)$ over $k$, which are allowed to be equal, and we set $Y=E_1\times_k E_2$. There is a free action of $\mathbb{Z}/2\mathbb{Z}$ on $Y$ defined by letting the nonunit element act via the involution $i\colon Y \to Y$ given by $(x,y)\mapsto (-x,y+t)$ for the unique $2$-torsion point $p_2\neq t\in E_2[2](k)$. Set $\pi:Y\rightarrow X$ to be the associated quotient of this action. By construction, $X$ is an Igusa surface.

We are going to construct two quaternion Azumaya algebras on $X$. The first is constructed in the following way. Let
\begin{equation}\label{eq: unique_bundle}
0\rightarrow \mathcal{O}_{E_1}\rightarrow \mathcal{V}_1\rightarrow \mathcal{O}_{E_1}(p_1)\rightarrow 0
\end{equation}
be the unique non-split extension on $E_1$ (which was discussed in \Cref{exmp: e2q}). The sheaf $\mathcal{V}_1$ is an indecomposable locally free sheaf of rank 2. Pulling back to $Y$ along the projection $\pi_1:Y=E_1\times E_2\rightarrow E_1$ yields an extension \begin{equation}
0\rightarrow \mathcal{O}_Y\rightarrow \pi_1^*\mathcal{V}_1\rightarrow \mathcal{O}_Y(p_1\times E_2)\rightarrow 0.
\end{equation}
Because we have a commutative diagram
\[
\begin{tikzcd}
Y \ar[d,"i"] \ar[r,"\pi_1"] & E_1 \ar[d,"i_0"] \\
Y \ar[r,"\pi_1"] & E_1
\end{tikzcd}
\]
where $i_0(x)=-x$, we have that $i^*(\pi_1^*(\cV_1)) = \pi_1^*(i_0^*(\cV_1))$. However, it is clear that $i_0^*(\cO_{E_1}) = \cO_{E_1}$ and that $i_0^*(\cO_{E_1}(p_1)) = \cO_{E_1}(p_1)$ because $i_0(p_1)=p_1$, and so by the uniqueness of \eqref{eq: unique_bundle}, we have that $i_0^*(\cV_1)=\cV_1$. In turn, we conclude that
\[
i^*(\pi_1^*(\cV_1)) = \pi_1^*(\cV_1).
\]
We now use the \'etale cover $\pi \colon Y \to X$ and glue the objects in \eqref{eq: unique_bundle} to themselves to define objects on $X$. First, we have a fiber product diagram
\[
\begin{tikzcd}
Y\times_k \ZZ/2\ZZ \ar[r,"\pr_1"] \ar[d,"\alpha"] & Y \ar[d,"\pi"]  \\
Y \ar[r,"\pi"] & X
\end{tikzcd}\quad\mbox{where}\quad \begin{aligned}
\alpha \colon Y\times_k \ZZ/2\ZZ &\to Y \\
(y,0) &\mapsto y\\
(y,1) & \mapsto i(y)
\end{aligned}
\]
i.e., $\ZZ/2\ZZ$ acts on $Y$ via the involution $i \colon Y \to Y$. Alternatively, there is a canonical isomorphism $Y\times_X Y \cong Y\sqcup Y$ and then the maps $\pr_1$ and $\alpha$ appear as
\[
\pr_1\colon \begin{tikzcd}[column sep = -1.5ex]
Y \ar[dr,swap,near start,"\Id"] & \sqcup & Y \ar[dl,near start,"\Id"] \\
 & Y &
\end{tikzcd} \quad\quad\mbox{and}\quad\quad
\alpha \colon \begin{tikzcd}[column sep = -1.5ex]
Y \ar[dr,swap,near start,"\Id"] & \sqcup & Y \ar[dl,near start,"i"] \\
 & Y &
\end{tikzcd}.
\]
A sheaf $\cF$ on $Y\sqcup Y$ is of the form $\cF(U\sqcup V) = \cF_1(U)\times \cF_2(V)$ for two sheaves $\cF_1,\cF_2$ on $Y$. For such a sheaf, we write $\cF = \cF_1\sqcup \cF_2$. Gluing data for $\pi_1^*(\cV_1)$ consists of an isomorphism
\[
\pr_1^*(\pi_1^*(\cV_1)) \iso \alpha^*(\pi_1^*(\cV_1))
\]
which satisfies the cocycle condition. However, we have that
\begin{align*}
\pr_1^*(\pi_1^*(\cV_1)) &= \pi_1^*(\cV_1)\sqcup \pi_1^*(\cV_1), \text{ and}\\
\alpha^*(\pi_1^*(\cV_1)) &= \pi_1^*(\cV_1)\sqcup i^*(\pi_1^*(\cV_1)) = \pi_1^*(\cV_1)\sqcup \pi_1^*(\cV_1)
\end{align*}
so we may simply take the identity as our cocycle. 

For similar reasons, the identity is also a valid gluing data for $\cO(p_1\times E_2)$ and of course for $\cO_Y$. After using this gluing data to descend to $X$, we obtain an extension of $\cO_X$--modules
\begin{equation}\label{eq: extv}
0\rightarrow \mathcal{O}_X\rightarrow \mathcal{V}\rightarrow \mathcal{L}\rightarrow 0
\end{equation}
where $\mathcal{V}$ is a locally free module of rank 2 such that $\pi^*(\mathcal{V}) = \pi_1^*(\mathcal{V}_1)$ and $\cL$ is a line bundle with $\pi^*(\mathcal{L}) = \mathcal{O}_Y(p_1\times E_2)$. The first Azumaya algebra that we consider on $X$ is the neutral endomorphism algebra $\mathcal{Q}:=\cEnd(\mathcal{V})$.

For later calculations, we want to describe $\cQ$ with a cocycle coming from $\PGL_2$ and some cover of $X$. As in \Cref{exmp: e2q}, set $E_1' = E_1$ and consider the fppf cover $m_0 \colon E_1'\to E_1$ which is multiplication by $2$. We have that $\cQ_1 = \cEnd_{\cO_{E_1}}(\cV_1)$ is split by this map, i.e., $m_0^*(\cQ_1) \cong \Mat_2(\cO_{E_1'})$. We consider two other maps related to this multiplication by $2$ map on $E_1$. 

First, let $Y'=Y$ and let $m \colon Y' = E_1\times_k E_2 \to E_1\times_k E_2 = Y$ be the map which is multiplication by $2$ on the first factor. Second, there is a map $m_X \colon X' \to X$ given by $m_X([a,b])=[2a,b]$, which is well defined since $m\circ i = i \circ m$ on $Y$. Using the same techniques as in \cite{GNR2023AzumayaObstructions}, one can see that $Y'\times_Y Y' \cong Y'\times_k (\mu_2\times \ZZ/2\ZZ)$ and that under this isomorphism, the second projection $\pr_2\colon Y'\times_Y Y' \to Y'$ appears as
\begin{align*}
\beta \colon Y'\times_k (\mu_2\times \ZZ/2\ZZ) &\to Y' \\
((a,b),\gamma) &\mapsto (a+\gamma,b).
\end{align*}
Likewise, we have that $X'\times_X X' \cong X'\times_k (\mu_2\times \ZZ/2\ZZ)$ with second projection
\begin{align*}
\beta_X \colon X'\times_k (\mu_2\times \ZZ/2\ZZ) &\to X' \\
([a,b],\gamma) &\mapsto [a+\gamma,b].
\end{align*}
This is well-defined since $i(a+\gamma,b) = (-a+\gamma,b+t)$ because $-\gamma = \gamma$ for any $\gamma$ of the $2$-torsion group scheme $\mu_2\times\ZZ/2\ZZ$. Hence, these maps fit into a diagram
\begin{equation}\label{eq: times_2_covers}
\begin{tikzcd}
E_1'\times_k (\mu_2\times\ZZ/2\ZZ) \ar[r,yshift=0.5ex,"\pr_1"] \ar[r,yshift=-0.5ex] & E_1' \ar[r,"m_0"] & E_1 \\
Y'\times_k (\mu_2\times\ZZ/2\ZZ) \ar[r,yshift=0.5ex,"\pr_1"] \ar[r,yshift=-0.5ex,swap,"\beta"] \ar[u,,swap,"\pi_1\times\Id"] \ar[d,"\pi\times\Id"] & Y' \ar[u,swap,"\pi_1"] \ar[r,"m"] \ar[d,"\pi"] & Y \ar[u,swap,"\pi_1"] \ar[d,"\pi"] \\
X'\times_k (\mu_2\times\ZZ/2\ZZ) \ar[r,yshift=0.5ex,"\pr_1"] \ar[r,yshift=-0.5ex,swap,"\beta_X"] & X' \ar[r,"m_X"] & X
\end{tikzcd}
\end{equation}
which commutes when only the first projections are considered from the double arrows, as well as when the other maps are taken from each pair of double arrows.

\begin{lem}\label{cocyle_for_Q}
The quaternion algebra $\cQ$ constructed above on $X$ is split by the \'etale cover $m_X \colon X' \to X$, i.e., $m_X^*(\cQ) \cong \Mat_2(\cO_{X'})$. In particular, $\cQ$ is defined by the $1$--cocycle
\begin{align*}
\phi=\Inn\left(\begin{bmatrix} 0 & 1 \\ x & 0 \end{bmatrix}, \begin{bmatrix} 1 & 0 \\ 0 & x \end{bmatrix} \right)&\in \PGL_2\left((k[x]/(x^2-1))^2\right) \\
&= \PGL_2(X'\times_k (\mu_2\times\ZZ/2\ZZ)).
\end{align*}
\end{lem}
\begin{proof}
We begin by noting that all the schemes in the left column of \eqref{eq: times_2_covers} have the same global sections, namely $k[x]/(x^2-1)\otimes(k\times k) \cong \left(k[x]/(x^2-1)\right)^2$, and that the restriction maps along both $\pi_1\times\Id$ and $\pi\times\Id$ are the identity on those global sections. Likewise, $\cO_X(X) = \cO_Y(Y) = \cO_{E_1}(E_1) = k$ and the restrictions along $\pi_1$ and $\pi$ are the identity on $k$.

By definition, $\cQ_1$ is split over $E_1'$ and is defined by the cocycle
\[
\phi = \Inn\left(\begin{bmatrix} 0 & 1 \\ x & 0 \end{bmatrix}, \begin{bmatrix} 1 & 0 \\ 0 & x \end{bmatrix} \right)\in \PGL_2(E_1'\times_{E_1} E_1').
\]
This means that $\pi_1^*(\cQ_1)$ is defined by this cocycle pulled back along $\pi_1\times\Id$, which yields ``the same" cocycle since the pullback along $\pi_1\times\Id$ is the identity. We now consider the following diagram, the bottom two rows of which are from \eqref{eq: times_2_covers}.
\[
\begin{tikzcd}
 & Y'\times_{X'} Y' \ar[r,"m\times m"] \ar[d,xshift=-0.5ex] \ar[d,xshift=0.5ex] & Y\times_X Y \ar[d,xshift=-0.5ex] \ar[d,xshift=0.5ex] \\
Y'\times_k (\mu_2\times\ZZ/2\ZZ) \ar[r,yshift=0.5ex,"\pr_1"] \ar[r,yshift=-0.5ex,swap,"\beta"] \ar[d,"\pi\times\Id"] & Y' \ar[d,"\pi"] \ar[r,"m"] & Y \ar[d,"\pi"] \\
X'\times_k (\mu_2\times\ZZ/2\ZZ) \ar[r,yshift=0.5ex,"\pr_1"] \ar[r,yshift=-0.5ex,swap,"\beta_X"] & X' \ar[r,"m_X"] & X
\end{tikzcd}
\]
First, we show that $\cQ$ is split over $X'$. The algebra $\cQ$ is defined by gluing $\pi_1^*(\cQ_1)$ on $Y$ with the identity isomorphism over $Y\times_X Y$. Hence, the pullback $m_X^*(\cQ)$ on $X'$ will be defined by the pullback of this gluing information. In particular, it will be defined by gluing $m^*(\pi_1^*(\cQ_1))$ along the identity over $Y'\times_{X'} Y'$. However, we know that
\[
m^*(\pi_1^*(\cQ_1)) = \Mat_2(\cO_{Y'}),
\]
and so when we glue the matrix algebra to itself with the identity, we simply obtain the matrix algebra $\Mat_2(\cO_{X'})$ on $X'$. Thus $m_X^*(\cQ) = \Mat_2(\cO_{X'})$.

It is now almost immediate that $\cQ$ is defined by the given cocycle. Since $\cQ$ splits over $X'$, it is defined by some cocycle $\psi \in \PGL_2(X'\times_k (\mu_2\times\ZZ/2\ZZ))$. Additionally, because $\pi^*(\cQ) = \pi_1^*(\cQ_1)$, the pullback of $\psi$ along $\pi\times\Id$ must be $\phi$. However, the pullback
\[
\begin{tikzcd}
(\pi\times\Id)^* \colon &[-3.5em] \PGL_2((k[X]/(x^2-1))^2) \ar[r] \ar[d,equals] & \PGL_2((k[X]/(x^2-1))^2) \ar[d,equals] \\[-1 em]
& \PGL_2(X'\times(\mu_2\times\ZZ/2\ZZ)) & \PGL_2(Y'\times(\mu_2\times\ZZ/2\ZZ))
\end{tikzcd}
\]
is also simply the identity and so $\psi = \phi$. Hence, $\phi$ defines $\cQ$ on $X$ with respect to the cover $X' \to X$ as claimed, and we are done.
\end{proof}

\begin{cor}\label{global_sections_Q}
Let $\cQ=\cEnd(\mathcal{V})$ be the quaternion algebra as defined above on $X$. Then, the following hold: 
\begin{enumerate}
\item $\cQ(X) = k$;
\item there is a locally quadratic orthogonal involution $\sigma_1$ on $\cQ$;
\item there exists a semi-trace $f_1$ making $(\cQ,\sigma_1,f_1)$ a quadratic triple.
\end{enumerate}
\end{cor}
\begin{proof}
We know from \Cref{cocyle_for_Q} that $\cQ$ is defined by the cocycle $\phi$ over the cover $X' \to X$. Since we also have that $\cO_{X'}(X')= k$ and
\[
\cO_{X'\times(\mu_2\times\ZZ/2\ZZ)}(X'\times(\mu_2\times\ZZ/2\ZZ)) = (k[x]/(x^2-1))^2,
\]
we may repeat the calculations appearing in \cite[6.2]{GNR2023AzumayaObstructions}.
\end{proof}

The second Azumaya algebra on $X$ that we consider is constructed as follows. Let $p_1\neq s\in E_1[3](k)$ be a nontrivial $3$-torsion point, so that there is a natural surjection $E_1 \surj E_1/\langle s \rangle$ with kernel $\ZZ/3\ZZ$. We consider the dual of these abelian varities and get another surjection $\phi:E=(E_1/\langle s \rangle)^\vee\rightarrow E_1^\vee \cong E_1$ which has kernel $(\ZZ/3\ZZ)^\vee = \mu_3$. Since we are in characteristic $2$, we have $\mu_3 \cong \ZZ/3\ZZ$ and so $\phi$ is a $\mathbb{Z}/3\mathbb{Z}$-cover of $E_1$. Then, the cover \begin{equation} Z=E\times E_2\xrightarrow{\phi\times \mathrm{id}} E_1\times E_2\xrightarrow{\pi} X\end{equation} is an $S_3=\PGL_2(\mathbb{F}_2)$-cover of $X$. Indeed, let $S_3$ be generated by elements $\rho,\tau\in S_3$ with relations $\rho^3=e$, $\tau^2=e$, and $\rho\tau\rho\tau=e$ for the unit $e$ of $S_3$. There is an action of $S_3$ on $Z$ over $X$ given by \begin{equation}\label{eq: s3tors}
\begin{aligned}\mu:S_3\times_k Z\rightarrow Z,\quad\quad\quad
& \mu(\rho,(x,y))=(x+s,y)\\
& \mu(\tau,(x,y))=(-x,y+t).
\end{aligned}
\end{equation}
One can immediately check that the action \eqref{eq: s3tors} gives $Z$ the structure of an $S_3$-torsor over $X$. Following \Cref{rem: torsquat}, we construct from this torsor a quaternion Azumaya algebra $\mathcal{P}$ on $X$. An easy computation using the cocycle given in \Cref{rem: torsquat} and the facts that $\cO_Z(Z)=k$ and $\cO_{Z\times_X Z}(Z\times_X Z)\cong k^6$ shows that $\cP(X)=k$. We also note that, since $\GL_2(\mathbb{F}_2)=\PGL_2(\mathbb{F}_2)$, this Azumaya algebra is neutral, i.e., $\mathcal{P}=\cEnd(\mathcal{W})$ for a locally free sheaf $\mathcal{W}$ of rank $2$ on $X$.

\begin{lem}\label{lem: globalsectionsofqp}
The algebra $\mathcal{Q}\otimes \mathcal{P}$ has only constant global sections, $(\mathcal{Q}\otimes \mathcal{P})(X)=k$. In particular, by \Cref{rem: assumptions}, the canonical quadratic pair $(\sigma_0,f_0)$ on $\mathcal{Q}\otimes \mathcal{P}$ satisfies the assumptions of \Cref{thm: zero}.
\end{lem}
\begin{proof}
Since $\cQ$ is defined over the cover $X' \to X$ and $\cP$ is defined over the cover $Z \to X$, we will be able to find a cocycle for the tensor product $\cQ\otimes\cP$ on the fppf cover $X'\times_X Z \to X$ and use this to calculate the global sections. Therefore, we begin by identifying this fiber product. Just as $m\colon Y' \to Y$ and $m_X \colon X' \to X$ are multiplication by $2$ maps on the first factor, we can set $Z' = Z$ and define
\begin{align*}
m_Z \colon Z' = E\times_k E_2 &\to E\times_k E_2 = Z \\
(e,e_2) &\mapsto (2e,e_2).
\end{align*}
We then get a commutative diagram
\[
\begin{tikzcd}
Z' \ar[r,"m_Z"] \ar[d,"\rho"] & Z \ar[d,"\rho"] \\
Y' \ar[r,"m"] \ar[d,"\pi"] & Y \ar[d,"\pi"] \\
X' \ar[r,"m_X"] & X
\end{tikzcd}
\]
in which all square are cartesian, thus showing that $X'\times_X Z \cong Z'$. Since $Z\to X$ is an $S_3$--torsor and $X'\to X$ is a $(\mu_2\times\ZZ/2\ZZ)$--torsor, the cover $Z' \to X$ is a $(\mu_2\times\ZZ/2\ZZ)\times_k S_3$--torsor. Thus, $Z'\times_X Z' \cong Z'\times_k (\mu_2\times\ZZ/2\ZZ)\times_k S_3$ and we obtain a large diagram
\[
\begin{tikzcd}
Z'\times_k (\mu_2\times \ZZ/2\ZZ)\times_k S_3 \ar[dr,yshift={0.5ex},xshift={0.25ex},"\pr_1"] \ar[dr,yshift={-0.5ex},xshift={-0.25ex},swap,"\zeta"] \ar[rr,"f_1"] \ar[dd,"f_2"]  & &[2em] Z \times_k S_3 \ar[d,swap,xshift={-0.5ex},"\pr_1"] \ar[d,xshift={+0.5ex},"\mu"] \\
 & Z' \ar[r,"m_Z"] \ar[d] \ar[dr] & Z \ar[d] \\
X'\times_k (\mu_2\times\ZZ/2\ZZ) \ar[r,yshift={-0.5ex},swap,"\beta_X"] \ar[r,yshift={0.5ex},"\pr_1"] & X' \ar[r,"m_X"] & X
\end{tikzcd}
\]
where $f_1 = (m_Z\times\Id)\circ \pr_{13}$ and $f_2 = ((\pi\circ\rho)\times\Id)\circ \pr_{12}$. The map $\mu$ was introduced above, the map $\beta_X$ is the analogue of the map $\beta \colon Y'\times(\mu_2\times\ZZ/2\ZZ) \to Y'$, and $\zeta$ similarly applies the group action of $(\mu_2\times \ZZ/2\ZZ)\times_k S_3$ on $Z'$. In particular, the maps $\beta_X$, $\zeta$, and $\mu$ play the roles of the second projections. Therefore, the diagram above commutes if one considers only the first projections from each set of double arrows, and it also commutes if one only considers $\beta_X$, $\zeta$, and $\mu$ instead.

Now, it is clear that $\cQ$ will be split over $Z'$ and therefore given by the pullback along $f_2$ of its cocycle in $\PGL_2(X'\times_k(\mu_2\times \ZZ/2\ZZ))$ identified in \Cref{cocyle_for_Q}. Likewise, $\cP$ will be split over $Z'$ and be defined by the pullback of its cocycle in $\PGL_2(Z\times_k S_3)$ along $f_1$. We have that
\[
\cO_{Z'\times_k (\mu_2\times \ZZ/2\ZZ)\times_k S_3}(Z'\times_k (\mu_2\times \ZZ/2\ZZ)\times_k S_3) = \left(k[x]/(x^2-1)\right)^2 \otimes k^6
\]
and that
\begin{align*}
f_1^* = 1\otimes \Id &\colon k^6 \to \left(k[x]/(x^2-1)\right)^2 \otimes k^6, \text{ and}\\
f_2^* = \Id\otimes 1 &\colon \left(k[x]/(x^2-1)\right)^2 \to k^6.
\end{align*}
This means that $\cQ$ is defined by the cocycle $\phi\otimes (1,1,1,1,1,1)$ and $\cP$ is defined by the cocycle $(1,1)\otimes(g_1,g_2,g_3,g_4,g_5,g_6)$ where $g_i$ are the elements of $\PGL_2(\FF_2)$. Using the isomorphism
\begin{align*}
\left(k[x]/(x^2-1)\right)^2 \otimes k^6 &\iso \left(k[x]/(x^2-1)\right)^{12} \\
(p_1,p_2)\otimes(a_1,\ldots,a_6) &\mapsto (p_1a_1,p_1a_2,\ldots,p_1a_6,p_2a_1,p_2a_2,\ldots,p_2a_6)
\end{align*}
and setting
\[
\phi_1 = \Inn\left(\begin{bmatrix} 0 & 1 \\ x & 0 \end{bmatrix}\right) \text{ and } \phi_2 = \Inn\left(\begin{bmatrix} 1 & 0 \\ 0 & x \end{bmatrix}\right)
\]
we obtain the cocycles
\begin{align*}
&(\phi_1,\phi_1,\phi_1,\phi_1,\phi_1,\phi_1,\phi_2,\phi_2,\phi_2,\phi_2,\phi_2,\phi_2), \text{ and}\\
&(g_1,g_2,g_3,g_4,g_5,g_6,g_1,g_2,g_3,g_4,g_5,g_6)
\end{align*}
in $\PGL_2(k[x]/(x^2-1))^{12}$ for $\cQ$ and $\cP$, respectively. Hence, the tensor product algebra $\cQ\otimes \cP$ is defined by the cocycle
\[
(\phi_1\otimes g_1,\ldots,\phi_1\otimes g_6,\phi_2\otimes g_1,\ldots,\phi_2\otimes g_6) \in \PGL_4(k[x]/(x^2-1))^{12}.
\]
Since $\cO_{Z'}(Z') = k$, the restriction maps along both projections $Z'\times_X Z' \to Z'$ are simply the diagonal map $k\to (k[x]/(x^2-1))^{12}$. Thus, we have that
\[
(\cQ\otimes\cP)(X) = \left\{ B\in \Mat_4(k) \mid (\phi_i\otimes g_j)(B) = B \text{ for } i\in\{1,2\}, j\in\{1,\ldots,6\} \right\}.
\]
Taking $B = [b_{ij}]_{1\leq i,j\leq 4} \in \Mat_4(k)$ indexed as usual, we have that
\[
(\phi_1\otimes\Id)(B) = \begin{bmatrix} b_{33} & b_{34} & b_{31}x & b_{32}x \\ b_{43} & b_{44} & b_{41}x & b_{42}x \\ b_{13}x & b_{14}x & b_{11} & b_{12} \\ b_{23}x & b_{24}x & b_{21} & b_{22} \end{bmatrix}
\]
and so to be fixed, $B$ must be of the form
\[
B = \begin{bmatrix} a & b & 0 & 0 \\ c & d & 0 & 0 \\ 0 & 0 & a & b \\ 0 & 0 & c & d \end{bmatrix} = I_2 \otimes \begin{bmatrix} a & b \\ c & d \end{bmatrix}.
\]
To then be fixed by all $\phi_1\otimes g_j$, the matrix $\begin{bmatrix} a & b \\ c & d \end{bmatrix}$ must be fixed by $\PGL_2(\FF_2)$, but only the scalar matrices are fixed this way. Thus $B=aI_4$ for some $a\in k$ and so we obtain that
\[
(\cQ\otimes\cP)(X) = k,
\]
as claimed.
\end{proof}

\subsection{Main theorem} 
We continue using the same notation as in the previous section, i.e.\ $k$ is an algebraically closed characteristic 2 field, $X$ is an Igusa surface realized as a quotient of a product of two ordinary elliptic curves over $k$, and $\mathcal{Q}$ and $\mathcal{P}$ are the two particular quaternion Azumaya algebras on $X$ constructed above. We will need the following lemma:

\begin{lem}\label{lem: qpobs} Let $k,X,\mathcal{Q}$, and $\mathcal{P}$ be as above. Let $\tau_3=\mathrm{Spec}(k[x]/(x-1)^3)$ and $\mu_2=\mathrm{Spec}(k[x]/(x-1)^2)$ with canonical inclusion $\mu_2\subset \tau_3$. Then there exists a deformation $\mathcal{B}$ of either $\mathcal{Q}$ or $\mathcal{P}$ to $X_{\mu_2}$ for which $ob(\mathcal{B};X_{\tau_3}/X_{\mu_2})\neq 0$.

In other words, at least one of either $\mathcal{Q}$ or $\mathcal{P}$ has obstructed deformations.
\end{lem}

\begin{proof}
We write $(\mathcal{Q},\sigma_1,f_1)$ for the quadratic triple associated to $\mathcal{Q}$ by \Cref{global_sections_Q} and $(\mathcal{P},\sigma_2,f_2)$ for the quadratic triple associated to $\mathcal{P}$ by \Cref{rem: torsquat}. Since both $\mathcal{Q}$ and $\mathcal{P}$ are neutral, we have $\mathcal{Q}=\cEnd(\mathcal{V})$ and $\mathcal{P}=\cEnd(\mathcal{W})$ for two rank $2$ vector bundles $\mathcal{V}$ and $\mathcal{W}$ on $X$. Our proof works in two parts. First, we show that the obstructions to deforming either $\mathcal{Q}$ or $\mathcal{P}$ (or deformations of these) are equivalent with the obstructions to deforming $\mathcal{V}$ or $\mathcal{W}$ (or deformations of these) respectively. Second, we show that at least one of $\mathcal{V}$ or $\mathcal{W}$ has obstructed deformations by lifting the obstruction for the associated determinant bundles.

To see the first claim, note that by applying Serre duality, the exact sequence
\[
\rH^2(X,\cO_X) \to \mathrm{H}^2(X,\mathcal{Q})\rightarrow \mathrm{H}^2(X,\mathcal{Q}/\mathcal{O}_X) \to \rH^3(X,\cO_X)=0
\]
is dual to the exact sequence
\[
0 \to \left(\mathcal{Q}/\mathcal{O}_X\right)^\vee(X) \to \mathcal{Q}^\vee(X) \to \cO_X^\vee(X).
\]
where the map $\cQ^\vee(X) \to \cO_X^\vee(X)$ sends a functional to its evaluation on scalars. Due to the isomorphism $\cQ\cong \cQ^\vee$ given by the reduced trace, any global section in $\cQ^\vee(X)$ is of the form $\Trd_\cQ(x\cdot \und)$ for some $x\in \cQ(X)$. Now, since $\mathrm{deg}(\cQ)=2$, $\cQ(X)=k$, and $\mathrm{char}(k)=2$, the evaluation of any such global section on a scalar will simply be
\[
\Trd_\cQ(x\left(c \cdot 1_\cQ\right)) = 0
\]
for $x,c\in k$. Thus, the map $\mathcal{Q}^\vee(X) \to \cO_X^\vee(X)$ is identically zero, which going back through duality, means that $\rH^2(X,\cO_X) \to \mathrm{H}^2(X,\mathcal{Q})$ is the zero morphism and we have that
\[
\mathrm{H}^2(X,\mathcal{Q})\iso \mathrm{H}^2(X,\mathcal{Q}/\mathcal{O}_X)
\]
is an isomorphism. The same argument applies to show that the canonical map \[\mathrm{H}^2(X,\mathcal{P})\iso \mathrm{H}^2(X,\mathcal{P}/\mathcal{O}_X)\] is also an isomorphism. Since these maps can be considered as taking the obstructions to deforming the underlying vector bundle to the obstructions to deforming the endomorphism algebra of these bundles (see \Cref{lem: detob}), we see that it suffices to show one of $\mathcal{V}$ or $\mathcal{W}$ has obstructed deformations. The remainder of the proof is devoted to showing this fact.

By construction, both $\mathcal{Q}$ and $\mathcal{P}$ have nonvanishing strong obstruction. By either direct analysis or the use of \Cref{cor: quaternion}, we can therefore write $\mathcal{Q}=\mathcal{E}_0\oplus \mathcal{E}_1$ and $\mathcal{P}=\mathcal{F}_0\oplus \mathcal{F}_1$ for four rank $2$ vector bundles $\mathcal{E}_0,\mathcal{E}_1,\mathcal{F}_0,\mathcal{F}_1$ on $X$. Further, we can assume that $\mathcal{E}_0$ and $\mathcal{F}_0$ are given as nonsplit extensions \[0\rightarrow \mathcal{O}_X\rightarrow \mathcal{E}_0\rightarrow \mathcal{O}_X\rightarrow 0 \quad\mbox{and} \quad 0\rightarrow \mathcal{O}_X\rightarrow \mathcal{F}_0\rightarrow \mathcal{O}_X\rightarrow 0.\] We note that $\mathcal{E}_0\not\cong \mathcal{F}_0$ since, if there was an isomorphism $\mathcal{E}_0\cong \mathcal{F}_0$, then there would be an inclusion \[ \cEnd(\mathcal{E}_0)\cong \mathcal{E}_0\otimes \mathcal{F}_0^\vee\subset \mathcal{Q}\otimes \mathcal{P}^\vee \cong \mathcal{Q}\otimes \mathcal{P}.\] However, the identity endomorphism of $\mathcal{E}_0$ and the composition $\mathcal{E}_0\twoheadrightarrow \mathcal{O}_X\subset \mathcal{E}_0$ coming from the morphisms defining the extension of $\mathcal{E}_0$ are two independent global sections of $\cEnd(\mathcal{E}_0)$ and $(\mathcal{Q}\otimes \mathcal{P})(X)\cong k$ by \Cref{lem: globalsectionsofqp}.

Since we have $\mathcal{Q}(X)=k$ and $\mathcal{P}(X)=k$, we find that $\mathcal{E}_1(X)=\mathcal{F}_1(X)=0$ and $(\mathcal{Q}/\mathcal{O}_X)(X)=(\mathcal{P}/\mathcal{O}_X)(X)=k$. There is then a commutative diagram with exact rows and vertical inclusions \[\begin{tikzcd} 0\arrow{r} & \mathcal{O}_X\arrow{r} &\mathcal{Q}\arrow{r}&  \mathcal{Q}/\mathcal{O}_X\arrow{r} &0\\ 0\arrow{r} & \mathcal{O}_X\arrow[equals]{u}\arrow{r} & \mathcal{E}_0\arrow{u}\arrow{r} & \mathcal{O}_X\arrow{u}\arrow{r} & 0.\end{tikzcd}\] It follows that the first boundary square in cohomology \[\begin{tikzcd}\mathrm{H}^0(X,\mathcal{Q}/\mathcal{O}_X)\arrow{r} & \mathrm{H}^1(X,\mathcal{O}_X)\\
\mathcal{O}_X(X)\arrow[equals]{u}\arrow{r} & \mathrm{H}^1(X,\mathcal{O}_X)\arrow[equals]{u}\end{tikzcd}\] sends $1\in\mathrm{H}^0(X,\mathcal{Q}/\mathcal{O}_X)=k$ to the class $\alpha(\mathcal{E}_0)$ in $\mathrm{H}^1(X,\mathcal{O}_X)\cong \mathrm{Ext}^1(\mathcal{O}_X,\mathcal{O}_X)$ of the given extension of $\mathcal{E}_0$. Similarly, the connecting map \[k=\mathrm{H}^0(X,\mathcal{P}/\mathcal{O}_X)\rightarrow \mathrm{H}^1(X,\mathcal{O}_X)\] sends $1$ to the extension class $\alpha(\mathcal{F}_0)$ of $\mathcal{F}_0$. Note that since $\mathcal{E}_0\not\cong \mathcal{F}_0$, we must have that these classes are spanning, i.e.\ $\mathrm{Span}_k\{\alpha(\mathcal{E}_0),\alpha(\mathcal{F}_0)\}=\mathrm{H}^1(X,\mathcal{O}_X)$.

 Consider, now, the exact sequences \[0\rightarrow (\mathcal{Q}/\mathcal{O}_X)^\vee\rightarrow \mathcal{Q}\xrightarrow{\mathrm{trd}_{\cQ}} \mathcal{O}_X\rightarrow 0\quad \mbox{and}\quad 0\rightarrow (\mathcal{P}/\mathcal{O}_X)^\vee\rightarrow \mathcal{P}\xrightarrow{\mathrm{trd}_{\cP}} \mathcal{O}_X\rightarrow 0.\] These induce exact sequences in cohomology \[\mathrm{H}^1(X,\mathcal{Q})\rightarrow \mathrm{H}^1(X,\mathcal{O}_X)\rightarrow \mathrm{H}^2(X,(\mathcal{Q}/\mathcal{O}_X)^\vee)\rightarrow 0\] and also \[\mathrm{H}^1(X,\mathcal{P})\rightarrow \mathrm{H}^1(X,\mathcal{O}_X)\rightarrow \mathrm{H}^2(X,(\mathcal{P}/\mathcal{O}_X)^\vee)\rightarrow 0.\] By Serre duality, the rightmost nonzero arrows in these sequences are dual to the inclusions of the $k$-span of $\alpha(\mathcal{E}_0)$ and $\alpha(\mathcal{F}_0)$ respectively.

By \Cref{lem: detob}, the maps induced by the reduced traces \begin{equation}\label{eq: trdqph1} \mathrm{H}^1(X,\mathcal{Q})\rightarrow \mathrm{H}^1(X,\mathcal{O}_X)\quad \mbox{and}\quad \mathrm{H}^1(X,\mathcal{P})\rightarrow \mathrm{H}^1(X,\mathcal{O}_X)\end{equation} may be interpreted as taking a deformation, of $\mathcal{V}$ or $\mathcal{W}$ to $X_{\mu_2}$, to their associated determinant line bundle. Since the image of either map in \eqref{eq: trdqph1} is $1$-dimensional, and since these images lie in linearly independent directions, there is at least one deformation of either $\mathcal{V}$ or $\mathcal{W}$ that must have a determinant bundle with obstructed deformations to $X_{\tau_3}$, by \Cref{rem: explicitob}. Any such deformation of either $\mathcal{V}$ or $\mathcal{W}$ must then also have obstructed deformations to $X_{\tau_3}$. Setting $\mathcal{B}$ to be the endomorphism algebra associated to such a deformation finishes the proof.
\end{proof}

Now, combining all of our work so far, we can prove:

\begin{thm}\label{thm: main}
Let $(\sigma_0,f_0)$ be the canonical quadratic triple on the tensor product Azumaya $\mathcal{O}_X$-algebra $\mathcal{Q}\otimes \mathcal{P}$, cf.\ \Cref{norm_functor}. Let $\mu_2=\mathrm{Spec}(k[x]/(x-1)^2)$ and $\tau_3=\mathrm{Spec}(k[x]/(x-1)^3)$ with inclusion $i:\mu_2\subset \tau_3$. Then there exists a deformation 
$(\mathcal{A},\sigma,f)$ of $(\mathcal{Q}\otimes \mathcal{P},\sigma_0,f_0)$ to $X_{\mu_2}$ such that \[ob(\mathcal{A},\sigma,f; X_{\tau_3}/X_{\mu_2})\neq 0\quad \mbox{and}\quad ob(\mathcal{A};X_{\tau_3}/X_{\mu_2})=0.\]

In other words, the quadratic triple $(\mathcal{A},\sigma,f)$ on $X_{\mu_2}$ does not deform to $X_{\tau_3}$, but $\mathcal{A}$ does deform to $X_{\tau_3}$ when considered as an Azumaya algebra alone.
\end{thm}

\begin{proof}
Let $\pi_2:X_{\mu_2}\rightarrow X$ be the projection. Let $\mathcal{B}$ be a quaternion algebra with obstructed deformations to $X_{\tau_3}$ as shown to exist in \Cref{lem: qpobs}. If $\cB$ is a deformation of $\cQ$, set $\mathcal{A}=\mathcal{B}\otimes \pi_2^*\mathcal{P}$, and if $\cB$ is a deformation of $\cP$, set $\cA=\pi_2^*\mathcal{Q}\otimes \mathcal{B}$. Equip $\mathcal{A}$ with its canonical quadratic pair $(\sigma,f)$.
It is immediate that $(\mathcal{A},\sigma,f)$ is a deformation of $(\mathcal{Q}\otimes \mathcal{P},\sigma_0,f_0)$.

To see that $(\mathcal{A},\sigma,f)$ does not deform to $X_{\tau_3}$, we use \Cref{deformations_through_stack_morphism}. Indeed, since $\mathcal{B}$ does not deform, neither does $(\mathcal{A},\sigma,f)$. The obstruction $ob(\mathcal{A},\sigma,f;X_{\tau_3}/X_{\mu_2})$ must then be nonzero. However, due to both \Cref{cor: forg} and \Cref{thm: zero} (or \Cref{prop: quartvan} and \Cref{cor: quartvan} noting \Cref{rem: specialcase}), all such obstructions vanish when one forgets the quadratic pair, i.e.\ $ob(\mathcal{A};X_{\tau_3}/X_{\mu_2})=0$.\end{proof}

\bibliographystyle{amsalpha}
\bibliography{bib}
\end{document}